\documentclass[ps]{imsart}

\pdfoutput=1
 
\RequirePackage[OT1]{fontenc}
\RequirePackage{amsthm,amsmath}
\RequirePackage[numbers]{natbib}
\setcitestyle{numbers,square,comma}
\RequirePackage[colorlinks,citecolor=blue,urlcolor=blue]{hyperref}
\usepackage{graphicx}
\usepackage{multicol}

\pubyear{to appear}
\volume{0}
\issue{0}
\firstpage{1}
\lastpage{105}
\arxiv{2204.04690}

\startlocaldefs

\usepackage{amsmath}
\usepackage[psamsfonts]{amssymb}
\usepackage{amsthm}
\usepackage{amsfonts, latexsym, amsbsy}
\usepackage[utf8]{inputenc}

\usepackage{dsfont,eucal,mathrsfs}
\usepackage[german=guillemets]{csquotes}
\setquotestyle[british]{english}

\usepackage[UKenglish]{babel}
\usepackage{booktabs,hhline}

\usepackage{mathtools}
\usepackage{verbatim}
\usepackage{url}
\usepackage{thmtools}
\usepackage{enumerate}

\usepackage{imakeidx}
\indexsetup{firstpagestyle=imsheadings}

\makeindex

\renewcommand{\leq}{\leqslant}
\renewcommand{\geq}{\geqslant}

\newcommand{\RE}{\ensuremath{\operatorname{Re}}}
\newcommand{\IM}{\ensuremath{\operatorname{Im}}}
\newcommand{\Ent}{\operatorname{\mathrm{Ent}}}
\newcommand{\diam}{\operatorname{\mathrm{diam}}}
\newcommand{\sgn}{\operatorname{\mathrm{sgn}}}
\newcommand{\supp}{\operatorname{\mathrm{supp}}}
\newcommand{\entier}[1]{\lfloor#1\rfloor}
\newcommand{\tr}{\operatorname{\mathrm{tr}}}
\newcommand{\bracket}[1]{\langle#1\rangle}
\newcommand{\sbracket}[1]{[#1]}

\newcommand{\kolmo}{\nu_0}

\DeclareFontFamily{U}{mathx}{\hyphenchar\font45}
\DeclareFontShape{U}{mathx}{m}{n}{
      <5> <6> <7> <8> <9> <10>
      <10.95> <12> <14.4> <17.28> <20.74> <24.88>
      mathx10
      }{}
\DeclareSymbolFont{mathx}{U}{mathx}{m}{n}
\DeclareFontSubstitution{U}{mathx}{m}{n}
\DeclareMathAccent{\widecheck}{0}{mathx}{"71}
\DeclareMathAccent{\wideparen}{0}{mathx}{"75}

\newcommand{\ball}[2]{\mathrm{B}_{#1}(#2)}
\newcommand{\cball}[2]{\overline{\mathrm{B}}_{#1}(#2)}
\newcommand{\coball}[2]{\mathrm{B}^c_{#1}(#2)}
\newcommand{\fa}{\quad\text{for all\ \ }}
\newcommand{\et}{\quad\text{and}\quad}

\newcommand{\primo}{1\textsuperscript{o}}
\newcommand{\primolist}{\par\smallskip\noindent\textbf\primo\;\;}
\newcommand{\secundo}{2\textsuperscript{o}}
\newcommand{\secundolist}{\par\smallskip\noindent\textbf\secundo\;\;}
\newcommand{\tertio}{3\textsuperscript{o}}
\newcommand{\tertiolist}{\par\smallskip\noindent\textbf\tertio\;\;}
\newcommand{\quarto}{4\textsuperscript{o}}
\newcommand{\quartolist}{\par\smallskip\noindent\textbf\quarto\;\;}
\newcommand{\quinto}{5\textsuperscript{o}}
\newcommand{\quintolist}{\par\smallskip\noindent\textbf\quinto\;\;}
\newcommand{\sexto}{6\textsuperscript{o}}
\newcommand{\sextolist}{\par\smallskip\noindent\textbf\sexto\;\;}
\newcommand{\septimo}{7\textsuperscript{o}}
\newcommand{\septimolist}{\par\smallskip\noindent\textbf\septimo\;\;}

\newcommand{\Normal}{\mathsf N}
\newcommand{\Uniform}{\mathsf U}

\newcommand{\comp}{\mathds C}
\newcommand{\integer}{\mathds Z}
\newcommand{\nat}{\mathds N}

\newcommand{\rat}{\mathds Q}
\newcommand{\real}{\mathds R}
\newcommand{\rn}{{{\mathds R}^n}}
\newcommand{\rd}{{{\mathds R}^d}}

\newcommand{\Ee}{\mathds{E}}

\newcommand{\Pp}{\mathds{P}}

\newcommand{\Ss}{\mathds{S}}
\newcommand{\Tt}{\mathds{T}}
\newcommand{\Vv}{\mathds{V}}

\newcommand{\I}{\mathds 1}

\newcommand{\Ascr}{{\mathscr A}}
\newcommand{\Bscr}{{\mathscr B}}

\newcommand{\Fscr}{{\mathscr F}}
\newcommand{\Gscr}{{\mathscr G}}

\newcommand{\Nscr}{{\mathscr N}}

\newcommand{\Bcal}{{\mathcal B}}

\newcommand{\Tcal}{{\mathcal T}}

\newtheorem{theorem}{Theorem}[section]
\newtheorem*{theorem*}{Theorem}
\newtheorem{corollary}[theorem]{Corollary}
\newtheorem*{corollary*}{Corollary}
\newtheorem{lemma}[theorem]{Lemma}
\newtheorem*{lemma*}{Lemma}
\newtheorem{proposition}[theorem]{Proposition}
\newtheorem*{proposition*}{Proposition}
\newtheorem{properties}[theorem]{Properties}
\newtheorem*{properties*}{Proposition}
\newtheorem{fact}{Fact}[section]

\theoremstyle{definition}
\newtheorem{definition}[theorem]{Definition}
\newtheorem*{definition*}{Definition}
\newtheorem{remark}[theorem]{Remark}
\newtheorem*{remark*}{Remark}
\newtheorem{example}[theorem]{Example}
\newtheorem*{example*}{Example}
\newtheorem{scholium}[theorem]{Scholium}

\newenvironment{nindex}[1][50pt]%
    {\begin{list}{}%
        {%
            \setlength{\labelwidth}{#1}%
            \setlength{\leftmargin}{\labelwidth+\labelsep}%
            \setlength{\itemsep}{2.5pt}%
            \setlength{\parsep}{0pt}%
            \setlength{\rightmargin}{0pt}%
        }%
    }%
    {\end{list}}

\numberwithin{equation}{section}

\begin{document}
\begin{frontmatter}
\title{Maximal Inequalities and Some Applications}
\runtitle{Maximal Inequalities}

\begin{aug}
\author{\fnms{Franziska} \snm{K\"{u}hn}
\ead[label=e1]{franziska.fr.kuehn@deutschebahn.com}}
\address{DB Regio AG, Richard-Wagner-Stra{\ss}e 1,\\ 04109 Leipzig, Germany\\
\printead{e1}}

\and
\author{\fnms{Ren\'e L.} \snm{Schilling}
\ead[label=e2]{rene.schilling@tu-dresden.de}}
\address{Institut f\"{u}r Mathematische Stochastik, Fakult\"{a}t Mathematik\\
Technische Universit\"{a}t Dresden, 01062 Dresden, Germany\\
\printead{e2}}


\runauthor{F.\ K\"{u}hn \& R.L.\ Schilling}

\end{aug}

\begin{abstract}
    A maximal inequality is an inequality which involves the (absolute) supremum $\sup_{s\leq t}|X_{s}|$ or the running maximum $\sup_{s\leq t}X_{s}$ of a stochastic process $(X_t)_{t\geq 0}$. We discuss maximal inequalities for several classes of stochastic processes with values in an Euclidean space: Martingales, L\'evy processes, L\'evy-type -- including Feller processes, (compound) pseudo Poisson processes, stable-like processes and solutions to SDEs driven by a L\'evy process --, strong Markov processes and Gaussian processes. Using the Burkholder--Davis--Gundy inequalities we also discuss some relations between maximal estimates in probability and the Hardy--Littlewood maximal functions from analysis.
\end{abstract}



\index{truncated moment|see{censored m.}}
\index{moment|see{censored/generalized m.}}

\tableofcontents

\end{frontmatter}

\section{Introduction}\label{intro}
A \emph{maximal inequality} is, quite generally, an inequality that involves the running maximum $M_t\coloneqq \sup_{s\leq t}X_s$ or the absolute supremum $X^*_t \coloneqq  \sup_{s\leq t}|X_s|$ (and $M\coloneqq M_\infty$, $X^*\coloneqq X^*_\infty$) of a stochastic process $(X_t)_{t\geq 0}$. Frequently, this happens in the form of a \emph{tail estimate}, i.e.\ estimates for $\Pp(X_t^* > \lambda)$, $\Pp(M_t \leq \lambda)$ etc.\ or as a \emph{generalized moment}, \index{generalized moment} i.e.\ bounds of $\Ee\left[g(X^*_t)\right]$ or $\Ee\left[\sup_{s\leq t} g(X_s)\right]$ or in similar expressions.

There are many situations where maximal inequalities appear naturally, for example when studying the growth and asymptotics of $t\mapsto X_t$ as $t\to 0$ or $t\to \infty$, or the Hausdorff dimension of the range. One reason for the usefulness of maximal inequalities are inclusions of the type
\begin{gather*}
     \left\{X_t^* < r\right\}\subseteq
     \left\{\sigma_r > t\right\} \subseteq
     \left\{X_t^* \leq r\right\} \subseteq
     \left\{\sigma_r \geq t\right\},
\end{gather*}
which should be seen as a kind of duality between the maximum process $X_t^*$ and the first exit time $\sigma_r \coloneqq  \inf\left\{s > 0 \mid |X_s|>r\right\}$  from the closed ball $\cball{r}{0}$ with centre $0$ and radius $r>0$; we assume that $t\mapsto X_t$ is c\`adl\`ag, i.e.\ right continuous with finite left hand limits. Each of these inclusions can be strict due to the presence of jumps.

Surprisingly, maximal tail inequalities are often easier to obtain than ordinary tail inequalities, or even heat kernel bounds. In many situations maximal tail inequalities may serve as a good substitute of a heat kernel estimate. In Remark~\ref{fel-25} we give a (by no means complete) list of potential applications and some relevant references. Having said this, we would like to point out that for many processes maximal (tail) estimates are equivalent to ordinary tail estimates: For Brownian motion this is just the reflection principle, for L\'evy processes we can use Etemadi's inequality, see Theorem~\ref{lp-39}.

In this survey paper we discuss various forms of maximal inequalities for several classes of stochastic processes: Martingales, L\'evy processes, L\'evy-type \& Feller processes, Markov processes and Gaussian processes -- and in Chapters~\ref{mar} through \ref{gau} each class is dealt with in a chapter on its own. We try to be as self-contained as sensibly possible, and each chapter starts with a brief introduction to the class of processes under investigation. On the one hand, this makes the chapters independent of each other and allows the novice and non-specialist to follow the presentation easily, on the other hand it allows us to fix the notation and to go in our presentation from the more special to the more general. This may not always be the most efficient approach, but it enables us to exploit the peculiarities of each class and to get clearer proofs.  Chapters~\ref{pre} and~\ref{gen} are some kind of exception, as we collect in Chapter~\ref{pre} the basic tools of the trade, which are frequently used but rarely systematically exposed. Chapter~\ref{gen} explores some basic principles for rather general processes without assuming too much structure. We hope that this is a useful resource for researchers and offers, even for the specialist, some new vistas on known material.

Most results and proofs are (more or less well) known. Wherever possible, we give precise references on the source(s) and point out further reading. Frequently, the arguments are streamlined, sometimes with minor improvements, but it would be utterly presumptive to claim any originality but in the composition.

A few words on \emph{what we did not \textup{(}even try to\textup{)} do} are in order: We do not cover infinite-dimensional state spaces and quite often we restrict ourselves to dimension $d=1$, since the multivariate version is pretty obvious and leads only to notational inconveniences. Naturally, areas where we contributed ourselves to the research, e.g.\ L\'evy and L\'evy-type processes, are covered in greater depth than fields like Gaussian processes, where we give only a glimpse into some of the methods -- this is partly also due to the fact that Gaussian processes require very different methods than the other classes of stochastic processes discussed here. Although it can be seen as a maximal tail estimate, the theory of large deviations is also excluded: this is too wide a field, and it is pretty difficult to surpass an exposition like Dembo \& Zeitouni~\cite{dem-zei92}. Finally, we do not discuss discrete time settings and refer instead to Lin \& Bai~\cite{lin-bai10}, Cs\"{o}rg\"{o} \& R\'ev\'esz~\cite{cso-rev81} or Petrov~\cite{petrov75,petrov95}.

\paragraph{Notation.} Overall, we try to stick to standard notation, and more specialized symbols are introduced locally and when needed. Here we would like to highlight some conventions which are used throughout this paper: If $(X_t)_{t\geq 0}$ is a stochastic process, we write
\begin{gather*}
    X_t^* \coloneqq  \sup_{s\leq t}|X_s|
    \quad\text{and}\quad
    X^* \coloneqq  \sup_{s\geq 0}|X_s|.
\end{gather*}
If $X$ and $Y$ are random variables, we denote by $\Ee X$ and $\Vv X$ the expectation and the variance of $X$; by $X\sim Y$ we indicate that $X$ and $Y$ have the same probability distribution.

Throughout, we assume the \emph{usual conditions} \index{usual conditions} on the filtered probability space $(\Omega,\Ascr,\Pp,\Fscr_t)$: The underlying probability space $(\Omega,\Ascr,\Pp)$ is complete (i.e.\ every subset of a $\Pp$-null set is measurable) and the filtration $(\Fscr_t)_{t\geq 0}$ is right-continuous (i.e.\ $\Fscr_t = \Fscr_{t+} = \bigcap_{u>t}\Fscr_u$) and $\Fscr_0$ contains all $\Pp$-null sets. This convention ensures, in particular, that the notions of optional time and stopping time coincide, and that first entrance times and first hitting times of Borel sets are stopping times.

The Fourier transform \index{Fourier transform} and its inverse are denoted by
\begin{gather*}\label{hidden-ft}
    \widehat f(\xi) \coloneqq (2\pi)^{-d}\int_{\rd} e^{-i\xi\cdot y} f(y)\,dy
    \quad\text{and}\quad
    \widecheck g(y) \coloneqq \int_{\rd} e^{i \eta\cdot y} g(\eta)\,d\eta,
\end{gather*}
respectively; this convention allows us to identify the inverse Fourier transform with the characteristic function of a random variable, and to stay compatible with most formulae used in harmonic analysis and the theory of pseudo differential operators.

We use \enquote{positive}, \enquote{negative}, \enquote{monotone}, \enquote{increasing} etc.\ always in the non-strict sense; the natural numbers $\nat$ are $\{1,2,3,\dots\}$ and $\nat_0 = \nat\cup\{0\}$. As usual, $a\wedge b$ and $a\vee b$ denote the minimum and the maximum of $a,b\in\real$. With $f\asymp g$ or $f(t)\asymp g(t)$ we indicate that there are constants $0<c\leq C<\infty$ such that $c f(t)\leq g(t)\leq C f(t)$ (for a given range of $t$'s). The positive and negative part of a function is $f^+ \coloneqq  f\vee 0$ and $f^- \coloneqq  -(f\wedge 0)$. By $\entier{x}$ we denote the largest integer $n\in\integer$ such that $n\leq x$. We write $\ball{r}{x}$ for the open ball with centre $x$ and radius $r>0$. Finally, we agree that $\inf\emptyset \coloneqq  \infty$.

 An \textbf{index of notation} is included on page \pageref{index-notation} at the end of the paper.

\section{Preliminaries}\label{pre}

\subsection{Censored Moments of Infinitely Divisible Random Variables}\label{pre-mom}

We restrict ourselves to the one-dimensional setting, since the extension to higher dimensions is often straightforward.

\paragraph{Infinite divisibility and characteristic exponents.}
Recall that a random variable $X$ is \emph{infinitely divisible}, \index{infinitely divisible} if $X\sim X_1+\dots+X_n$ for any $n\in\nat$ where $X_1,\dots,X_n$ are iid copies of some random variable $X^{(n)}$. It is easy to see that $X$ is infinitely divisible, if the characteristic function $\phi_X(\xi) = \Ee e^{i\xi X}$ satisfies $\phi_X(\xi) = \phi_{(n)}(\xi)^n$ where $\phi_{(n)}(\xi)$ is itself a characteristic function. Since the characteristic function determines the law of a random variable, we have $\phi_{(n)} = \phi_{X^{(n)}}$ for the random variable $X^{(n)}$ from above. A classical result of L\'evy and Khintchine shows that $X$ is infinitely divisible if, and only if, $\phi_X(\xi) = \exp\left(-\psi_X(\xi)\right)$ with the \emph{characteristic exponent} \index{characteristic exponent} given by the \emph{L\'evy--Khintchine formula} \index{Levy--Khintchine formula@L\'evy--Khintchine formula}
\begin{gather}\label{pre-e71}
        \psi_X(\xi)
        = -ib\xi + \frac 12 \sigma^2\xi^2 + \int_{y\neq 0}\left(1 - e^{iy\xi} + iy\xi\I_{(0,1)}(|y|)\right)\nu(dy).
\end{gather}
The \emph{L\'evy triplet} \index{Levy triplet@L\'evy triplet} $(b,\sigma^2,\nu)$ -- comprising $b\in\real$, $\sigma\geq 0$ and a Borel measure $\nu$ on $\real\setminus\{0\}$ such that $\int_{0<|y|<1} y^2\,\nu(dy) + \int_{|y|\geq 1}\nu(dy) < \infty$ -- determines $\psi_X$, hence $\phi_X$ and $X$, uniquely. Note that $2\RE\psi_X(\xi)$ is the characteristic exponent of the symmetrization of $X$, i.e.\ $X-X'$ where $X'$ is an iid copy of $X$. We refer to Sato~\cite{sato99} as standard reference, also for the connection of infinitely divisible random variables and L\'evy processes.

A further characterization of the exponent $\psi_X$ is to say that it is a \emph{continuous negative definite function} \index{negative definite function} with $\psi_X(0)=0$; negative definite means that the matrix $\left(\psi(\xi_i)+\overline{\psi(\xi_j)}-\psi(\xi_i-\xi_j)\right)_{i,j}$ is for any $n\in\nat$ and $\xi_1,\dots,\xi_n\in\real$ positive hermitian.

We remark that, in particular, $\RE\psi_X(\xi)\geq 0$ and $\xi\mapsto\sqrt{|\psi_X(\xi)|}$ is subadditive, i.e.\ $\sqrt{|\psi_X(\xi+\eta)|} \leq \sqrt{|\psi_X(\xi)|}+\sqrt{|\psi_X(\eta)|}$ for all $\xi,\eta\in\real$. Standard references on (continuous) negative definite functions are Berg \& Forst~\cite{ber-for75} and Jacob~\cite[Vol.~1]{jacob1-3}.

We will also need the following \emph{maximum functions} \index{maximum function $\psi^*, \vert\psi\vert^*$} which describe the maximum of $|\psi_X(\xi)|$ and $\RE\psi_X(\xi)$ on balls with radius $\ell>0$:
\begin{gather}\label{pre-e73}
    \psi_X^*(\ell) \coloneqq  \sup_{|\xi|\leq\ell} \RE\psi_X(\xi)
    \et
    |\psi_X|^*(\ell) \coloneqq  \sup_{|\xi|\leq\ell} \left|\psi_X(\xi)\right|.
\end{gather}
These maximum functions are also important for estimates of the transition densities of L\'evy and L\'evy-type processes, see e.g.\ Knopova \& Kulik~\cite{kno-kul13,kno-kul17}, Knopova~\cite{knopova14}, Grzywny~\cite{grz14} or Grzywny \& Szczypkowski~\cite{grz-scz20}.

\paragraph{Censored second moments.}
Thinking of variances of random variables, it is often useful to require finite second moments -- but this is not always the case. The following construction will serve as a partial substitute when there are no finite second moments.  If $\mu$ is a measure on $\real$, we define its \emph{truncated second moment}, the \emph{tail} and the \emph{censored second moment}, respectively, as \index{censored moment}
\begin{gather}\label{pre-e75}
\begin{gathered}
    K(\mu;\ell) \coloneqq  \frac 1{\ell^2}\int_{0\leq |y|\leq\ell} y^2\,\mu(dy),
    \quad
    G(\mu;\ell) \coloneqq  \int_{|y|>\ell}\mu(dy),\\
    \text{and}\quad
    D(\mu;\ell) \coloneqq  K(\mu;\ell) + G(\mu;\ell),\quad \ell>0.
\end{gathered}
\end{gather}
It is not hard to see that $D(\mu;\ell) \leq D(\mu;\ell_0)$ for all $0<\ell_0 < \ell$.

For an infinitely divisible random variable with L\'evy triplet $(b,\sigma^2,\nu)$ we (formally) insert $\kolmo(dy) = \sigma^2 y^{-2} \delta_0(dy) + \nu(dy)$ instead of $\mu$, which means that
\begin{gather}\label{pre-e76}
\begin{aligned}
    D(\kolmo;\ell)
    &= \frac{\sigma^2}{\ell^2} + \frac 1{\ell^2}\int_{0<|y|\leq\ell} y^2\,\nu(dy) + \int_{|y|>\ell}\nu(dy)\\
    &= \frac{\sigma^2}{\ell^2} + \int_{y\neq 0} \min\left\{1,\;\frac{y^2}{\ell^2}\right\} \nu(dy).
\end{aligned}
\end{gather}

\begin{scholium}[probabilistic meaning of truncated moments]\label{pre-70}
 The interpretation of the truncated and censored moments in terms of moments of the random variable $X$ is obvious if $\mu$ is the corresponding probability distribution. If $X$ is an infinitely divisible random variable with exponent $\psi_X$ and L\'evy triplet $(b,\sigma^2,\nu)$, the truncated moments are moments of the characteristic triplet, and the probabilistic interpretation in terms of moments of $X$ is less straightforward.

    In order to appreciate the construction fully, it is useful to introduce yet another truncated moment that captures the non-symmetry of $X$, as suggested by Pruitt~\cite{pruitt81}:
    \begin{gather}\label{pre-e85}\begin{aligned}
        M(b,\nu;\ell)
        &= \frac 1\ell \left|b + \int_{|y|\leq\ell} y\,\I_{[1,\infty)}(|y|)\,\nu(dy) - \int_{|y|>\ell} y\,\I_{(0,1)}(|y|)\,\nu(dy) \right|\\
        &= \frac 1\ell \left|b + \bigg\{\int_{1}^\ell\!\! +\!\! \int_{-\ell}^{-1}\bigg\}\, y\,\nu(dy)\right|;
    \end{aligned}\end{gather}
    here we use the \enquote{Riemannian} convention that $\int_a^b = - \int_b^a$ if $a>b$.

    For fixed $\ell>0$ we split the random variable $X$ into a sum of two infinitely divisible random variables $Y+J$ where $Y=Y(\ell)$ and $J=J(\ell)$ are determined through their exponents:
    \begin{align*}
        \psi_{J}(\xi) &= \int_{|y|>\ell} \left(1-e^{iy\xi}\right) \nu(dy)\\
        \psi_{Y}(\xi) &= -i\xi\left(b - \smash[b]{\int_{|y|>\ell}} y\I_{(0,1)}(|y|)\,\nu(dy)\right) + \frac 12\sigma^2\xi^2\\
         &\qquad \mbox{} + \int_{0<|y|\leq\ell} \left(1-e^{iy\xi} + iy\xi\I_{(0,1)}(|y|)\right) \nu(dy).
    \end{align*}
    We can work out the moments of $X$, $Y$ and $J$ by differentiating the respective characteristic functions, e.g.
    \begin{gather*}
        \Ee\left[Y\right] = i\psi_Y'(0),\quad
        \Ee\left[Y^2\right] = \psi''_Y(0) - \psi'_Y(0)^2
        \et
        \Vv\left[Y\right] = \psi''_Y(0).
    \end{gather*}
    A short calculation reveals that $M(b,\nu;\ell)$ and $K(\kolmo;\ell)$ are the moments of the \enquote{truncated} random variable $Y=Y(\ell)$:
    \begin{gather*}
        \left|\Ee\left[\ell^{-1}Y(\ell)\right]\right| = M(b,\nu;\ell)
        \et
        \Vv\left[\ell^{-1}Y(\ell)\right] = K(\kolmo;\ell).
    \end{gather*}
\end{scholium}

We will now connect $D(\kolmo;\cdot)$ with the maximum exponent $\psi^*$. To do so, we need some preparations. A direct calculation, see e.g.\ Schilling \& Schnurr~\cite[Lem.~6.1]{sch-schnurr10} yields:
\begin{lemma}\label{pre-72}
    The function $x^2/(1+x^2)$, $x\in\real$, is a characteristic exponent and its L\'evy--Khintchine formula has the form
    \begin{gather}\label{pre-e77}
        \frac{x^2}{1+x^2}
        = \int_{\xi\neq 0} \left(1-\cos(x\xi)\right) g(\xi)\,d\xi
    \end{gather}
    where $g(\xi) = \frac 12\int_0^\infty (2\pi\lambda)^{-1/2} e^{-\xi^2/(2\lambda)} e^{-\lambda/2}\,d\lambda$ is an integrable function which has absolute moments of any order.
\end{lemma}

The next lemma contains a (by now) standard estimate. The short proofs are adapted from Schilling \& Schnurr~\cite[Lem.~6.1]{sch-schnurr10} and Grzywny~\cite[Lem.~4]{grz14}.
\begin{lemma}\label{pre-74} \index{negative definite function!upper bound}
    Let $\psi\colon \real\to\comp$ be a continuous negative definite function. Then
    \begin{gather}\label{pre-e79}
        |\psi(\xi)| \leq 2\sup_{|\eta|\leq 1}|\psi(\eta)| \left(1+|\xi|^2\right),\quad\xi\in\real.
    \end{gather}
    If $\Psi^*$ denotes one of the maximum functions $\psi^*$ or $|\psi|^*$ from \eqref{pre-e73}, then
    \index{maximum function $\psi^*, \vert\psi\vert^*$!estimates}
    \begin{gather}\label{pre-e81}
        \frac 12 \frac{r^2}{1+r^2} \,\Psi^*(\ell) \leq \Psi^*(r\ell) \leq 2(1+r^2)\Psi^*(\ell),\quad r,\ell > 0.
    \end{gather}
\end{lemma}
\begin{proof}
    Fix $\xi$ and write it in the form $\xi = \eta + n\xi_0$ where $\xi, \eta, \xi_0$ have the same sign, $|\eta|<1$, $|\xi_0|=1$ and $n\in\nat$. Using the subadditivity of $\sqrt{|\psi(\xi)|}$ we get
    \begin{gather*}
        \sqrt{|\psi(\xi)|}
        \leq \sqrt{|\psi(\eta)|} + n\sqrt{|\psi(\xi_0)|}
        \leq (1+|\xi|) \sup_{|\eta|\leq 1}\sqrt{|\psi(\eta)|},
    \end{gather*}
    and the first part of the lemma follows.

    We will show the second part only for $\psi^*(\ell)$. Since $\RE\psi(\xi)$ is also continuous and negative definite, the upper estimate follows from the first part applied to $\xi\mapsto\RE\psi(\ell \xi)$:
    \begin{gather*}
        \psi^*(r\ell)
        = \sup_{|\xi|\leq r}\RE\psi(\ell\xi)
        \leq 2\sup_{|\eta|\leq 1}\RE\psi(\ell\eta) \left(1+r^2\right)
        = 2\left(1+r^2\right)\psi^*(\ell),\quad r,\ell>0.
    \end{gather*}
    From this we get at once the lower bound, taking first $1/r$ instead of $r$, and then changing $\ell\rightsquigarrow \ell r$ in the resulting estimate.
\end{proof}

The following estimate from Schilling~\cite[Rem.~4.8]{rs-growth} is an improvement of results by Le Cam~\cite[Prop.~5]{lecam65} and Esseen
~\cite[Lem.~2.1]{esseen68}. We follow again the presentation of Grzywny~\cite[Lem.~4]{grz14}.
\begin{theorem}\label{pre-76}     \index{maximum function $\psi^*, \vert\psi\vert^*$!vs.\ censored moments}\index{censored moment}
    Let $\psi = \psi_X$ be the characteristic exponent of an infinitely divisible real random variable $X$ and denote by $D(\kolmo;\ell)$ the censored second moment of the measure $\kolmo \coloneqq  \sigma^2\delta_0 + \nu$; here $\sigma^2$ and $\nu$ are from the L\'evy triplet of $\psi$. Then the following estimates are satisfied with an absolute constant $c<1$:\footnote{This result remains valid if $X$ takes values in $\real^d$; then we can take $c = 1/(8(1+2d))$, cf.\ Grzywny~\cite[p.~9]{grz14}.}
    \begin{gather}\label{pre-e83}
        c D(\kolmo;\ell) \leq \psi^*(\ell^{-1}) \leq 2 D(\kolmo;\ell).
    \end{gather}
\end{theorem}
\begin{proof}
    Since the supremum is subadditive, we have
    \begin{gather*}
        \psi^*(\ell^{-1})
        \leq \sup_{|\xi|\leq 1/\ell} \frac 12\sigma^2\xi^2 + \sup_{|\xi|\leq 1/\ell} \int_{y\neq 0} \left(1-\cos(y\xi)\right)\nu(dy)
        \leq 2\psi^*(\ell^{-1}).
    \end{gather*}
    Therefore, it is enough to prove that for some constant $c$
    \begin{align*}
        c \int_{y\neq 0} \min\left\{1,\; \frac{y^2}{\ell^2}\right\}\nu(dy)
        &\leq\sup_{|\xi|\leq 1/\ell}\int_{y\neq 0} \left(1-\cos(y\xi)\right)\nu(dy)\\
        &\leq 2 \int_{y\neq 0} \min\left\{1,\; \frac{y^2}{\ell^2}\right\}\nu(dy).
    \end{align*}
    Thus, without loss of generality, we may assume $\sigma=0$.
    Combining the elementary estimate $\min\{1,y^2\} \leq 2y^2/(1+y^2)$ with Lemma~\ref{pre-72} and Tonelli's theorem shows
    \begin{align*}
        \int_{y\neq 0} \min\left\{1,\; \frac{y^2}{\ell^{2}}\right\}\nu(dy)
        &\leq 2\int_{y\neq 0} \frac{{y^2}{\ell^{-2}}}{1+{y^2}{\ell^{-2}}}\,\nu(dy)\\
        &\leq 2\int\int_{y\neq 0} \left(1-\cos\frac{\xi y}{\ell}\right) \nu(dy)\,g(\xi)\,d\xi\\
        &= 2\int \RE\psi(\xi/\ell) \,g(\xi)\,d\xi.
    \end{align*}
    Since $\RE\psi$ is continuous and negative definite, we can use Lemma~\ref{pre-74} to get
    \begin{align*}
        \int_{y\neq 0} \min\left\{1,\; \frac{y^2}{\ell^2}\right\}\nu(dy)
        &\leq 4\sup_{|\eta|\leq 1/\ell} \RE\psi(\eta) \int \left(1+\xi^2\right) g(\xi)\,d\xi
        = c \psi^*(\ell^{-1}).
    \end{align*}
    Using the estimate $1-\cos u = 2\sin^2\left(\frac 12 u\right) \leq 2\min\{1,u^2\}$ we see for all $|\xi|\leq 1/\ell$
    \begin{gather*}
        \RE\psi(\xi)
        \leq \sup_{|\xi|\leq 1/\ell} 2\int_{y\neq 0} \min\left\{1,y^2\xi^2\right\} \nu(dy)
        = 2\int_{y\neq 0} \min\left\{1, \frac{y^2}{\ell^2}\right\} \nu(dy)
    \end{gather*}
    finishing the proof.
\end{proof}

There is a similar estimate for $|\psi|^*(\ell^{-1})$ that takes into account the asymmetry of the random variable $X$.
\begin{theorem}\label{pre-78} \index{maximum function $\psi^*, \vert\psi\vert^*$!vs.\ censored moments}
    Let $\psi = \psi_X$ be the characteristic exponent of an infinitely divisible real random variable $X$ with L\'evy triplet $(b,\sigma^2,\nu)$ and denote by $D(\kolmo;\ell) = K(\kolmo;\ell)+G(\nu;\ell)$ and $M(b,\nu;\ell)$ the censored and truncated moments of the measure $\kolmo \coloneqq  \sigma^2\delta_0 + \nu$. Then the following estimates are satisfied with absolute constants:\footnote{In the $d$-dimensional analogue of this result, the constants depend only on the space dimension $d$.}
    \begin{gather}\label{pre-e87}
        |\psi|^*(\ell^{-1})
        \asymp
        K(\kolmo;\ell)+G(\nu;\ell)+M(b,\nu;\ell).
    \end{gather}
\end{theorem}
\begin{proof}
    Since we know already the result for the real part from Theorem~\ref{pre-76}, it is enough to estimate the imaginary part. Recall that
    \begin{gather*}
        -\IM\psi(\xi)
        = b\xi + \int_{y\neq 0} \left(\sin(y\xi) - y\xi\I_{(0,1)}(|y|)\right) \nu(dy).
    \end{gather*}
    We show the upper bound for $\ell < 1$, the case $\ell \geq 1$ differs only by a sign within the definition of $M(b,\nu;\ell)$.
    Fix $0<\ell< 1$ and take any $|\xi|< 1/\ell$.

    Using the elementary estimate $|\sin(x)-x|\leq \frac 16 |x|^3$, we get
    \begin{align*}
        &\left|\IM\psi(\xi)\right|\\
        &\quad= \bigg|\int\limits_{0<|y|\leq \ell} \left(\sin(y\xi) - y\xi\right) \nu(dy)
        + b\xi - \!\!\! \int\limits_{\ell < |y| < 1} \!\!\! y\xi\, \nu(dy)
        + \! \int\limits_{|y|>\ell} \! \sin(y\xi)\,\nu(dy)\bigg|\\
        &\quad\leq \frac 16 K(\kolmo;\ell) + M(b,\nu;\ell) + G(\nu;\ell).
    \end{align*}
    Since $|\xi|\leq 1/\ell$ is arbitrary, the upper bound follows if we use the upper bound in \eqref{pre-e83}.

    The lower bound can be shown in a similar way, using the lower triangle inequality and the observation that
    \begin{gather*}
        \left|\sin x - x\right|
        = \left|\int_0^x \left(1-\cos t\right)\,dt\right|
        \leq |x|\left(1-\cos x\right)
        \leq \left(1-\cos x\right),\quad |x|\leq 1.
    \end{gather*}
    Thus,
    \begin{align*}
        &\left|\IM\psi(\xi)\right|\\
        &\quad= \bigg|\int\limits_{0<|y|\leq \ell} \left(\sin(y\xi) - y\xi\right) \nu(dy)
        + b\xi - \!\!\! \int\limits_{\ell < |y| < 1} \!\!\! y\xi\, \nu(dy)
        + \!\int\limits_{|y|>\ell}\! \sin(y\xi)\,\nu(dy)\bigg|\\
        &\quad\geq M(b,\nu;\ell) - G(\nu;\ell) - \int\limits_{0<|y|\leq\ell} \left(1-\cos(y\xi)\right)\,\nu(dy)\\
        &\quad\geq M(b,\nu;\ell) - G(\nu;\ell) - \psi^*(\ell^{-1}).
    \end{align*}
    So we get $M(b,\nu;\ell)\leq 2\left(|\psi|^*(\ell^{-1}) + G(\nu;\ell)\right)$, and the lower estimate follows.
\end{proof}

\paragraph{Comparability of truncated and censored second moments.}\index{censored moment}
We want to compare the truncated moments $K(\mu;\ell)$ and the censored moments $D(\mu;\ell)$. Recall that
\begin{gather*}
    K(\mu;\ell) = \int_{|y|\leq\ell}  \frac{y^2}{\ell^2}\,\mu(dy)
    \et
    D(\mu;\ell) = \int_\real \min\left\{1,\; \frac{y^2}{\ell^2}\right\} \mu(dy).
\end{gather*}
Clearly, $K(\mu;\ell)\leq D(\mu;\ell)$. Let us introduce the condition
\begin{gather}\label{pre-e91}
    D(\mu;\ell) \leq \beta K(\mu;\ell)\quad
    \text{for some $\beta\geq 1$ and all $\ell>0$}.
\end{gather}
We have $\frac d{dx} \min\{1,x\} = \I_{[0,1]}(x)$ Lebesgue a.e.\ on $[0,\infty)$. Therefore,
\begin{gather*}
    \frac d{d\ell} D(\mu;\ell)
    = - \int_\real \I_{[0,1]}\left(\tfrac{y^2}{\ell^2}\right) \frac{2y^2}{\ell^3}\, \mu(dy)
    = - \frac 2{\ell} K(\mu;\ell)\quad\text{Leb.~a.e.\ for $\ell>0$}.
\end{gather*}
If we assume \eqref{pre-e91}, we get for almost all $\ell>0$
\begin{gather*}
    -\frac 2\ell \leq \frac{D'(\mu;\ell)}{D(\mu;\ell)} \leq - \frac 2{\beta\ell},
\end{gather*}
and, by integration,
\begin{gather*}
    \frac{1}{\ell^2} \leq \frac{D(\mu,\ell)}{D(\mu,1)} \leq \frac{1}{\ell^{2/\beta}} \quad\text{if $\ell\geq 1$, and}
    \quad
    \frac{1}{\ell^{2/\beta}} \leq \frac{D(\mu,\ell)}{D(\mu,1)} \leq \frac{1}{\ell^{2}} \quad\text{if $\ell\leq 1$}.
\end{gather*}
If we combine this with Theorem~\ref{pre-76} and use $\mu=\kolmo\coloneqq \sigma^2\delta_0+\nu$,  we finally arrive at the following result which was first observed by Knopova \& Kulik~\cite[Lem.~3.1]{kno-kul13}.
\begin{theorem}\label{pre-90} \index{maximum function $\psi^*, \vert\psi\vert^*$!estimates}
    Let $X$ be an infinitely divisible random variable with characteristic exponent $\psi$, maximum function $\psi^*$ and L\'evy triplet $(b,\sigma^2,\nu)$.

    If $\int_{y\neq 0} \min\left\{\ell^2,\:y^2\right\} \nu(dy)\leq \beta \int_{|y|\leq\ell} y^2\,\nu(dy)$ for some $\beta\geq 1$, then
    \begin{gather}
        \psi^*(r) \geq c r^{2/\beta}\quad\text{if $r\geq 1$, and}\quad
        \psi^*(r) \leq C r^{2/\beta}\quad\text{if $r\leq 1$}.
    \end{gather}
\end{theorem}

\paragraph{The sector condition.} Let us briefly discuss under which circumstances one can use $\psi^*$ and $|\psi|^*$ interchangeably. Clearly,
\begin{gather*}
    \RE\psi(\xi) \leq |\psi(\xi)|,\quad \xi\in\real
    \et
    \psi^*(\ell) \leq |\psi|^*(\ell),\quad \ell>0
\end{gather*}
are always satisfied. For the converse estimates one needs the so-called \emph{sector condition}.
\begin{lemma}\label{pre-92} \index{sector condition}
    Let $\psi\colon \real\to\comp$ be a continuous negative definite function. If the \emph{sector condition} with sector constant $\kappa\in (0,\infty)$ holds, i.e.
    \begin{gather}\label{pre-e99}
        \left|\IM\psi(\xi)\right| \leq \kappa \RE\psi(\xi),\quad\xi\in\real,
    \end{gather}
    then $\RE\psi$ and $|\psi|$ as well as $\psi^*$ and $|\psi|^*$ are comparable. Conversely, the comparability of $\RE\psi$ and $|\psi|$ entails the sector condition for some suitable $\kappa>0$.
\end{lemma}
\begin{proof}
The assertion follows immediately from the estimates
\begin{gather*}
 \RE\psi(\xi)
    \leq
    |\psi(\xi)|
    \leq
    (1+\kappa)\RE\psi(\xi).
\end{gather*}
which remain true if we take the supremum over all $|\xi|\leq\ell$. The converse follows then from $|\IM\psi(\xi)|\leq |\psi(\xi)|$ and the comparability of $\RE\psi(\xi)$ and $|\psi(\xi)|$.
\end{proof}

The sector condition has a few further interesting consequences. First of all, see Berg \& Forst~\cite{ber-for73}, it is equivalent to saying that the bilinear form $(u,v)\mapsto \int \psi(\xi) \widehat u(\xi) \overline{\widehat v(\xi)}\,d\xi$, $u,v\in L^2(\real;dx)$, is a Dirichlet form -- the Dirichlet form that uniquely characterizes the L\'evy process associated with $\psi$, see Jacob~\cite[Chapter 4.7. 4.8]{jacob1-3}. Stochastically, it is a condition that ensures that the \enquote{true random character} is not dominated by the \enquote{drift} which is essentially governed by the imaginary part $\IM\psi(\xi)$.

\subsection{Concentration Functions}\label{pre-con}

\paragraph{L\'evy's concentration function.}
Concentration functions were introduced by L\'evy~\cite[Chapter III.16]{levy37} in order to study the convergence behaviour of sums and sequences of random variables.
\begin{definition}[L\'evy]\label{pre-40} \index{concentration function!L\'evy}
    Let $X$ be a real-valued random variable. \emph{L\'evy's concentration function} is the function $Q_X \colon [0,\infty) \to [0,1]$ defined by
    \begin{gather}\label{pre-e41}
            Q_X(\ell) \coloneqq  \sup_{x\in\real} \Pp\left(X\in [x,x+\ell] \right),\quad \ell\geq 0.
    \end{gather}
\end{definition}
Sometimes, the generalized inverse  $D_X(p) \coloneqq  \inf\left\{\ell \mid Q_X(\ell) > p\right\}$ of $Q_X$ is used instead of $Q_X$; it is usually called \emph{dispersion} or \emph{scattering function}. The standard reference for L\'evy's concentration functions is Hengartner \& Theodorescu~\cite{hengartner73}, further material can be found in Petrov~\cite[\S\S 1.5, 1.8, 2.5]{petrov95} and LeCam~\cite[Chapter 15.2]{lecam86}. Hengartner \& Theodorescu also discuss multivariate versions of the concentration function. We will summarize the most important facts from these sources.
\begin{properties}[of L\'evy's concentration function]\label{pre-42}
    Let $X$ be a real-valued random variable, $F_X(x) = \Pp\left(X\leq x\right)$ the \textup{(}right continuous\textup{)} distribution function and $Q_X$ the concentration function.
    \begin{enumerate}
    \item\label{pre-42-1}
        $Q_X$ is right continuous, increasing and satisfies $\lim_{\ell\to\infty}Q_X(\ell)=1$. In particular, $Q_X$ is the distribution function of some random variable $Y\geq 0$, i.e.\ $Q_X = F_Y$.
    \item\label{pre-42-2}
        $Q_X(0) = \sup_{x\in\real}\Pp\left(X=x\right)$, and $Q_X$ is continuous if, and only if, the distribution function $F_X$ is continuous.
    \item\label{pre-42-3}
        Replacing in \eqref{pre-e41} the closed interval $[x,x+\ell]$ by $(x,x+\ell)$, $(x,x+\ell]$ or $[x,x+\ell)$ gives $Q^-_X(\ell) \coloneqq  Q_X(\ell-)$ if $\ell>0$ and $Q^-_X(0)=0$.
    \item\label{pre-42-4}
        $Q_X$ is subadditive: $Q_X(\ell+\ell')\leq Q_X(\ell)+Q_{X}(\ell')$ for all $\ell,\ell'\geq 0$.
    \item\label{pre-42-5}
        Every subadditive distribution function of a random variable $Y\geq 0$ is also the concentration function of some random variable $X$.
    \end{enumerate}
\end{properties}

Let us mention some typical examples of concentration functions.
\begin{example}\label{pre-44}
    \begin{enumerate}
    \item
        If $X\sim \Normal(m,\sigma^2)$ is normal with mean $m\in\real$, standard deviation $\sigma\geq 0$, and distribution function $\Phi$, then $Q_X(\ell) = 2\Phi(\ell/2\sigma)-1$, $\ell\geq 0$.
    \item
        If $X\sim \lambda \pi^{-1} (\lambda^2 + (x-\alpha)^2)^{-1}$ is Cauchy distributed ($\lambda>0$, $\alpha\in\real$), then $Q_X(\ell) = 2\pi^{-1}\arctan(\ell/2\lambda)$, $\ell\geq 0$.
    \item
        If $X\sim\Uniform[a,b]$ is a uniform random variable on the interval $[a,b]$, then $Q_X(\ell) =  \min\left\{\ell/(b-a); 1\right\}$, $\ell\geq 0$.
    \item
        If $X \sim Y + c$ or $X\sim -Y$, then $Q_X = Q_Y$, i.e.\ the concentration function is invariant under shifts and reflections.
    \end{enumerate}
\end{example}

The concentration function $Q_X$ is often used for sums of independent random variables. In this connection the following lemma is important. The proof is a rather straightforward calculation based on the definition of the concentration function.
\begin{lemma}\label{pre-46}
    Let $X$ and $Y$ be two independent real-valued random variables. Then
    \begin{gather*}
        Q_X(t\ell) \cdot Q_Y((1-t)\ell) \leq Q_{X+Y}(\ell) \leq \min\left\{Q_X(\ell); Q_Y(\ell)\right\},\quad \ell\geq 0,\; t\in [0,1].
    \end{gather*}
\end{lemma}

If we consider the symmetrization of $X^{\mathrm{sy}} \coloneqq  X- X'$ where $Y= X'$ is an iid copy of $X$, then we get for $t=\frac 12$ that $Q_X(\ell)^2 \leq Q_{X^{\mathrm{sy}}}(2\ell)\leq Q_X(2\ell)$.

We are most interested in the case where $X$ is \emph{infinitely divisible}, see Section~\ref{pre-mom}. We write $\phi_X(\xi) = \exp\left(-\psi(\xi)\right)$ for the characteristic function, with the characteristic exponent given by the L\'evy--Khintchine formula \eqref{pre-e71}. The following lemma is, essentially, a refinement of the classical L\'evy truncation estimate. The key to its proof is the observation that
\begin{gather*}
    u(x) = \max\left\{0, 1-|x|\right\}
    \xleftrightarrow[\text{(inverse) Fourier transform}]{\text{characteristic function}}
    \widehat u(\xi) = \left(\frac{\sin(\xi/2)}{\xi/2}\right)^2
\end{gather*}
is a Fourier transform pair consisting of positive functions such that
\begin{gather*}
    \frac 12\I_{[-1/2,1/2]} \leq u(x) \leq \I_{[-1,1]}(x)
    \text{\ \ and\ \ }
    \left(\frac{95}{96}\right)^2\!\!\I_{[-1/2,1/2]}(\xi) \leq \widehat u(\xi) \leq \I_{[-1,1]}(\xi),
\end{gather*}
for all $x\in\real$ and $\xi\in [-1,1]$, respectively; see Fig.~\ref{fig-fourier}.
\begin{figure}[!ht]\centering
    \includegraphics[width = .9\textwidth]{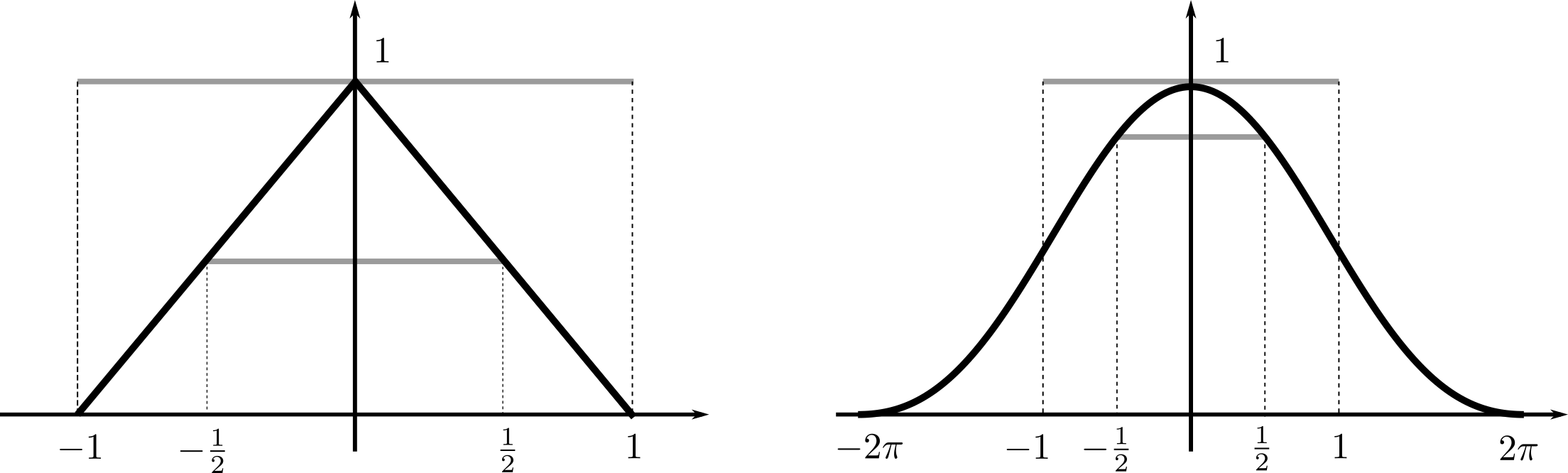}
    \caption{The Fourier pair $u(x) = \max\left\{0, 1-|x|\right\}$ \textup{(}left\textup{)} and $\widehat u(\xi) = \left(\frac{\sin(\xi/2)}{\xi/2}\right)^2$ \textup{(}right\textup{)}.}\label{fig-fourier}
\end{figure}

\begin{lemma}[Esseen~\cite{esseen68}]\label{pre-48} \index{concentration function!vs.\ characteristic function}
    Let $X$ be a real-valued random variable and denote by $\phi_X(\xi) = \Ee e^{i\xi X}$ its characteristic function. Then, for all $\ell,a>0$
    \begin{gather*}
            \frac{95 \ell}{256 \pi\left( 1+2a\ell\right)} \int_{-a}^a |\phi_X(\xi)|^2\,d\xi
            \leq Q_X(\ell)
            \leq \left(\frac{96}{95}\right)^2 \max\left\{\ell, a^{-1}\right\} \int_{-a}^a |\phi_X(\xi)|\,d\xi.
    \end{gather*}
\end{lemma}

We will now apply Lemma~\ref{pre-48} to infinitely divisible random variables. Recall the definition of the truncated censored moment $D(\kolmo;\ell)$ where $\kolmo = \sigma^2\delta_0+\nu$, see \eqref{pre-e76} and the notation introduced in \eqref{pre-e75}. The following result is due to Le Cam~\cite{lecam65} and Esseen~\cite{esseen68}.

\begin{theorem}[Le Cam; Esseen]\label{pre-50} \index{concentration function!vs.\ censored moments}
    Let $X$ be a \textup{(}non-degenerate\textup{)} infinitely divisible real random variable with characteristic exponent $\psi$ and L\'evy triplet $(b,\sigma^2,\nu)$. Set $\kolmo \coloneqq  \sigma^2\delta_0 + \nu$. There are absolute constants $0<c,C<\infty$ such that for all $\ell>0$
    \begin{gather}\label{pre-e49}
        c\min\left\{1,\; \frac{1}{\sqrt{K(\kolmo;\ell)}}\right\} e^{- 4 G(\nu;\ell)}
        \leq Q_X(\ell)
        \leq \frac{C}{\sqrt{D(\kolmo;\ell)}}.
    \end{gather}
\end{theorem}
\begin{proof}
    We begin with an elementary estimate for the cosine. Since $\sin(u)$ is concave on $\left[0,\frac \pi2\right]$, we have $\sin(u)\geq \frac{\sin\left(\frac\pi 2\right)}{\frac \pi 2}\cdot u \geq \frac 12u$, and so
    \begin{gather*}
        1-\cos t
        = \int_0^t \sin(u)\,du
        \geq \frac 12\int_0^t u\,du
        = \frac{1}{4}\,t^2,
        \quad t\in \left[0,\tfrac \pi2\right].
    \end{gather*}
    Since we want to use Lemma~\ref{pre-48}, we need to bound the characteristic function $|\phi_X(\xi)|$. Take $a=\ell^{-1}$ in the upper bound of Lemma~\ref{pre-48} and use the L\'evy--Khintchine formula \eqref{pre-e71} to get
    \begin{align*}
        &\ell\int_{-1/\ell}^{1/\ell} \left|e^{-\psi_X(\xi)}\right| d\xi
        = \int_{-1}^1 \exp\left[-\RE\psi_X(\xi/\ell)\right] d\xi\\
        &= \int_{-1}^1 \exp\left[
            -\tfrac{\sigma^2\xi^2}{2\ell^2}
            -\int_{ |y| \leq \ell} \left(1-\cos\tfrac{y\xi}{\ell}\right)\nu(dy)
            -\int_{ |y| > \ell} \left(1-\cos\tfrac{y\xi}{\ell}\right)\nu(dy)
            \right] d\xi\\
        &\leq \int_{-1}^1
            \exp\left[- \frac{1}{4}K(\kolmo;\ell)\xi^2\right]
            \exp\left[-\int_{ |y| > \ell} \left(1-\cos\frac{y\xi}{\ell}\right)\nu(dy)\right]
            d\xi.
    \end{align*}
    Now we use H\"{o}lder's inequality with
    \begin{gather*}
        \frac 1p = \frac{K(\kolmo;\ell)}{D(\kolmo;\ell)}
        \et
        \frac 1q = \frac{G(\nu;\ell)}{D(\kolmo;\ell)},
    \end{gather*}
    see \eqref{pre-e75} for the definition of $K$ and $G$, and we obtain
    \begin{align*}
        \ell\int_{-1/\ell}^{1/\ell} \left|e^{-\psi_X(\xi)}\right| d\xi
        &\leq (\mathrm{I})^{1/p}\cdot (\mathrm{II})^{1/q}.
    \end{align*}
    The two factors can be estimated in the following way:
    \begin{gather*}
        \mathrm{I}
        = \int_{-1}^1 \exp\left[- \frac{1}{4}D(\kolmo;\ell)\xi^2\right] d\xi
        \leq \frac{C_1}{\sqrt{D(\kolmo;\ell)}}.
    \end{gather*}
 If $\nu(|y|>\ell)=0$, then $\mathrm{II}=1$, and we are done. Otherwise, we use the normalization $\mathrlap{\rule[3pt]{2.5pt}{0pt}\rule[2.5pt]{5pt}{1pt}}\int_A \cdots d\nu \coloneqq  \int_A \cdots \frac{d\nu}{\nu(A)}$, and Jensen's inequality:
    \begin{align*}
    \mathrm{II}
    &= \int_{-1}^1 \exp\left[- D(\kolmo;\ell) \mathrlap{\rule[3pt]{4.25pt}{0pt}\rule[2.5pt]{5pt}{1pt}}\int_{ |y| > \ell} \left(1-\cos\frac{y\xi}{\ell}\right)\nu(dy) \right] d\xi\\
    &\leq \int_{-1}^1 \mathrlap{\rule[3pt]{4.25pt}{0pt}\rule[2.5pt]{5pt}{1pt}}\int_{ |y| > \ell} \exp\left[- D(\kolmo;\ell)\left(1-\cos\frac{y\xi}{\ell}\right) \right] \nu(dy)\, d\xi\\
    &= 2\mathrlap{\rule[3pt]{4.25pt}{0pt}\rule[2.5pt]{5pt}{1pt}}\int_{ |y| > \ell} \int_{0}^{|y|/\ell} \exp\left[- D(\kolmo;\ell)\left(1-\cos t\right) \right] \frac{dt}{|y|/\ell}\,\nu(dy).
    \end{align*}
    Since $1-\cos t$ is periodic, symmetric (even) around $\frac\pi 2$ in the interval $[0,\pi]$, and $|y|/\ell > 1$, we get
    \begin{gather*}
    \mathrm{II}
    \leq  4\mathrlap{\rule[3pt]{4.25pt}{0pt}\rule[2.5pt]{5pt}{1pt}}\int_{|y| > \ell} \frac{\left\lfloor\frac{|y|/\ell}{\pi}\right\rfloor+1}{|y|/\ell} \int_{0}^{\pi/2} \exp\left[- D(\kolmo;\ell)\frac{1}{2}\,t^2 \right] dt\,\nu(dy)
    \leq \frac{C_2}{\sqrt{D(\kolmo;\ell)}}.
    \end{gather*}
    Since $C_1,C_2$ are absolute constants and $p^{-1}+q^{-1}=1$, the upper bound in \eqref{pre-e49} follows.

    For the lower bound in \eqref{pre-e49} we use the lower estimate in Lemma~\ref{pre-48} with $a=\ell^{-1}$ and the elementary inequality $1-\cos u \leq \frac 12 u^2$.
    \begin{equation*}
    \mathclap{\begin{aligned}[b]
        &Q_X(\ell)
        \geq c\ell \int_{-1/\ell}^{1/\ell} \left|e^{-\psi_X(\xi)}\right|^2\,d\xi\\
        &= c\int_{-1}^{1} \!\!\exp\left[-\tfrac{\sigma^2\xi^2}{\ell^2}
            - 2\int_{|y|\leq\ell}\!\left[1-\cos\tfrac{y\xi}{\ell}\right]\nu(dy)
            - 2\int_{|y|>\ell}\!\left[1-\cos\tfrac{y\xi}{\ell}\right]\nu(dy)
        \right] d\xi\\
        &\geq c\int_{-1}^{1} \exp\left[-\xi^2\left(\frac{\sigma^2}{\ell^2}
            + \frac 1{\ell^2}\int_{|y|\leq\ell}y^2\,\nu(dy)\right)
            - 4\int_{|y|>\ell}\nu(dy)
        \right] d\xi\\
        &\geq c\int_{-1}^{1} e^{- \xi^2K(\kolmo;\ell)}\,d\xi \, e^{-4G(\nu;\ell)}\\
        &\geq c'\min\left\{1;\;\frac{1}{\sqrt{K(\kolmo;\ell)}}\right\} e^{-4G(\nu;\ell)}.
    \end{aligned}}\qedhere
    \end{equation*}
\end{proof}

\begin{remark}\label{pre-52}
{a)}
With a more careful argument based on Lemma~\ref{pre-46} one can improve the lower bound in \eqref{pre-e49} and get $G(\nu,\ell)$ instead of $4G(\nu;\ell)$ in the exponent, see Hengartner \& Theodorescu~\cite[pp.~53--54]{hengartner73}.

\smallskip{b)}
There is an interesting extension of the upper bound in \eqref{pre-e49} to arbitrary, not necessarily infinitely divisible, random variables $X$. The starting point is the elementary convexity estimate $(z+1)\leq e^z$, $z\in\real$, which yields
\begin{gather*}
    e^z \geq 1+z
    \iff -2x \leq -\left(1-e^{-2x}\right)
    \iff \sqrt{e^{-2x}} \leq \sqrt{e^{-\left(1-e^{-2x}\right)}},
\end{gather*}
and this implies
\begin{gather*}
    \left|\phi_X(\xi)\right| \leq \exp\left[-\tfrac 12\left(1-\left|\phi_X(\xi)\right|^2\right)\right].
\end{gather*}
The right hand side is the characteristic function of a compound-Poisson type random variable, since the exponent is given by a L\'evy--Khintchine formula:
\begin{gather*}
    \frac 12\left(1-\left|\phi_X(\xi)\right|^2\right)
    = \frac 12\int \left(1-\cos(y\xi)\right) \rho*\widetilde\rho(dy)
\end{gather*}
where $X\sim\rho$ and $-X\sim\widetilde\rho$. The proof of Theorem~\ref{pre-50} immediately gives
\begin{gather*}
    Q_X(\ell) \leq \frac{C}{\sqrt{D\bigl(\frac 12\rho*\widetilde\rho;\ell\bigr)}}.
\end{gather*}
\end{remark}

\begin{remark}\label{pre-54}\index{concentration function!vs.\ censored moments}
    Combining the estimates from Theorem~\ref{pre-50} with Theorem~\ref{pre-76} gives at once the following estimates for an infinitely divisible random variable $X$ with characteristic function $\psi$ in terms of the censored second moment $D(\kolmo;\cdot)$, $\kolmo = \sigma^2\delta_0 + \nu$,
    \begin{align}
        c\min\left\{1,\;\frac{1}{\sqrt{D(\kolmo;\ell)}}\right\} e^{-4 D(\kolmo;\ell)}
        \leq Q_X(\ell)
        \leq \frac{C}{\sqrt{D(\kolmo;\ell)}},
    \intertext{($0<c<C<\infty$ are absolute constants) and the maximum function $\psi^*$: }
        c'\min\left\{1,\;\frac{1}{\sqrt{\psi^*(\ell^{-1})}}\right\} e^{- c'' \psi^*(\ell^{-1})}
        \leq Q_X(\ell)
        \leq \frac{C'}{\sqrt{\psi^*(\ell^{-1})}}
    \end{align}
 ($0<c'<C''<\infty$ and $c''>0$ are absolute constants).
\end{remark}

\paragraph{Milman's concentration function.}

A different type of concentration function was introduced by Amir \& Milman~\cite{ami-mil80} to study isoperimetric inequalities using the \emph{concentration of measure phenomenon}. This approach goes back to Milman~\cite{milman71} who proved Dvoretzky's theorem on spherical sections of convex sets using an estimate for the volume of the $n$-dimensional spherical cap due to L\'evy~\cite[pp.~260--285]{levy22}, see also L\'evy~\cite[209--234]{levy51}: if $S$ is a cap of the sphere $\Ss^{n-1}$ in $\real^n$ and $A\subseteq\Ss^{n-1}$ another set having the same Haar measure $\sigma(\cdot)$ as $S$, then the $\epsilon$-enlargements $A_\epsilon \coloneqq  \{x\in\Ss^{n-1} \mid \mathrm{dist}(x,A) \leq \epsilon\}$ satisfy $\sigma(A_\epsilon)\geq \sigma(S_\epsilon)$. Starting in 1988, in a series of papers Talagrand turned the concentration of measures approach into a universal tool in probability, analysis and the geometry of Banach spaces. Arguably, the contribution Talagrand~\cite{talagrand95} on the measure concentration on product spaces marks the broad breakthrough of the theory. Very good presentations are due to McDiarmid~\cite{mcdiarmid98}, Ledoux~\cite{ledoux01} and Boucheron \emph{et al.}~\cite{boucheron-et-al13}. Below we restrict ourselves to the concentration of the law of a random variable.
\begin{definition}[Milman]\label{pre-60} \index{concentration function!Milman}
    Let $X$ be a random variable taking values in a metric space $(M,d)$. Denote by $\mathrm{K}_\ell(0)$ the closed metric ball of radius $\ell$ and write $A_\ell = A+\mathrm{K}_\ell(0)$ for the $\ell$-enlargement of a Borel set $A\subseteq M$. \emph{Milman's concentration function} is the function $\alpha_X\colon  (0,\infty)\to [0,1/2]$ defined by
    \begin{gather}\label{pre-e61}
        \alpha_X(\ell)
        = \sup\left\{1 - \Pp\left(X\in A_\ell\right) \;\big|\;  \Pp\left(X\in A\right)\geq \frac 12\right\},
        \quad \ell>0.
    \end{gather}
\end{definition}

\begin{properties}[Milman's concentration function]\label{pre-62}
    Milman's con\-cen\-tra\-tion \linebreak function $\alpha_X$ of an $M$-valued random variable  $X$ is decreasing, right continuous and satisfies $\alpha_X(0+)\leq\frac 12$ and $\lim_{\ell\to\infty}\alpha_X(\ell) = 0$.
\end{properties}
\begin{proof}
    Let $\ell < \ell+\epsilon$. Since $A_\ell\subseteq A_{\ell+\epsilon}$ we have $\Pp(X \in A_{\ell+\epsilon}) \geq \Pp(X \in A_{\ell})$, and so $\alpha_X(\ell+\epsilon)\leq\alpha_X(\ell)$.
    The properties $\alpha_X(0)\leq\frac 12$ and $\lim_{\ell\to\infty}\alpha_X(\ell)=0$ are trivial.

    Continuity from the right follows from the right continuity of measures and the fact that we can always interchange suprema:
    \begin{equation*}
    \begin{aligned}[b]
        \sup_{\epsilon>0} \alpha_X(\ell+\epsilon)
        &= \sup \left\{1-\inf_{\epsilon>0}\Pp\left(X\in A_{\ell+\epsilon}\right) \;\big|\; \Pp\left(X\in A\right)\geq\frac 12\right\}\\
        &= \sup \left\{1-\Pp\left(X\in A_{\ell}\right) \;\big|\; \Pp\left(X\in A\right)\geq\frac 12\right\}
        = \alpha_X(\ell).
    \end{aligned}\qedhere
    \end{equation*}
\end{proof}

If $M=\real$ with the usual Euclidean distance, then we have the following relation between L\'evy's and Milman's concentration functions.
\begin{lemma}\label{pre-64}
    Let $X$ be a real random variable and denote by $Q_X$ and $\alpha_X$ the concentration functions of L\'evy and Milman, respectively. Then
    \begin{gather}\label{pre-e65}
        1-2\alpha_X(\ell) \leq Q_X(2\ell) \leq 2 Q_X(\ell).
    \end{gather}
\end{lemma}
\begin{proof}
    Let $m\in\real$ be a median of $X$, i.e.\ a value such that both $\Pp\left(X\geq m\right)\geq \frac 12$ and $\Pp\left(X > m\right)\leq \frac 12$ (or, equivalently, $\Pp\left(X\leq m\right)\geq\frac 12$). Taking $A=(-\infty,m]$ in the definition of $\alpha_X$ shows
    \begin{gather*}
        1-\Pp\left(X\in (-\infty, m+\ell]\right) \leq \alpha_X(\ell)
        \quad\text{and so}\quad
        \Pp\left(X-m>\ell\right)\leq \alpha_X(\ell).
    \end{gather*}
    In the same way we conclude with $A=[m,\infty)$ that $\Pp\left(X-m<-\ell\right)\leq\alpha_X(\ell)$, i.e.
    \begin{gather*}
        \Pp\left(|X-m|>\ell\right)
        \leq \Pp\left(X-m > \ell\right) + \Pp\left(X-m < -\ell\right)
        \leq 2\alpha_X(\ell).
    \end{gather*}
    From this we get
    \begin{equation*}
        1-2\alpha_X(\ell)
        \leq \Pp\left(|X-m|\leq\ell\right)
        \leq \sup_{x\in\real} \Pp\left(|X-x|\leq\ell\right)
        = Q_X(2\ell).
    \end{equation*}
    The estimate $Q_X(2\ell)\leq 2Q_X(\ell)$ follows from the subadditivity of $Q_X$, see Property~\ref{pre-42} no.~\ref{pre-42-4}.
\end{proof}
The following example shows that the estimate in Lemma~\ref{pre-64} is not sharp and that one cannot expect estimates like $Q_X(2\ell)\leq C(1-2\alpha_X(\ell))$ or $Q_X(\ell)\leq C(1-2\alpha_X(\ell))$.
\begin{example}\label{pre-66}
    Consider a Bernoulli random variable $X$ with $\Pp\left(X=-1\right)=1-p$ and $\Pp\left(X=1\right)=p$. Without loss of generality we assume that $p\geq \frac 12$. Then we have
    \begin{gather*}
        Q_X(\ell) = p\I_{[0,2)}(\ell) + \I_{[2,\infty)}(\ell),\quad
        \alpha_X(\ell) = (1-p)\I_{[0,2)}(\ell),\\
        \text{and}\quad
        1-2\alpha_X(\ell) = (2p-1)\I_{[0,2)}(\ell) + \I_{[2,\infty)}(\ell).
    \end{gather*}
    Thus, $1-2\alpha_X(\ell)\leq Q_X(2\ell)$ always holds, while $Q_X(2\ell)\leq c(1-2\alpha_X(2\ell))$ is only possible for $p> \frac 12$ and with a constant depending on $p$.
\end{example}

\subsection{Good-$\lambda$ Inequalities}\label{pre-lambda}

Good $\lambda$-inequalities were introduced by Burkholder \& Gundy in order to study both (maximal) inequalities for martingales and maximal functions in analysis. Rough\-ly speaking, good-$\lambda$ inequalities provide a means to get from estimates for distribution functions to estimates of $L^p$-norms $\|X\|_{L^p}^p = \Ee\left[|X|^p\right]$, or more general moments, of the random variable $X$. The key tool is the following \emph{layer cake formula}, which expresses (generalized) moments of a random variable via the distribution function. As usual, we allow integrals of positive functions to attain values in $[0,\infty]$.
\begin{lemma}\label{pre-10} \index{layer cake formula}
    Let $X\geq 0$ be a positive random variable and $\phi\colon [0,\infty)\to [0,\infty]$ a positive, right continuous increasing function with $\phi(0)=0$. Then
    \begin{gather}\label{pre-e11}
        \Ee \left[\phi(X)\right]
        = \int_0^\infty \Pp(X\geq\lambda)\,d\phi(\lambda).
    \intertext{If, in addition, $\phi$ is absolutely continuous, then}
    \label{pre-e13}
        \Ee \left[\phi(X)\right]
        = \int_0^\infty \phi'(\lambda) \,\Pp(X \geq \lambda)\,d\lambda
        = \int_0^\infty \phi'(\lambda) \,\Pp(X > \lambda)\,d\lambda.
    \end{gather}
\end{lemma}
\begin{proof}
    Since $\phi$ is right continuous, $\nu_\phi((a,b]) = d\phi((a,b]) \coloneqq  \phi(b)-\phi(a)$ defines a Borel measure. An application of Tonelli's theorem gives
    \begin{align*}
        \Ee \left[\phi(X)\right]
        = \int_\Omega \int_0^\infty \I_{(0,X]}(\lambda)\, d\phi(\lambda)\,d\Pp
        &= \int_0^\infty \int_\Omega \I_{[\lambda,\infty)}(X) \,d\Pp\,d\phi(\lambda)\\
        &= \int_0^\infty \Pp(X\geq \lambda)\,d\phi(\lambda).
    \end{align*}
    If $\phi$ is absolutely continuous, the same argument with $\int_0^\infty \I_{(0,X]}(\lambda) \,d\phi(\lambda) = \phi(X) = \int_0^X \phi'(\lambda)\,d\lambda$ yields the second formula. Here we use that the set of $\lambda\geq 0$ for which $\Pp(X\geq \lambda) \neq \Pp(X > \lambda)$ is countable, hence a Lebesgue null set.
\end{proof}

Let us begin with Doob's~\cite[Chapter VII.\S3]{doob53} classical argument, which illustrates the use of the layer cake formula.
\begin{lemma}[Doob]\label{pre-12}
    Let $X,Y\geq 0$ be positive random variables which satisfy
    \begin{gather}\label{pre-e15}
        \Pp(X> \lambda) \leq \frac 1\lambda \int_{\{X\geq\lambda\}} Y\,d\Pp  \fa  \lambda>0.
    \end{gather}
    Then one has
    \begin{gather}\label{pre-e17}
        \Ee\left[X^p\right]
        \leq
        \begin{dcases}
            \left(\frac p{p-1}\right)^p \Ee\left[Y^p\right], &\text{if\ \ }p\in (1,\infty),\\
            \frac{e}{e-1}\left(1+\Ee\left[Y\log^+Y\right]\right), &\text{if\ \ }p=1.
        \end{dcases}
    \end{gather}
\end{lemma}
\begin{proof}
    Without loss of generality we may assume that $X$ is bounded. To see this, we note that \eqref{pre-e15} remains valid if we replace $X \rightsquigarrow X\wedge n$ and observe that $\Ee\left[X^p\right] = \sup_{n\geq 1} \Ee\left[(X\wedge n)^p\right]$ by monotone convergence.

    Let $\phi(x) = x^p$ or $\phi(x) = (x-1)^+$. Combining the layer cake formula with the estimate \eqref{pre-e15} yields
    \begin{align*}
        \Ee\left[\phi(X)\right]
        &= \int_0^\infty \phi'(\lambda)\,\Pp(X>\lambda)\,d\lambda\\
        &\leq \int_0^\infty \phi'(\lambda)\left(\frac 1\lambda\int_\Omega \I_{\{X>\lambda\}} Y\,d\Pp\right)d\lambda\\
        &= \int_\Omega \int_0^X \frac{\phi'(\lambda)}{\lambda}\,d\lambda \: Y\,d\Pp.
    \end{align*}
    If $p\in (1,\infty)$ and $\phi(x)=x^p$, we see that
    \begin{gather*}
        \int_0^X \frac{\phi'(\lambda)}{\lambda}\,d\lambda
        = \int_0^X p\lambda^{p-2}\,d\lambda
        = \frac 1{p-1} X^{p-1}.
    \end{gather*}
    Therefore, we can use H\"{o}lder's inequality with $p^{-1}+q^{-1}=1$,
    \begin{gather*}
        \Ee\left[X^p\right]
        \leq \frac 1{p-1} \Ee\left[Y X^{p-1}\right]
        \leq \frac 1{p-1} \left(\Ee\left[Y^p\right]\right)^{\frac 1p} \left(\Ee\left[X^p\right]\right)^{1- \frac 1p}
    \end{gather*}
    which yields \eqref{pre-e17}.

    If $p=1$  and $\phi(x)=(x-1)^+$, we have
    \begin{gather*}
        \int_0^X \frac{\phi'(\lambda)}{\lambda}\,d\lambda
        = \I_{\{X\geq 1\}}\int_1^X \frac 1\lambda\,d\lambda
        = \I_{\{X\geq 1\}} \log(X).
    \end{gather*}
    From the elementary inequality $x^{-1}\log(x) \leq e^{-1}$, $x>0$,---just find the maximum of $x^{-1}\log(x)$---we see that
    \begin{gather*}
        \log(x) = \log \left(y\frac xy\right) = \log(y) + \log\left(\frac xy\right) \leq \log^+(y) + \frac{x}{ye},
    \end{gather*}
    hence $y\log(x) \leq y\log^+(y) + x/e$.
    Therefore,
    \begin{gather*}
        \Ee\left[X\right]
        \leq 1 + \Ee\left[Y \I_{\{X\geq 1\}} \log(X)\right]
        \leq  1 + \Ee\left[Y\log^+(Y)\right] + \frac 1e\Ee\left[X\right],
    \end{gather*}
    from which it is easy to get \eqref{pre-e17}.
\end{proof}

The following lemma will be a substitute for Lemma~\ref{pre-12} if $0<p<1$.
\begin{lemma}\label{pre-13}
    Let $X,Y\geq 0$ be positive random variables which satisfy
    \begin{gather}\label{pre-e18}
        \Pp(X > \lambda) \leq \frac 1\lambda \Ee Y  \fa  \lambda>0.
    \end{gather}
    Then one has
    \begin{gather}\label{pre-e19}
        \Ee\left[X^p\right]
        \leq
        \frac 1{1-p} \left(\Ee Y\right)^p
        \fa p\in (0,1).
    \end{gather}
\end{lemma}
\begin{proof}
    Using the layer-cake formula for $\phi(\lambda)=\lambda$ we get
    \begin{align*}
        \Ee \left[X^p\right]
        = \int_0^\infty \Pp\left(X > \lambda^{1/p}\right)  d\lambda
        &\leq \int_0^x d\lambda + \Ee Y \int_x^\infty \lambda^{-1/p}\,d\lambda\\
        &= x + \frac p{1-p} x^{1-1/p}\Ee Y.
    \end{align*}
    The right hand side becomes minimal if $x=\left(\Ee Y \right)^p$, and the claim follows.
\end{proof}

The next lemma can be seen as a partial converse of \eqref{pre-e17} for $p=1$.
\begin{lemma}\label{pre-13-2}
    Let $0\leq X \leq Y$ be positive random variables such that for some constants $a,b>1$
    \begin{gather}\label{pre-e20}
        \frac 1\lambda \int_{\{Y>\lambda\}} X\,d\Pp \leq b\Pp(Y \geq \lambda) \fa  \lambda>a.
    \end{gather}
    Then $\Ee Y < \infty$ implies that $\Ee\left[X\log^+ X\right]<\infty$.
\end{lemma}
\begin{proof}
    Assume that $\Ee Y<\infty$. From $X\leq Y$ we see that $\Ee X<\infty$.
    Let $a,b>0$ be as in the statement of the lemma. The layer-cake formula shows
    \begin{align*}
        \Ee\left[X\log^+X\right]
        &= \int_1^a \int_{\{X>\lambda\}} X\,d\Pp\,\frac{d\lambda}{\lambda}
        + \int_a^\infty\int_{\{X>\lambda\}} X\,d\Pp\,\frac{d\lambda}{\lambda}\\
        &\leq \int_1^a \frac{d\lambda}{\lambda}\:\Ee X + \int_a^\infty\int_{\{Y>\lambda\}} X\,d\Pp\,\frac{d\lambda}{\lambda}\\
        &\leq \log(a) \, \Ee X + \int_0^\infty b\,\Pp(X \geq \lambda)\,d\lambda\\
        &= \log(a)\, \Ee X + b \, \Ee Y.
    \end{align*}
    This proves that $\Ee\left[X\log^+X\right]<\infty$ if $\Ee\left[Y\right]<\infty$.
\end{proof}

The following \emph{good-$\lambda$-inequality} is due to Burkholder~\cite[Lem.~7.1]{burkholder73}, earlier versions can be found in Burkholder \& Gundy~\cite[Proof Thm.~4.1]{bur-gun70}, Burkholder, Davies \& Gundy~\cite[Proof Lem.~3.1]{bur-dav-gun72}, and Burkholder \& Gundy~\cite[Proof Thm.~1]{bur-gun72}.
\begin{lemma}\label{pre-14} \index{inequality!good-$\lambda$}
    Let $0<p<\infty$ and assume that $X,Y\geq 0$ are positive random variables. If there exist $\beta>1$, $\delta\in (0,1)$ and $0 < \epsilon < \frac 12\beta^{-p}$ such that
    \begin{gather}\label{pre-e21}
        \Pp(X > \beta\lambda,\: Y < \delta\lambda)
        \leq \epsilon\Pp(X\geq \lambda)
         \fa \lambda>0
    \end{gather}
    holds, then there exists a constant $c=c(\beta,\delta,\epsilon,p)$ such that
    \begin{gather}\label{pre-e22}
        \Ee\left[X^p\right] \leq c \Ee\left[Y^p\right].
    \end{gather}
\end{lemma}
\begin{proof}
    As in the proof of Lemma~\ref{pre-12} we can assume that the random variable $X$ is bounded---note that \eqref{pre-e21} remains valid, if we truncate $X$. We have
    \begin{align*}
        \Pp(X > \beta\lambda)
        &\leq \Pp(X > \beta\lambda,\; Y < \delta\lambda) + \Pp(Y \geq \delta\lambda)\\
        &\leq \epsilon\Pp(X\geq\lambda) + \Pp(Y\geq \delta\lambda).
    \end{align*}
    An application of the layer cake formula for $\phi(x)=x^p$ shows that
    \begin{align*}
        \Ee\left[\left(\frac 1\beta X\right)^p\right]
        &\leq \epsilon\Ee\left[X^p\right] + \Ee\left[\left(\frac 1\delta Y\right)^p\right]\\
        &\leq \frac 1{2\beta^p} \Ee\left[X^p\right] + \Ee\left[\left(\frac 1\delta Y\right)^p\right].
    \end{align*}
    This yields \eqref{pre-e22} with $c={2\beta^p}/{\delta^p}$.
\end{proof}

\begin{remark}\label{pre-16}
It is possible to replace $x\mapsto x^p$ in \eqref{pre-e22} by a positive, increasing function $\phi\colon [0,\infty)\to [0,\infty)$ with $\phi(0)=0$ which satisfies a \enquote{doubling condition}, i.e.\
\begin{gather*} \index{doubling condition}
    \phi(2x)\leq c\phi(x)\quad\text{for some $c\geq 1$ and all $x\geq 0$}.
\end{gather*}
The doubling condition implies that $\phi(x)$ has at most polynomial growth. In fact, if $\beta \geq 0$ and $\kappa = \log(\beta)/\log(2)$, then
\begin{gather*}
    \phi(\beta x)
    = \phi(2^\kappa x)
    \leq \phi(2^{\entier\kappa + 1}x)
    \leq c c^\kappa\phi(x)
    = c \beta^{\log(c)/\log(2)}\phi(x).
\end{gather*}
As a consequence, we see for $\beta>1$, $\delta\in (0,1)$, and suitable constants $c_\beta$ and $c_\delta$
\begin{gather*}
    \Ee\left[\phi(X/\beta)\right] \geq c_\beta \Ee\left[\phi(X)\right]
    \quad\text{and}\quad
    \Ee\left[\phi(Y/\delta)\right] \leq c_\delta \Ee\left[\phi(Y)\right],
\end{gather*}
which means that Lemma~\ref{pre-14} remains valid for this class of functions. A thorough discussion for this and other classes can be found in Lenglart \emph{et al.}~\cite{len-lep-pra80}.
\end{remark}

The following variant of the good-$\lambda$-inequality reveals its (formal) connection with interpolation theorems. The proof (for $0<p<1$) is taken from Revuz \& Yor~\cite[Lem.~IV(4.6)]{rev-yor99}, see also Schilling~\cite[Lem.~19.19]{schilling-bm}, while the general case $0<p<r$ follows the lines of Stein~\cite[Chapter I.\S4]{stein70}.
\begin{lemma}\label{pre-18}\index{inequality!good-$\lambda$}
    Let $X,Y\geq 0$ be positive random variables such that for some $r>0$
    \begin{align}\label{pre-e23}
        \Pp(X> \lambda, Y < \lambda)
        &\leq \frac 1{\lambda^r}\int_{\{Y<\lambda\}} Y^r\,d\Pp \fa  \lambda>0,
    \intertext{or}\label{pre-e25}
        \Pp(X> \lambda, Y < \lambda)
        &\leq \frac 1{\lambda^r}\int (Y\wedge\lambda)^r\,d\Pp \fa  \lambda>0,
    \end{align}
    holds. Then one has for all $0<p<r$
    \begin{gather}\label{pre-e27}
        \Ee\left[X^p\right]
        \leq \left(\frac r{r-p} + 1\right) \Ee\left[Y^p\right].
    \end{gather}
\end{lemma}
\begin{proof}
    Since
    \begin{gather*}
        \frac 1{\lambda^r} \int (Y\wedge\lambda)^r\,d\Pp
        = \frac 1{\lambda^r}\int_{\{Y<\lambda\}} Y^r\,d\Pp + \Pp(Y\geq\lambda),
    \end{gather*}
    the argument below shows that either \eqref{pre-e23} or \eqref{pre-e25} yield \eqref{pre-e27}, and that the respective constants in \eqref{pre-e27} differ by $1$.

    From \eqref{pre-e23} we get
    \begin{align*}
        \Pp(X>\lambda)
        &= \Pp(X> \lambda, Y\geq \lambda) + \Pp(X> \lambda, Y < \lambda)\\
        &\leq \Pp(Y\geq \lambda) + \frac 1{\lambda^r}\int_{\{Y<\lambda\}} Y^r\,d\Pp.
    \end{align*}
    At this point we see that \eqref{pre-e25} just adds one more term $\Pp(Y\geq \lambda)$ on the right hand side. The layer cake formula with $\phi(x) = x^p$ yields $\Ee\left[X^p\right]$ on the left, and $\Ee\left[Y^p\right]$ from the first summand on the right side. The second term contributes
    \begin{align*}
        \int_0^\infty p\lambda^{p-1} \lambda^{-r} \int_{\{Y < \lambda\}}  Y^r  \,d\Pp\,d\lambda
        &= \int_\Omega\int_Y^\infty p\lambda^{-1 - (r-p)}\,d\lambda \, Y^r\,d\Pp\\
        &= \frac{p}{r-p}\int_\Omega  Y^{p}\,d\Pp;
    \end{align*}
    here we use that $p<r$. Together we arrive at \eqref{pre-e27} with the (slightly better) constant $r/(r-p)$.
\end{proof}

\begin{remark}\label{pre-20}
    Let us briefly make the connection to interpolation of operators in the scale $L^p(\mu)$, $1\leq p<\infty$, where $\mu$ is a $\sigma$-finite Borel measure on $\real^d$. An operator $T$ defined on $L^p(\mu)$ and taking values in the Borel-measurable functions $\Bcal(\real^d)$ is said to be \emph{of weak type $(p,q)$}, $1\leq p,q < \infty$, if there is a constant $c=c_{p,q}$ such that
    \begin{gather}\label{pre-e31}
        \mu\left(\left\{ |Tf| > \lambda\right\}\right)
        \leq \left(\frac{c}{\lambda} \|f\|_{L^p}\right)^{q} \fa \lambda > 0,\; f\in L^p(\mu).
    \end{gather}
    The operator is called of (\emph{strong}) \emph{type $(p,q)$}, if for some constant $C_{p,q}$
    \begin{gather}\label{pre-e33}
        \|Tf\|_{L^q}
        \leq C_{p,q}\|f\|_{L^p}
         \fa  f\in L^p(\mu).
    \end{gather}
 The techniques developed in this section allow us to prove a simple version of the \emph{Marcinkiewicz interpolation theorem}, see e.g.\ Stein~\cite[Chapter I.\S4.2]{stein70}:\footnote{There are also stronger versions of the Marcinkiewicz interpolation theorem for general weak-type $(p_0,q_0)$ and weak-type $(p_1,q_1)$ operators, cf.\ Bennett \& Sharpley~\cite[Chapter 4.\S4]{ben-sha88}.}

    \begin{theorem*}[Marcinkiewicz] \index{interpolation}
 Let $T$ be an operator, which is defined on the space $L^{1}(\mu)+L^{r}(\mu) \coloneqq \left\{f+g \mid f\in L^1(\mu),\; g\in L^r(\mu)\right\}$ for some $r\in (1,\infty)$. If $T$ is sublinear, i.e.\ $|T(f+g)(x)|\leq |Tf(x)|+|Tg(x)|$, and both of weak-type $(1,1)$ and of weak-type $(r,r)$, then it is of \textup{(}strong\textup{)} type $(p,p)$ for any $1<p<r$, i.e.\ $\|Tf\|_{L^p}\leq c_{r,p}\|f\|_{L^p}$.
    \end{theorem*}
    \begin{proof}
    Because of the subadditivity of $T$ it is enough to consider positive $f\geq 0$, since we have
    \begin{align*}
        \left\{|Tf|> 4\lambda\right\}
        \subseteq
        \left\{|T(f^+)|> 2\lambda\right\}
        \cup \left\{|T(f^-)|> 2\lambda\right\}.
    \end{align*}
    Let $f\geq 0$ and use again subadditivity to see
    \begin{align*}
        \left\{|Tf|> 2\lambda\right\}
        \subseteq
        \left\{|T(f\wedge \lambda)|>\lambda\right\}
        \cup \left\{|T(f-f\wedge \lambda)|> \lambda\right\},
    \end{align*}
    hence
    \begin{align}\label{pre-e35}
        \mu\left(\left\{|Tf|> 2\lambda\right\}\right)
        \leq
        \mu\left(\left\{|T(f\wedge \lambda)|> \lambda\right\}\right)
        + \mu\left(\left\{|T(f-f\wedge \lambda)|> \lambda\right\}\right).
    \end{align}
    For the measures appearing on the right hand side we use the weak-$(r,r)$ resp.\ weak-$(1,1)$ bounds to get
    \begin{gather}\label{pre-e37}
        \mu\left(\left\{|T(f\wedge \lambda)|> \lambda\right\}\right)
        \leq \frac {c_{r,r}}{\lambda^r} \int (f\wedge \lambda)^r\,d\mu
    \intertext{as well as}\label{pre-e39}
        \mu\left(\left\{|T(f - f\wedge \lambda)|> \lambda\right\}\right)
        \leq \frac {c_{1,1}}{\lambda} \int (f-f\wedge \lambda)\,d\mu
        \leq \frac {c_{1,1}}{\lambda} \int_{\{f\geq\lambda\}} f\,d\mu.
    \end{gather}
    We can now apply the layer cake formula with $\phi(x)=x^p$ to \eqref{pre-e35} and use the estimates \eqref{pre-e35}, \eqref{pre-e37}. On the left-hand side of \eqref{pre-e35} we get $\frac 12 \|Tf\|_{L^p}^{{p}}$.  The strategy used in the proof of Lemma~\ref{pre-12} works for $p>1$, and yields the bound $K_{1,p} \|f\|_{L^p}^{{p}}$ for the term in \eqref{pre-e39}. Indeed, we have
    \begin{equation*}\begin{aligned}
        \int_0^\infty p\lambda^{p-1} \mu\left(\left\{|T(f-f\wedge \lambda)|> \lambda\right\}\right) d\lambda
        &\leq c_{1,1} \int_0^\infty p\lambda^{p-1} \left(\frac 1\lambda \int \I_{\{f\geq\lambda\}} f\,d\mu\right) d\lambda\\
        &=    c_{1,1} \int f(x) \int_0^{f(x)} p\lambda^{p-2} d\lambda\,\mu(dx)\\
        &=    \frac{c_{1,1}p}{p-1}\|f\|_{L^p}^p.
    \end{aligned}\end{equation*}
    For the term \eqref{pre-e37} appearing on the right-hand side, we can literally use the calculation from the proof of Lemma~\ref{pre-18} with $p<r$ to get the bound $\kappa_{r,p} \|f\|_{L^p}^{{p}}$ with $\kappa_{r,p}= r/(r-p)$.

    As one would expect, the constants $\kappa_{r,p}$ and $K_{1,p}$ become infinite as $p\uparrow r$ and $p\downarrow 1$, respectively.
    \end{proof}

    Further applications of the good-$\lambda$ inequality in harmonic analysis are discussed in Torchinsky~\cite[Chapter~XIII]{torchinsky86}.
\end{remark}

 Good-$\lambda$ inequalities play also an important role in the extrapolation problem for function spaces, see Geiss~\cite{geiss97} and Hyt\"{o}nen \emph{et al.}~\cite[Ch.~3]{hyt-et-al16}. That paper also contains a discussion how one can formulate certain weak-type $L^1$-estimates in a $\mathrm{BMO}$ (functions of bounded mean oscillation, see Remark~\ref{gen-37}) setting in order to obtain strong $L^p$-estimates.

\section{Some General Principles}\label{gen}

Let $(X_t)_{t\geq 0}$ be a stochastic process with values in $\rd$ defined on a filtered probability space $(\Omega,\Ascr,\Pp,\Fscr_t)$; we assume that the process is \emph{adapted} to the filtration, i.e.\ each random variable $X_t$ is $\Fscr_t$-measurable. Without further assumptions, the supremum processes $X_t^* \coloneqq  \sup_{s\leq t}|X_{s}|$ or $M_t \coloneqq  \sup_{s\leq t} X_s$ may not be measurable. Typically, if $t\mapsto X_t$ is left- or right continuous, this is not an issue. For a long time the notion of \emph{separability} was used.

\paragraph{Separability.}
A process is called \emph{separable} with \emph{separating set} $D\subseteq [0,\infty)$, if for some countable set $D$ the set $\{(t,X_t(\omega)) \mid t\in D\}$ is for almost all $\omega\in\Omega$ dense (w.r.t.\ the Euclidean topology in $[0,\infty)\times\rd$) in the graph $\{(t,X_t(\omega))\mid t\geq 0\}$. In particular, this allows us to calculate suprema etc.\ along the countable set $D$, rendering everything measurable.
A full discussion on separability can be found in Dellacherie \& Meyer~\cite[pp.~153--161]{dm-1} and Lo\`eve~\cite[Vol.~2, pp.~170--178]{loeve77}.
\index{separable process}

The notion of \emph{separability} is hardly seen in modern texts. This is mainly due to the following two results:
\begin{fact}[Doob]\label{gen-01}\index{separable process!sufficient condition}
    Every stochastic process $(X_t)_{t\geq 0}$ with values in $\rd$ has a separable modification taking values in $[-\infty,+\infty]^d$.
\end{fact}
A proof of this can be found in Dellacherie \& Meyer~\cite[Thm.~29, p.~157]{dm-1} or in Doob~\cite[Thm.~II.2.4]{doob53}.
Note that the modification is, in general, not $\rd$-valued; moreover, we have no control on what the separating set $D$ will be.
\begin{fact}\label{gen-03}\index{separable process!sufficient condition}
    Let $(X_t)_{t\geq 0}$ be a stochastic process with values in $\rd$ which is defined on a complete probability space.
    \begin{enumerate}\itemsep5pt
    \item[\upshape (a)]\label{gen-03-a}
        If $(X_t)_{t\geq 0}$ is continuous in probability, i.e.\ $\lim_{s\to t}\Pp\left(|X_s - X_t| > \epsilon\right)=0$ for all $\epsilon>0$ and $t\geq 0$, then there is a separable modification taking values in $[-\infty,\infty]^d$.
    \item[\upshape (b)]\label{gen-03-b}
        If $\rho(s,t)\coloneqq  \sqrt{\Ee\left[ |X_t-X_s|^2\right]}$ is finite for all $s,t\geq 0$ and the pseudo metric space $([0,\infty),\rho)$ is separable, then there is a separable modification taking values in $\rd$.
    \end{enumerate}
    Moreover, if \textup{(a)} or \textup{(b)} holds, every countable dense set $D\subseteq [0,\infty)$ can be used as a separating set.
\end{fact}
We refer to Wentzell~\cite[Satz 3, p.~79]{wentzell79} and Lifshits~\cite[Prop.~1, p.~26]{lifshits95} for the simple proof.

\paragraph{Comparison of \boldmath $|X_t|, X_t^*$ and $X_t, M_t\coloneqq  \sup_{s\leq t} X_s$.\unboldmath}
Frequently we state a theorem only for $|X_t|$ and $X_t^*$ or for $X_t$ and $M_t$. The following argument gives a useful criterion that compares $\Ee\left[X_t^*\right]$ and $\Ee\left[M_t\right]$.
\begin{lemma}\label{gen-05}
    Let $(X_t)_{t\geq 0}$ be a separable real-valued process such that $\Ee |X_t|<\infty$ for all $t \geq 0$. Then
    \begin{gather}\label{gen-e04}
        \Ee \left[X^*\right] < \infty
        \iff
        -\infty < \Ee\left[\inf_{t\geq 0}X_t\right] \leq \Ee\left[\sup_{t\geq 0}X_t\right] < \infty.
    \end{gather}
    If the process is symmetric, i.e.\ $X_t\sim -X_t$, then
    \begin{gather}\label{gen-e06}\begin{aligned}
        \Ee\left[\sup_{t\geq 0} |X_t|\right]
        &\leq 3 \Ee\left[\sup_{t\geq 0} X_t\right] + 2\inf_{t\geq 0} \Ee\bigl[|X_t|\bigr]\\
        &\leq 3 \Ee\left[\sup_{t\geq 0} X_t\right] + 2\inf_{t\geq 0} \sqrt{\Vv\left[X_t\right]}.
    \end{aligned}\end{gather}
\end{lemma}
\begin{proof}
Note that $|X_t| = 2X^+_t - X_t$ and
\begin{gather*}
    \sup_{t\geq 0} |X_t|
    \leq 2\sup_{t\geq 0} X_t^+ + \sup_{t\geq 0} (-X_t)
    = 2\sup_{t\geq 0} X_t^+ - \inf_{t\geq 0} X_t.
\end{gather*}
Moreover, $\sup_{t\geq 0} X_t^+ \leq \sup_{t\geq 0} X_t + \inf_{t\geq 0}|X_t|$. This clever estimate is taken from Lifshits~\cite[Prop.~10.2]{lifshits12}, and it follows from considering two cases: If $\sup_{t\geq 0} X_t \geq 0$, then $\sup_{t\geq 0} X_t = \sup_{t\geq 0} X_t^+$, while for $\sup_{t\geq 0} X_t <0$ the left hand side is zero. Together we have
\begin{gather*}
    \sup_{t\geq 0} |X_t|
    \leq 2\sup_{t\geq 0} X_t + 2\inf_{t\geq 0}|X_t| - \inf_{t\geq 0} X_t
    \leq 5\sup_{t\geq 0}|X_t|
\end{gather*}
from which the claim follows at once. Taking expectations, we get
\begin{gather*}
    \Ee\left[\sup_{t\geq 0} |X_t|\right]
    \leq 2 \Ee\left[\sup_{t\geq 0} X_t\right] + 2\inf_{t\geq 0} \Ee\bigl[|X_t|\bigr] - \Ee\left[\inf_{t\geq 0} X_t\right].
\end{gather*}
If the process is symmetric, then $-\Ee\left[\inf\limits_{t\geq 0} X_t\right] = \Ee\left[\sup\limits_{t\geq 0}(- X_t)\right] = \Ee\left[\sup\limits_{t\geq 0}X_t\right]$, and \eqref{gen-e06} follows.
\end{proof}

\paragraph{The Kolmogorov--Chentsov theorem.}
When talking about maximal inequalities, the Kolmogorov--Chentsov continuity theorem sounds like an unlikely candidate. However, the formulation pioneered by Totoki~\cite{totoki61} and widely used in the theory of stochastic flows, see e.g.\ Meyer~\cite[pp.~105--106]{meyer-flot} and Kunita~\cite[Appendix]{kunita04}, clearly shows the \enquote{maximal estimate} character of this result. Our statement follows the exposition in Revuz \& Yor~\cite[Thm.~I.(2.1)]{rev-yor99} and Schilling~\cite[Thm.~10.1]{schilling-bm}, and we refer to these monographs for the (by now standard) proof.

\begin{theorem}[Kolmogorov--Chentsov--Totoki]\label{gen-12} \index{theorem!Kolmogorov--Chentsov}
    Let $(\xi(x))_{x\in\rn}$ be a stochastic process taking values in $\rd$ and indexed by $x\in\rn$. If
    \begin{equation}\label{gen-e11}
        \forall x,y\in\rn \::\: \Ee\left[ |\xi(x)-\xi(y)|^\alpha\right] \leq c |x-y|^{n+\beta}
    \end{equation}
    holds for some constants $c>0$ and $\alpha, \beta>0$, then $(\xi(x))_{x\in\rn}$ has a modification with continuous sample paths. Moreover
    \begin{equation}\label{gen-e13}
        \Ee\left[\sup_{\substack{0<|x-y|<1\\ x,y\in [-T,T)^n}} \frac{|\xi(x)-\xi(y)|^\alpha}{|x-y|^\gamma}\right] < \infty
    \end{equation}
    for all $T>0$ and $0\leq \gamma < \beta$. In particular, $x\mapsto\xi(x)$ is a.s.\ locally H\"{o}lder continuous of order $\gamma/\alpha$.
\end{theorem}

Without real changes in the argument, we may in \eqref{gen-e13} replace $0<|x-y|<1$ by $0<|x-y|<\kappa$ for some $\kappa<T$. Therefore, we may read \eqref{gen-e13} as
\begin{gather*}
        \Ee\left[\sup_{\substack{0<|x-y|<\kappa\\ x,y\in [-T,T)^n}} |\xi(x)-\xi(y)|^\alpha\right] \leq C_T \kappa^\gamma,\quad 0<\gamma<\beta.
\end{gather*}
Sometimes, a further variant is useful. Recall that $\ball{\kappa}{z}$ is an open ball with centre $z$ and radius $\kappa$. Since $\{x,y\in [-T,T)^n \mid x,y\in\ball{\kappa}{z}, z\in [-T,T)^n\}\subseteq \{x,y\in [-T,T)^n \mid |x-y| < 2\kappa\}$, we also have
\begin{gather*}
        \Ee\left[\sup_{z\in [-T,T)^n} \sup_{\substack{x,y\in\ball{\kappa}{z}\cap [-T,T)^n}} |\xi(x)-\xi(y)|^\alpha\right] \leq C_T \kappa^\gamma,\quad 0<\gamma<\beta.
\end{gather*}

If we replace $n \rightsquigarrow 1+n$ and $x\rightsquigarrow (i,x)\in \real\times\rn$, the above theorem also shows that the estimate $\Ee\left[ |\xi(i,x)-\xi(i,y)|^\alpha\right] \leq c |x-y|^{1+n+\beta}$ for all $(i,x),(i,y)\in\real\times\rn$ implies
\begin{gather*}
    \Ee\left[\sup_{\substack{i\in [-T,T),\:0<|x-y|<\kappa\\ x,y\in [-T,T)^n}} |\xi(i,x)-\xi(i,y)|^\alpha\right] \leq C_T \kappa^\gamma,\quad \gamma<\beta.
\end{gather*}
The supremum over the extra parameter $i$ \enquote{costs us} one extra power in the estimate \eqref{gen-e11}.

There are several refinements of Theorem~\ref{gen-12}, which use methods that were discovered when studying Gaussian processes. A good survey is Pisier~\cite{pisier80} and the paper Schilling~\cite{schilling00} explains the connection to the Sobolev embedding theorem. The following result is due to Pisier~\cite[Thm.~2.1]{pisier80}, we follow the strategy worked out in Lifshits~\cite[Thm.~10.1, Exercise~10.1]{lifshits12}. For the proof we need the notion of a covering number \index{covering number} (see also Definition~\ref{gau-31} further down): Let $(\Tt,d)$ be a compact (pseudo-)metric space. The covering number is defined as
\begin{gather}\label{gen-e15-a}
    N(\epsilon) \coloneqq  \min\left\{n \mid \Tt \subseteq {\textstyle \bigcup_{k=1}^n} A_{k},\; \diam(A_i)\leq\epsilon \right\},
    \quad \epsilon > 0.
\end{gather}

\begin{lemma}[Pisier]\label{gen-15} \index{theorem!Pisier}
    Let $(X_t)_{t\geq 0}$ be a separable stochastic process taking values in $\rd$ and such that $m_t^p \coloneqq  \sup_{s\leq t} \Ee\left[|X_s|^p\right] < \infty$ for some $p\in (1,\infty)$ and all $t>0$. Let $\rho(r,s) \coloneqq  \|X_r - X_s\|_{L^p}$ and denote by $N(\epsilon)$ the covering number induced by $\rho$ on every compact set $[0,t]$. Then
    \begin{gather}\label{gen-e16}
        \Ee\left[X_t^*\right] \leq 4 \int_0^{m_t/2} N(\epsilon)^{1/p}\,d\epsilon.
    \end{gather}
\end{lemma}
\begin{proof}
    We set up a chaining \index{chaining} scheme.\footnote{An abstract discussion of \enquote{generic} chaining schemes is in Talagrand~\cite{talagrand14}.} Fix $t>0$, set $\epsilon_i \coloneqq  m_t 2^{-i}$, $i\in\nat_0$, and cover $[0,t]$ by a minimal covering with sets of $\rho$-diameter less than $\epsilon_i$. By the definition of the covering number, the covering contains $N(\epsilon_i)$ sets. Take in each covering set a point $t_k = t_k(i)$ and set $T_i = \{t_1,\dots,t_{N(\epsilon_i)}\}$. Now we construct a map
    \begin{gather*}
        \pi_i\colon [0,t] \to T_i
        \quad\text{satisfying}\quad
        \forall s\in [0,t] \::\: \rho(s,\pi_i(s)) = \|X_s-X_{\pi_i(s)}\|_{L^p} \leq \epsilon_i .
    \end{gather*}
    Clearly,
    \begin{gather*}
        \sup_{s\in T_i}|X_s|
        \leq \sup_{s\in T_{i-1}}|X_s| + \sup_{s\in T_i}|X_s - X_{\pi_{i-1}(s)}|.
    \end{gather*}
    Taking expectations, we see that
    \begin{align*}
        \Ee\left[ \sup_{s\in T_i}|X_s| \right]
        &\leq \Ee\biggl[ \sup_{s\in T_{i-1}}|X_s| \biggr] +  \left(\Ee\left[ \sup_{s\in T_i}\left|X_s - X_{\pi_{i-1}(s)}\right|^p \right]\right)^{1/p}\\
        &\leq \Ee\biggl[ \sup_{s\in T_{i-1}}|X_s| \biggr] +  \biggl({\textstyle\sum\limits_{s\in T_i}} \Ee\left[\left|X_s - X_{\pi_{i-1}(s)}\right|^p \right]\biggr)^{1/p}\\
        &\leq \Ee\biggl[ \sup_{s\in T_{i-1}}|X_s| \biggr] +  N(\epsilon_i)^{1/p} \epsilon_{i-1}.
    \end{align*}
    If we iterate this argument, observe that $N(\epsilon)$ is decreasing and that $\epsilon_{i}-\epsilon_{i+1} = \epsilon_{i+1} = \frac 12\epsilon_{i}$, then we get
    \begin{align*}
        \Ee\left[ \sup_{s\in T_n}|X_s| \right]
        &\leq \sum_{i=1}^n N(\epsilon_i)^{1/p} \epsilon_{i-1}
        = 4\sum_{i=1}^n \int_{\epsilon_{i+1}}^{\epsilon_{i}} N(\epsilon_i)^{1/p}\,d\epsilon\\
        &\leq 4\sum_{i=1}^n \int_{\epsilon_{i+1}}^{\epsilon_{i}} N(\epsilon)^{1/p}\,d\epsilon
        \leq 4\int_0^{m_t/2} N(\epsilon)^{1/p}\,d\epsilon.
    \end{align*}
    If we apply the above argument to the process $(X_s)_{s\in D_n\cap[0,t]}$ defined on the finite set $D_n = \{k 2^{-n} \mid k\in\nat_0\}\cup\{t\}$, we get the estimate \eqref{gen-e16} first for $\Ee \left[\max_{s\in D_n\cap[0,t]}|X_s|\right]$; note that $D_n$ is reached by some $T_i$ after finitely many iterations. Since the right-hand side of \eqref{gen-e16} does not depend on $D_n$, we can use monotone convergence to bound $\Ee \left[\sup_{s\in D\cap[0,t]}|X_s|\right]$ for $D = \bigcup_n D_n$. Finally, by separability, $X^*_t  = \sup_{s\in D\cap[0,t]}|X_s|$, finishing the proof.
\end{proof}

The proof also works if we use a further metric $d(s,t)$ on $[0,\infty)$ such that $\rho(s,t)\leq c d(s,t)$. If $d(s,t)=d(0,t-s)$ is invariant under translations, the $d$-covering number is $N(\epsilon)\approx d(0,\epsilon)^{-1}$. If, in particular, $\|X_s-X_t\|_{L^p}\leq c|t-s|^\alpha$ for some $\alpha\leq 1$, we get
\begin{gather*}
    \Ee\left[X_t^*\right] \leq c\int_0^{m_t} \epsilon^{-\alpha/p}\,d\epsilon
    = c_{\alpha,p} m_t^{1-\alpha/p}
    = c_{\alpha,p} \left(\sup_{s\leq t}\left(\Ee|X_s|^p\right)^{1/p}\right)^{1-\alpha/p}.
\end{gather*}
This is essentially a refinement of the Kolmogorov--Chentsov theorem.

A systematic exposition of processes $(X_t)_{t\geq 0}$ satisfying an \enquote{increment condition}, i.e.\ an inequality of the form $\Pp\left(|X_t-X_s|>\lambda\right)\leq\exp\left[-\lambda^2/2d^2(s,t)\right]$ where $d(s,t)$ is some metric on the parameter space of the process $(X_t)_{t\geq 0}$, can be found in Talagrand~\cite[Ch.~2]{talagrand14}. For example, for a Gaussian process we have $d(s,t)=\|X_s-X_t\|_{L^2}$, see Chapter~\ref{gau}.

\paragraph{A maximal inequality for processes with controlled increments.}
The following theorem was found by Liu \& Xi~\cite{liu-xi20}, without making the connection to the Kolmogorov--Chentsov--Totoki theorem. Let $(X_t,\Fscr_t)_{t\in [0,T]}$ be an adapted $\rd$-valued stochastic process satisfying a \emph{conditional increment control} \index{conditionally controlled increments} with parameters $(p,h)$, $p>1$, $h\in (0,1)$ and $ph > 1$, if $\Ee\left[|X_t|^p\right] < \infty$ and if there is a constant $A_{p,h}$ such that
\begin{gather}\label{gen-e15}
    \forall 0\leq s < t\leq T \::\quad
    \Ee\left[ \left|\Ee(X_t\mid\Fscr_s) - X_s\right|^p\right] \leq A_{p,h}|t-s|^{ph}.
\end{gather}

\begin{theorem}\label{gen-17}\index{inequality!controlled jumps}
    Let $(X_t,\Fscr_t)_{t\in [0,T]}$ be an $\rd$-valued c\`adl\`ag process with conditionally controlled increments with parameters $(p,h)$ such that $p>1$, $h\in (0,1)$ and $ph> 1$. Define for $[u,v]\subseteq [0,T]$ the maximal process $X^*_{u,v}\coloneqq \sup_{u\leq s\leq v}|X_s|$. There is a constant $C_{p,h,\epsilon}$ such that for any $0<\epsilon < ph-1$
    \begin{gather*}
        \Ee\left[\left(X_{u,v}^*\right)^p\right]
        \leq \frac{c_p p^p}{(p-1)^p}\left(C_{p,h,\epsilon} A_{p,h}|u-v|^{ph- 1-\epsilon} + \Ee\left[|X_v|^p\right]\right).
    \end{gather*}
\end{theorem}
The paper Liu \& Xi~\cite{liu-xi20} shows the result of Theorem~\ref{gen-17} with $|u-v|^{ph}$ instead of $|u-v|^{ph-1-\epsilon}$, but the proof is much more complicated. The key is an improved version of Lemma~\ref{gen-19}, which is not directly deduced from Theorem~\ref{gen-12}, but the authors give a more refined version of the argument used in the proof of Theorem~\ref{gen-12}. If one is interested in large values of $p$, our shortcut is not too much of a loss. The following two auxiliary results are simplified versions of results in Liu \& Xi~\cite{liu-xi20}.

\begin{lemma}\label{gen-19}
    Let $(X_t)_{t\in [0,T]}$ be a process as in Theorem~\ref{gen-17}. There is a constant $C_{p,h}$ such that for all $[u,v]\subseteq [0,T]$ and $t\in [u,v]$ and any $0<\epsilon < ph-1$
    \begin{gather*}
        \Ee\left[\sup_{u\leq s\leq v} \left|\Ee\left(X_t\mid\Fscr_s\right) - X_s\right|^p\right]
        \leq C_{p,h,\epsilon} A_{p,h}|u-v|^{ph- 1-\epsilon}.
    \end{gather*}
\end{lemma}
\begin{proof}
    Pick $0\leq u<v\leq T$ and consider the process $Y_t \coloneqq  \Ee\left(X_t\mid\Fscr_u\right)$ for $u\leq s<t\leq v$. By the tower property and the conditional Jensen inequality we have
    \begin{align*}
        \Ee\left[\left|Y_t-Y_s\right|^p\right]
        &= \Ee\left[\left|\Ee\left(X_t-X_s \mid \Fscr_u\right)\right|^p\right]\\
        &= \Ee\left[\left|\Ee\left( \Ee(X_t\mid\Fscr_s)-X_s \mid \Fscr_u\right)\right|^p\right]\\
        &\leq \Ee\left[\left|\Ee\left(X_t\mid\Fscr_s\right) - X_s\right|^p\right]\\
        &\leq A_{p,h}|t-s|^{ph}.
    \end{align*}
    We can now apply Theorem~\ref{gen-12}, cf.\ see the discussion following the statement of that theorem, to conclude
    \begin{align*}
        \Ee\left[\sup_{u\in [0,T]}\sup_{t\in [u, v]} \left|\Ee(X_t\mid\Fscr_u) - X_u\right|^p\right]
        &= \Ee\left[\sup_{u\in [0,T]}\sup_{t\in [u, v]} \left|Y_t-Y_u\right|^p\right]\\
        &\leq \Ee\left[\sup_{u\in [0,T]}\sup_{u\leq s < t\leq v} \left|Y_t-Y_s\right|^p\right]\\
        &\leq C_{p,h} |u-v|^{ph- 1-\epsilon}
    \end{align*}
    from which the claim follows.
\end{proof}

\begin{lemma}\label{gen-21}
    Let $(X_t,\Fscr_t)_{t\in [0,T]}$ be an adapted c\`adl\`ag process such that $\Ee|X_t|<\infty$ for all $t\in [0,T]$. For all $[u,v]\subseteq [0,T]$, $X_{u,v}^* \coloneqq  \sup_{u\leq s\leq v}|X_s|$ and $\lambda > 0$
    \begin{gather}\label{gen-e22}
        \Pp\left(X_{u,v}^* > \lambda\right)
        \leq \frac 1\lambda \int_{X_{u,v}^* > \lambda} \left( \sup_{s\in [u,v]} \left|\Ee\left(X_v\mid\Fscr_s\right) - X_s\right| + |X_v|\right)\,d\Pp.
    \end{gather}
\end{lemma}
\begin{proof}
    Fix $\lambda>0$ and define a stopping time $\tau = \inf\left\{s\in [u,v] \mid |X_s|>\lambda\right\}\wedge T$ with $\inf\emptyset \coloneqq  \infty$. Since the process has c\`adl\`ag paths,
    \begin{gather*}
        \left\{\tau < T\right\} = \left\{u\leq\tau\leq v\right\} \supseteq \left\{X_{u,v}^*>\lambda\right\},
    \end{gather*}
    and $|X_\tau| > \lambda$ on the set $\{u\leq\tau\leq v\}$. Thus,
    \begin{equation*}\begin{aligned}[b]
        \Pp\left(X_{u,v}^* > \lambda\right)
        &\leq \frac 1\lambda \int_{u\leq\tau\leq v} |X_\tau|\,d\Pp\\
        &\leq \frac 1\lambda \int_{u\leq\tau\leq v} \Big[\left| \Ee(X_v\mid\Fscr_\tau) - X_\tau\right| + \left| \Ee(X_v\mid\Fscr_\tau) \right| \Big]\, d\Pp\\
        &\leq \frac 1\lambda \int_{u\leq\tau\leq v} \Big[ \sup_{s\in [u,v]} \left|\Ee(X_v\mid\Fscr_s) - X_s\right| + \Ee(|X_v|\mid\Fscr_\tau) \Big]\, d\Pp\\
        &\leq \frac 1\lambda \int_{u\leq\tau\leq v} \Big[ \sup_{s\in [u,v]} \left|\Ee(X_v\mid\Fscr_s) - X_s\right| + \left|X_v\right|\Big]\, d\Pp.
    \end{aligned}\qedhere\end{equation*}
\end{proof}

\begin{proof}[Proof of Theorem~\ref{gen-17}]
    Lemma~\ref{gen-21} and Lemma~\ref{pre-12} show for $p>1$ with $ph> 1$
    \begin{align*}
        \Ee\left[ \left(X_{u,v}^*\right)^p\right]
        &\leq c_p\left(\frac{p}{p-1}\right)^p \Ee\left[ \sup_{s\in [u,v]} \left|\Ee(X_v\mid\Fscr_s) - X_s\right|^p + \left|X_v\right|^p\right].
    \end{align*}
    From this the claim follows with the estimate in Lemma~\ref{gen-19}.
\end{proof}

\begin{remark}\label{gen-23}
    The appearance of the conditional increment control \eqref{gen-e15} in connection with the Kolmogorov-Chentsov-Totoki theorem is not too surprising. A similar condition appears in tightness and compactness theorems for the weak convergence of stochastic processes, see
    Ethier \& Kurtz~\cite[p.~138, Rem.~3.8.7]{ethier-kurtz}.
\end{remark}

Another type  of conditional domination can be used to get exponential moments for c\`adl\`ag processes with bounded jumps $\Delta X_t \coloneqq  X_t - X_{t-}$. The clever proof of the following lemma is due to Bass~\cite[Thm.~I.(6.11)]{bass88}.
\begin{lemma}\label{gen-31} \index{inequality!controlled jumps}
    Let $(X_t,\Fscr_t)_{t\in [0,\infty]}$ be an adapted c\`adl\`ag process with bounded jumps, i.e.\ $\sup_{t>0}|\Delta X_t| \leq c$ for some constant $c>0$. If there is some constant $C>0$ such that
    \begin{gather}\label{gen-e32}
        \Ee\left( |X_\infty - X_{\tau-}| \mid \Fscr_{\tau}\right) \leq C
        \quad\text{for all stopping times $\tau$,}
    \end{gather}
    then there exists an exponent $p>0$ such that $\Ee\left[\exp(p X^*)\right]<\infty$.
\end{lemma}
\begin{proof}
    There is no loss of generality in assuming $c=C=1$, as we can always consider the scaled process $(c+C)^{-1}X_t$. Let $\tau$ be any stopping time, $\lambda>0$ and $F\in\Fscr_\tau$. By the Markov inequality we get
    \begin{align*}
        \Pp\left(\left\{|X_\infty - X_{\tau-}| > \lambda\right\}\cap F\right)
        &\leq \frac 1\lambda \int_F |X_\infty - X_{\tau-}|\,d\Pp\\
        &= \frac 1\lambda \int_F \underbracket[.6pt]{\Ee\left(|X_\infty - X_{\tau-}|\mid\Fscr_\tau\right)}_{\leq 1} d\Pp
        \leq \frac{1}{\lambda}\Pp(F).
    \end{align*}
    We will now bound the tail of $X^*$. For this, we introduce stopping times
    \begin{gather*}
        \sigma \coloneqq  \inf\left\{t>0 \mid |X_t|\geq \lambda + 5\right\}
        \et
        \tau \coloneqq  \inf\left\{t>0 \mid |X_t|\geq \lambda + 15\right\},
    \end{gather*}
    and we use in the following calculation that $|X_{\tau-}|\geq \lambda + 14$ and $|X_{\sigma-}|\leq \lambda + 5$, since $t\mapsto X_t$ is c\`adl\`ag with jumps of size less than one. For any $\lambda>0$
    \begin{align*}
        \Pp&\left(X^* > \lambda + 20\right)\\
        &= \Pp\left(X^* > \lambda + 20,\: |X_\infty|\leq \lambda + 10 \right) + \Pp\left(|X_\infty|> \lambda + 10 \right)\\
        &\leq \Pp\left(\tau<\infty,\: |X_\infty|\leq \lambda + 10 \right) + \Pp\left(\sigma < \infty,\:|X_\infty|> \lambda + 10 \right)\\
        &\leq \Pp\left(\tau<\infty,\: |X_\infty - X_{\tau-}| > 4\right) + \Pp\left(\sigma < \infty,\:|X_\infty-X_{\sigma-}|> 5 \right)\\
        &\leq \frac 14 \Pp(\tau<\infty) + \frac 15\Pp(\sigma<\infty).
    \end{align*}
    In the last inequality we use the estimate from the beginning of the proof. Since $\{\sigma<\infty\}\cup\{\tau<\infty\}\subseteq\{X^*>\lambda\}$, we get
    \begin{gather*}
        \Pp\left(X^* > \lambda + 20\right) \leq \frac 12 \Pp\left(X^* > \lambda\right),
    \end{gather*}
    and if we iterate this inequality, taking $\lambda = 20 n$, $n\in\nat_0$, we end up with $\Pp\left(X^* > 20 n\right)\leq 2^{-n}$. Now we can apply the layer cake formula and conclude
    \begin{equation*}
        \Ee\left[ e^{pX^*}\right]
        = \sum_{n=0}^\infty \int_{20 n}^{20(n+1)} p e^{p\lambda}\Pp\left(X^* > \lambda\right)\,d\lambda
        \leq p e^{20 p}\sum_{n=0}^\infty e^{20 pn}2^{-n},
    \end{equation*}
    which is finite if $p$ is small enough.
\end{proof}

Before we give an application of Lemma~\ref{gen-31} to $\mathrm{BMO}$ (bounded mean oscillation) martingales, we want to record the analogue of Lemma~\ref{gen-31} for increasing processes. This result is due to Dellacherie~\cite{dellacherie79} and Lenglart \emph{et al.}~\cite{len-lep-pra80} who, in turn, credit Garsia and Neveu for the discrete-time case.
\begin{lemma}\label{gen-33}
    Let $(A_t,\Fscr_t)_{t\geq 0}$ be an adapted, increasing c\`adl\`ag process such that $A_0=0$, $A_\infty \coloneqq  \lim_{t\to\infty} A_t$, $\Ee A_\infty < \infty$ and
    \begin{gather}\label{gen-e34}
        \Ee\left[A_\infty - A_{\tau-} \mid \Fscr_\tau \right] \leq X\I_{\{\tau<\infty\}}
    \end{gather}
    for some integrable random variable $X\geq 0$ and all stopping times $\tau$. Then, for every positive convex function $F$, whose right derivative is denoted by $f = F'_+$, one has\footnote{There is a misprint in Lenglart \emph{et al.}~\cite[Lem.~1.2.a)]{len-lep-pra80}: $\Ee[A_\infty f(X)]$ should read $\Ee[f(A_\infty) X]$}
    \begin{gather*}
        \Ee\left[F(A_\infty)\right] \leq \Ee\left[X f(A_\infty)\right] + F(0).
    \end{gather*}
\end{lemma}
If we use the lemma for $X\equiv c$, $\kappa \wedge A_t$ (note that $\kappa\wedge A_\infty - \kappa \wedge A_t \leq A_\infty-A_t$ since $A_t\leq A_\infty$), $F(x) = e^{px}$ and $f(x) = p e^{px}$ for some $p<1/c$, we get
\begin{gather*}
    \Ee\left[\exp(p \kappa\wedge A_\infty)\right]
    \leq pc \Ee\left[\exp(p \kappa\wedge A_\infty)\right] + 1,
\intertext{hence}
    \Ee\left[\exp(p \kappa\wedge A_\infty)\right]
    \leq (1-pc)^{-1};
\end{gather*}
now we use Fatou's lemma to let $\kappa\to\infty$. This gives $\Ee\left[\exp(p A_\infty)\right]\leq (1-pc)^{-1}$ making the connection to Lemma~\ref{gen-31}.

\begin{proof}
    We multiply both sides of \eqref{gen-e34} with $\I_{\{\tau<\infty\}}$ and take expectations to get $\Ee\left[(A_\infty- A_{\tau-})\right] \leq \Ee\left[X\I_{\tau<\infty}\right]$ -- here we use that $A_\infty = A_{\infty-}$. Fix $\lambda>0$ and use the stopping time $\tau = \inf\{t>0\mid A_t \geq \lambda\}$. Since $A_{\tau-}\leq\lambda$ and $\{A_\infty\geq\lambda\}\supseteq\{\tau<\infty\}$
    \begin{gather*}
        \Ee\left[(A_\infty - \lambda)\I_{\{A_\infty\geq\lambda\}}\right] \leq \Ee\left[A_\infty - A_{\tau-}\right] \leq \Ee\left[X\I_{\{\tau<\infty\}}\right]
        \leq\Ee\left[X\I_{\{A_\infty\geq\lambda\}}\right].
    \end{gather*}
    Now we can integrate both sides with respect to $df(\lambda)$ and get
    \begin{gather*}
        \int_0^{A_\infty}(A_\infty-\lambda)\,df(\lambda)
        = \int_0^{A_\infty}\int_{\lambda}^{A_\infty}\,dx\,df(\lambda)
        = \int_0^{A_\infty}\int_{0}^{x}df(\lambda)\,dx\\
        = \int_0^{A_\infty} (f(x)-f(0))\,dx
        = F(A_\infty) - F(0) - f(0)A_\infty
    \end{gather*}
    on the left, and $X(f(A_\infty)-f(0))$ on the right. Thus,
    \begin{equation*}
        \Ee\left[F(A_\infty)\right] - F(0)
        \leq \Ee\left[(A_\infty-X)f(0) + Xf(A_\infty)\right]
        \leq \Ee\left[Xf(A_\infty)\right].
    \qedhere
    \end{equation*}
\end{proof}

We will finally apply Lemma~\ref{gen-31} to derive the John--Nirenberg inequality for $\mathrm{BMO}$ martingales. For details on martingales we refer to the following Chapter~\ref{mar}. A martingale $(X_t,\Fscr_t)_{t\geq 0}$, which is $L^2$-bounded, i.e.\ $\sup_{t>0} \Ee[X_t^2] < \infty$, and has continuous paths, is of \emph{class $\mathrm{BMO}$} (\emph{bounded mean oscillation}), if there exists a constant $c>0$ such that
\begin{gather}\label{gen-e36} \index{theorem!John--Nirenberg} \index{inequality!John--Nirenberg} \index{BMO}
    \Ee\left[ (X_\infty - X_\tau)^2\mid \Fscr_\tau\right]\leq c^2\quad\text{for all stopping times $\tau$}.
\end{gather}
We define the norm $\|X\|_{\mathrm{BMO}}$ as the infimum of all constants $c$ in \eqref{gen-e36}.

\begin{theorem}[John, Nirenberg]\label{gen-35}
    Let $(X_t)_{t\geq 0}$  be a \textup{(}continuous\textup{)} martingale of class $\mathrm{BMO}$ such that $0<\|X\|_{\mathrm{BMO}}<\infty$. There exists a constant $p>0$ such that
    \begin{gather*}
        \Ee\left[\exp\left(p X^*/\|X\|_{\mathrm{BMO}}\right)\right] < \infty.
    \end{gather*}
\end{theorem}
\begin{proof}
    Since $(X_t)_{t\geq 0}$ is $L^2$-bounded, it is uniformly integrable and so $X_\infty=\lim_{t\to\infty} X_t$ and $X^*$ exist in $L^2(\Pp)$.
    Since $t\mapsto X_t$ is continuous, the martingale has no jumps. Moreover, by the conditional Jensen inequality, we see for all stopping times
    \begin{gather*}
        \Ee\left[\left|X_\infty-X_\tau\right| \mid \Fscr_\tau\right]
        \leq \sqrt{\Ee\left[\left|X_\infty-X_\tau\right|^2 \mid \Fscr_\tau\right]}
        \leq \|X\|_{\mathrm{BMO}},
    \end{gather*}
    and the claim follows from Lemma~\ref{gen-31}.
\end{proof}

    A systematic study of (weighted) $\mathrm{BMO}$-spaces for stochastic processes is in Geiss~\cite[Sec.~2]{geiss05}, see also Geiss~\cite{geiss97} for martingale difference $\mathrm{BMO}$-spaces.

\begin{remark}\label{gen-37}\index{BMO}
    In (harmonic) analysis a measurable function $f\colon \rd\to\real$ is said to be in $\mathrm{BMO}$ if
    \begin{gather}\label{gen-e38}
        \sup_{Q} \frac{1}{|Q|} \int_Q |f(x)-f_Q|\,dx \leq c < \infty
    \end{gather}
    where the supremum ranges over all axis-parallel cubes $Q\subseteq\rd$, $|Q|$ is the Lebesgue measure of $Q$ and $f_Q = |Q|^{-1} \int_Q f(y)\,dy$ is the mean over $Q$. The $\mathrm{BMO}$-norm $\|f\|_{\mathrm{BMO}}$ is the infimum of all constants $c$ appearing in the inequality \eqref{gen-e38}. We will again encounter a similar kind of \enquote{dyadic maximal function} in Example~\ref{mar-15}.

    The original analyst's version of the John--Nirenberg inequality reads as follows: \textit{There is some constant $\alpha >0$ such that for each axis-parallel cube $Q\subseteq\rd$ the following inequality holds:
    \begin{gather*}
        \left|\left\{ x\in Q \mid {|f(x)-f_Q|} > \lambda \right\}\right|
        \leq |Q| e^{-\alpha\lambda/\|f\|_{\mathrm{BMO}}},\quad \lambda>0.
    \end{gather*}}
    A martingale proof can be found in Bass~\cite[Prop.~IV.(7.6)]{bass95}, the standard proof is in Torchinsky~\cite[Ch.~VIII, Thm.~1.3]{torchinsky86}.

    Let us briefly mention the origins of this class. In analysis, the $L^p$ scale is often used but the end points of this scale, the spaces $L^1$ and $L^\infty$, are special since they are not reflexive. In many problems in analysis, therefore, the space $L^1$ is replaced by a slightly smaller space such as the $L\log L$ class in $\rd$ (see also \S~\ref{mar-doo}) or the real Hardy space $H^1(\rd)$. For definitions and details we refer to the monograph by Torchinsky~\cite[Ch.~XIV, XV]{torchinsky86}. It turned out -- and even led to a Fields medal for Charles Fefferman -- that the space $\mathrm{BMO}(\rd)$ is the topological dual of $H^1(\rd)$. Of course, there are many more useful applications of $\mathrm{BMO}$ in harmonic analysis and the study of partial differential equations.
\end{remark}

\section{Martingales}\label{mar}

For most mathematicians, the first encounter with maximal inequalities in probability theory happens in connection with martingales. Here we will consider martingales indexed with $I = \nat_0$ or $I = [0,\infty)$. Recall that the underlying filtered probability space $(\Omega,\Ascr,\Pp,\Fscr_t)$ satisfies the usual conditions, i.e.\ $(\Omega,\Ascr,\Pp)$ is complete, $\Fscr_0$ contains all $\Pp$-null sets and $(\Fscr_t)_{t\geq 0}$ is right-continuous (if $I=[0,\infty)$).
\begin{definition}\label{mar-03} \index{martingale}
    Let $X=(X_t,\Fscr_t)_{t\in I}$ be a real-valued stochastic process; if $I=[0,\infty)$, we assume, in addition, that $t\mapsto X_t(\omega)$ is c\`adl\`ag.
    The process $X$ is said to be a \emph{martingale}, if
    \begin{enumerate}
    \item\label{mar-03-a} for every $t\in I$, the random variable $X_t$ is $\Fscr_t$ measurable;
    \item\label{mar-03-b} $\Ee|X_t|<\infty$ for every $t\in I$;
    \item\label{mar-03-c} $\int_F X_s\,d\Pp = \int_F X_t\,d\Pp$ for all $s,t\in I$  with $s<t$ and all $F\in\Fscr_s$;
    \end{enumerate}
    If we have \enquote{$\leq$} or \enquote{$\geq$} in Property \ref{mar-03-c}, then we speak of a sub- or super-martingale, respectively.
\end{definition}

\begin{remark}\label{mar-05}
a)
    Property~\ref{mar-03-c} in Definition~\ref{mar-03} can be equivalently expressed as $\Ee\left(X_t\mid\Fscr_s\right) = X_s$ using the conditional expectation $\Ee(\cdots \mid \Fscr_s)$. The present formulation has the advantage, that it naturally extends the notion of martingales to general $\sigma$-additive measure spaces -- which is useful for many applications of martingales in real analysis, see e.g.\ Stein~\cite[Ch.~IV]{stein70b}. Note that most martingale results, in particular convergence theorems and inequalities, remain valid in the $\sigma$-additive setting, see Schilling~\cite[Ch.~22--25, 27]{schilling-mims} for a systematic presentation. An interesting side effect is that this also allows one to treat martingales without formally introducing conditional expectations right from the beginning.

\smallskip b)
    If $t\mapsto X_t$ is not c\`adl\`ag but satisfies the Properties~\ref{mar-03-a}--\ref{mar-03-c} from Definition~\ref{mar-03}, it is always possible to find a c\`adl\`ag modification; this is due to the \enquote{usual conditions} for the filtered probability space $(\Omega,\Ascr,\Pp,\Fscr_t)$. For sub-martingales one needs, in addition, that $t\mapsto \Ee X_t$ is right continuous. For a proof see Revuz \& Yor~\cite[Ch.~II.\S2]{rev-yor99} or Schilling~\cite[Thm.~46.11]{schilling-dc}.
\end{remark}

Since we focus on maximal inequalities, we assume the reader to be familiar with the standard martingale convergence results, see e.g.\ Revuz \& Yor~\cite{rev-yor99} or Schilling~\cite{schilling-dc}. Throughout this chapter, we use the following standard notation for maximal functions:
\begin{gather}\label{mar-e02}
    X_t^* \coloneqq  \sup_{s\leq t}|X_s|
    \et
    X^* \coloneqq  \sup_{t\in I} X_t^* = \sup_{t\in I}|X_t|.
\end{gather}

\subsection{Doob's Maximal Inequalities}\label{mar-doo}
The basis of most martingale maximal inequalities is the following lemma.
\begin{lemma}[Doob--Kolmogorov maximal estimate]\label{mar-07} \index{inequality!Doob's maximal}
    Let $(X_t,\Fscr_t)_{t\in I}$ be a submartingale. For every $t \in I$ and $\lambda> 0$
    \begin{gather}\label{mar-e04}
        \Pp\left(X_t^* > \lambda\right) \leq \frac 1\lambda \int_{\{X_t^* \geq \lambda\}} X_t^+\,d\Pp.
    \end{gather}
    If $(X_t,\Fscr_t)_{t\in I}$ is a martingale or a positive submartingale such that $X_t\in L^p(\Pp)$ for some $p\geq 1$, then
    \begin{gather}\label{mar-e06}
        \Pp\left(X_t^* > \lambda\right) \leq \frac 1{\lambda^p} \int_{\{X_t^* \geq \lambda\}} |X_t|^p\,d\Pp.
    \end{gather}
\end{lemma}
\begin{proof}
    Fix $t\in I$, $\lambda > 0$ and define the stopping time $\tau \coloneqq  \inf\{s\leq t \mid X_s\geq\lambda\}$; as usual, $\inf\emptyset \coloneqq  \infty$. Note that
    \begin{gather*}
        \left\{\tau < t\right\}
        \subseteq
        \left\{X_t^* > \lambda\right\}
        \subseteq
        \left\{\tau\leq t\right\}
        \subseteq
        \left\{X_t^* \geq \lambda\right\}.
    \end{gather*}
    These inclusions, along with the Markov inequality and the submartingale property for the stopped submartingale $(X_{t\wedge\tau}^{+})_{t\in I}$ yield
    \begin{align*}
        \Pp\left(X_t^* > \lambda\right)
        &\leq \Pp\left(\tau\leq t,\: X_{\tau\wedge t}^+ \geq\lambda\right)\\
        &\leq \frac 1\lambda \int_{\{\tau\leq t\}} X_{\tau\wedge t}^+\,d\Pp\\
        &\leq \frac 1\lambda \int_{\{\tau\leq t\}} X_{t}^+\,d\Pp\\
        &\leq \frac 1\lambda \int_{\{X^*\geq\lambda\}} X_{t}^+\,d\Pp.
    \end{align*}
    The estimate \eqref{mar-e06} follows from \eqref{mar-e04} since for a (positive sub-)martingale satisfying $\Ee\left[|X_t|^p\right]<\infty$ the process $(|X_t|^p,\Fscr_t)$ is a submartingale.
\end{proof}

Combining the estimates from Lemma~\ref{mar-07} with the good-$\lambda$ techniques from Section~\ref{pre-lambda} leads to Doob's maximal $L^p$ inequalities.
\begin{theorem}[Doob]\label{mar-09}\index{inequality!Doob's $L^p$}
    Let $(X_t,\Fscr_t)_{t\in I}$ be a martingale or a positive submartingale. Then
    \begin{gather}\label{mar-e08}
        \Ee\left[\left(X^*\right)^p\right]
        \leq
        \begin{dcases}
            \left(\frac p{p-1}\right)^p \sup_{t\in I} \Ee\left[|X_t|^p\right], &\text{if\ \ } p > 1,\\
            \frac{e}{e-1} \sup_{t\in I} \left(1+\Ee\left[|X_t| \log^+|X_t|\right]\right), &\text{if\ \ } p = 1,\\
            \frac 1{1-p} \sup_{t\in I} \left(\Ee\left|X_t\right|\right)^p, &\text{if\ \ } 0< p < 1.
        \end{dcases}
    \end{gather}
    If the terms on the right hand side of \eqref{mar-e08} are finite, $X^*\in L^p(\Pp)$ and the limit $X_\infty = \lim_{t\to\infty} X_t$ exists a.s.; if $p\geq 1$ the limit exists also in $L^p(\Pp)$ and $\sup_{t\in I} \Ee\left[|X_t|^p\right] = \Ee\left[|X_\infty|^p\right]$.
\end{theorem}
\begin{proof}
    The estimates in \eqref{mar-e08} follow if we combine Lemma~\ref{mar-07} and Lemma~\ref{pre-12} or Lemma~\ref{pre-13}, respectively, using $X=X_t^*$, $Y=X_t$, and taking the supremum over all $t\in I$.

    If the right hand side of \eqref{mar-e08} is finite, the family $(X_t)_{t\in I}$ is $L^1$-bounded, and the a.s.\ convergence $X_t\to X_\infty$ follows from the standard martingale convergence theorem. If $p\geq 1$, we also have uniform integrability, since $|X_t|^p\leq \left(X^*\right)^p\in L^1(\Pp)$, and we get $L^p$-convergence from Vitali's convergence theorem, see e.g.\
    Schilling~\cite[Thms.~22.9, 22.7]{schilling-mims}; in particular, $\sup_{t\in I}\Ee\left[|X_t|^p\right] = \Ee\left[|X_\infty|^p\right]$.
\end{proof}

Using Lemma~\ref{pre-13-2} we can obtain the following partial converse of the $L\log L$-case (i.e.\ for $p=1$) of Theorem~\ref{mar-09}. Originally, this assertion is due to Stein~\cite{stein69} in connection with the Hardy--Littlewood maximal function, see also Torchinsky
~\cite[Ch.~IV.5]{torchinsky86}; the probabilistic point of view is Gundy's~\cite{gundy69}, and our presentation follows Chersi~\cite{chersi70}.
\begin{theorem}[Stein; Gundy]\label{mar-11}
    Let $(X_t,\Fscr_t)_{t\in I}$ be a uniformly integrable martingale such that $X_t\geq 0$. Assume that there exist constants $a,b>0$ and a sequence $t_0<t_1<\dots < t_n <\dots < t_\infty = \infty$ such that
    \begin{gather*}
        X_0 \leq a
        \et
        X_{t_n} \leq b X_{t_{n-1}},\quad\text{a.s.\ and for all $n\in\nat$}.
    \end{gather*}
    If $\Ee\left[X^*\right] < \infty$, then $\Ee\left[X_\infty \log^+ X_\infty\right]<\infty$ for $X_\infty \coloneqq  \lim_{t\to\infty}X_t$.
\end{theorem}
\begin{proof}
    Since $(X_t)_{t\in I}$ is uniformly integrable, the limit $X_\infty = \lim_{t\to\infty}X_t$ exists a.s.\ and in $L^1(\Pp)$. Set $Y\coloneqq \sup_{n\in\nat} X_{t_n}\leq X^*$ and define stopping times
    \begin{gather*}
        \tau \coloneqq  t_{\nu},\quad
        \nu \coloneqq  \inf\left\{n \mid X_{t_n}>\lambda\right\},\quad \inf\emptyset = \infty.
    \end{gather*}
    By assumption,
    \begin{gather*}
        X_{t_\nu} \leq b X_{t_{\nu-1}} \leq b\lambda
    \end{gather*}
    for all $\lambda>a$, and this estimate still holds if $\nu=\infty$. Since $\left\{Y>\lambda\right\}=\left\{t_\nu<\infty\right\}$, we can use optional stopping to see for all $\lambda > a$
    \begin{align*}
        \int_{\{Y>\lambda\}} X_\infty\,d\Pp
        = \int_{\{t_\nu<\infty\}} X_{t_\nu}\,d\Pp
        \leq b\lambda\Pp\left(t_\nu<\infty\right)
        = b\lambda\Pp\left(Y>\lambda\right).
    \end{align*}
    From this the assertion follows with Lemma~\ref{pre-13-2}.
\end{proof}

There is an interesting variant of the Doob--Kolmogorov maximal estimate in Lemma~\ref{mar-07}. It is a continuous-time version of the following inequality due to Chow~\cite{chow60} and Birnbaum \& Marshall~\cite[Thm.~2.1]{bir-mar61}: Let  $(X_n)_{n\in\nat}$ be a positive submartingale and $(c_n)_{n\in\nat}$ a decreasing sequence. For every $\lambda\geq 0$ \index{inequality!Birnbaum--Marshall}
\begin{gather}\label{mar-e10}\begin{aligned}
    \lambda \Pp\left(\max_{1\leq k\leq n} c_k X_k > \lambda\right)
    &\leq \sum_{k=1}^{n-1} (c_k-c_{k+1}) \Ee \left[X_k\right] + c_n \Ee \left[X_n\right]\\
    &= c_1\Ee\left[X_1\right] + \sum_{k=2}^{n} c_k \left(\Ee \left[X_k\right] - \Ee \left[X_{k-1}\right]\right).
\end{aligned}\end{gather}
\begin{lemma}\label{mar-12}
    Let $(X_t,\Fscr_t)_{t\geq 0}$, $X_0=0$, be a positive submartingale with finite second moments $m(t)\coloneqq  \Ee\left[X_t^2\right]$, and $c\colon [0,\infty)\to[0,\infty)$ a right continuous, decreasing function. Then
    \begin{gather}\label{mar-e12}
        \Pp\left(\sup_{0\leq s\leq t} c(s)X_s > 1\right)
        \leq \int_0^t c^2(s)\,dm(s),\quad t > 0.
    \end{gather}
\end{lemma}
\begin{proof}
    Fix $t>0$ and define $t_i \coloneqq  t_{i,n} \coloneqq  (it)/n$ for $i=0,1,\dots,n$ and $n\in\nat$. By right continuity and with the inequality \eqref{mar-e10} applied to the submartingale $(X_{t_i}^2, \Fscr_{t_i})_{i=0,\dots, n}$ we get
    \begin{equation*}
    \mathclap{\begin{aligned}[b]
        \Pp\left(\sup_{0\leq s\leq t} c(s)X_s > 1\right)
        &= \lim_{n\to\infty} \Pp\left(\max_{i=1,\dots,n} c(t_i) X_{t_i} > 1\right)\\
        &\leq \lim_{n\to\infty} \sum_{i=1}^n c^2(t_i)\left(\Ee\left[X_{t_i}^2\right] - \Ee\left[\smash[b]{X_{t_{i-1}}^2}\right]\right)\\
        &= \int_0^t c^2(s)\,dm(s).
    \end{aligned}}\qedhere
    \end{equation*}
\end{proof}

We will close this first part with two applications of Doob's maximal inequalities.
\begin{example}[Kolmogorov's SLLN]\label{mar-13} \index{strong law of large numbers}
    Let us briefly sum up the standard proof of Kolmogorov's strong law of large numbers using martingales with an index set running to the left (\enquote{backwards or reverse martingale}), see e.g.\ Schilling~\cite[Ex.~24.8]{schilling-mims}
    or Schilling~\cite[Thms.~45.5, 45.6]{schilling-dc}:

    \medskip\noindent
    \textit{Let $(X_n)_{n\in\nat}$ be a sequence of identically distributed, integrable and independent \textup{(}or exchangeable\textup{)} random variables, then $S_n \coloneqq  (X_1+\dots+X_n)/n$ tends to $\Ee X_1$ a.s.\ and in $L^1$-sense.}

    \medskip\noindent
    The key observation in the argument is that
    \begin{gather*}
        M_{-n} \coloneqq  \frac{X_1+\dots+X_n}{n},\quad
        \Fscr_{-n} \coloneqq  \sigma\left(S_n, X_{n+1}, X_{n+2}, \dots\right)
    \end{gather*}
    is a  martingale with backwards directed index set. Such martingales are automatically uniformly integrable, and so the limit $M_{-n}\to M_{-\infty}$ exists a.s.\ and in $L^1$. In order to identify the limit as $M_{-\infty} = \Ee X_1$, one uses either Kolmogorov's zero-one law (in the iid case) or the Hewitt--Savage zero-one law (in the exchangeable case).

    \medskip\noindent
    There is also the usual converse: \textit{If a sequence of iid random variables is such that $\lim_{n\to\infty}\frac 1n(X_1+\dots+X_n)$ exists a.s., then $\Ee\left[|X_1|\right]<\infty$.}

    \medskip\noindent
    The standard proof -- along with a short argument for positive random variables -- can be found in
    Schilling~\cite[continuation of Ex.~24.8, pp.~296--7]{schilling-mims}. We are now interested in the following converse:

    \medskip\noindent
    \textit{If, in the above setting, $X_1\geq 0$ and $\Ee \left[M^*\right] = \Ee \left[\sup\limits_{n\in\nat} \frac 1n(X_1+\dots+X_n)\right] < \infty$, then $\Ee\left[X_1\log^+ X_1\right]< \infty$.}

    \medskip\noindent
    This follows at once from Theorem~\ref{mar-11} -- suitably modified for martingales with backwards directed index set: The roles of $X_0$ and $X_\infty$ in this theorem are now played by $M_{-\infty}=\Ee X_1$ and $M_{1}$, respectively. Moreover,
    \begin{gather*}
        M_{-n}
        = \frac{X_1+\dots+X_n}{n}
        \leq \frac{n+1}{n} \frac{X_1+\dots+X_n+X_{n+1}}{n+1}
        \leq 2 M_{-n-1}
    \end{gather*}
    shows that the condition of Theorem~\ref{mar-11} holds with $a = \Ee X_1>0$ and $b=2$. Thus, $\Ee\left[X_1\log^+ X_1\right]<\infty$.
\end{example}

\begin{example}[Hardy--Littlewood maximal functions]\label{mar-15} \index{maximal function!Hardy--Littlewood}
Write $C = [0,1)^d$ for the half-open unit cube in $\rd$, $\lambda$ for Lebesgue measure on $(C,\Bscr(C))$, and define the dyadic $\sigma$-algebras
\begin{gather*}
    \Fscr_k \coloneqq  \sigma\left(\Gscr_k\right),\quad
    \Gscr_k \coloneqq  \left\{(j_1,\dots,j_d) 2^{-k} + \left[0,2^{-k}\right)^d \mid \begin{aligned} &j_i = 0,\dots,2^k-1,\\ &i=1,\dots d\end{aligned}\right\}.
\end{gather*}
Clearly, $\Gscr_\infty \coloneqq  \bigcup_{k\in\nat} \Gscr_k$ generates the Borel sets $\Bscr(C)$. Each $f\in L^1(\lambda) = L^1(C,\lambda)$ gives rise to a uniformly integrable martingale by
\begin{gather*}
    f_k(x) \coloneqq  \Ee\left[ f\mid\Fscr_k\right](x) = \sum_{Q\in\Gscr_k} \frac{1}{\lambda(Q)}\int_Q f\,d\lambda \, \I_Q(x),
\end{gather*}
and we can define the \emph{dyadic maximal function} $f^{\diamond}$ \index{maximal function!dyadic} as
\begin{gather*}\label{square-max}
    f^\diamond(x)
    \coloneqq  \sup_{k\in\nat} \Ee\left[|f| \mid \Fscr_k\right](x)
    = \sup_{Q\in\Gscr_\infty,\: Q\ni x} \frac 1{\lambda(Q)} \int_{Q} |f|\,d\lambda.
\end{gather*}
The usual \emph{Hardy--Littlewood maximal function} is given by
\begin{gather*}\label{H-L-max}
    f^*(x) = \sup_{\ball{r}{x}, r\leq 1} \frac 1{\lambda(B)} \int_{B} |f|\,d\lambda,\quad x\in C;
\end{gather*}
if necessary, we extend $f$ periodically to $\rd$ and use Lebesgue measure $\lambda$ on $\rd$. In the definition of $f^*$ it does not matter whether the balls are centred around $x$ or just contain them, and we also may only consider balls with rational centres and rational radii containing $x$, see Schilling~\cite[Lem.~25.16]{schilling-mims}.

\begin{theorem*}\index{LlogL@$L\log L$ class}
\textit{The following estimates for the Hardy-Littlewood maximal function hold true: For any $f\in L^1(C,\lambda)$
\begin{gather}\label{mar-e20}
    \|f^*\|_{L^p}
    \leq
    \begin{dcases}
        c_d\frac{p}{p-1}\|f\|_{L^p}, &\text{if\ \ } p\in (1,\infty);\\
        c_d(1+\|f\log^+|f|\|_{L^1}), &\text{if\ \ } p = 1.
    \end{dcases}
\end{gather}
Moreover, $\|f^*\|_{L^p} < \infty$ if, and only if, $\|f\log^+|f|\|_{L^1}<\infty$ \textup{(}$p=1$\textup{)} or $\|f\|_{L^p}<\infty$ \textup{(}$1<p<\infty$\textup{)}, respectively.}
\end{theorem*}

If we can show that the dyadic maximal function $f^\diamond$ satisfies $f^\diamond \geq c f^*$, then the estimates \eqref{mar-e20} follow from Doob's inequalities \eqref{mar-e08}. The trouble is, that a ball $\ball{r}{x}$ of radius $r < \frac 12 2^{-k}$, $k=-1,0,1,2,3,\dots$, need not entirely fall into any single square of our lattice $\Gscr_k$, and it even need not be entirely contained in $C$.
\begin{figure}[ht]
    \includegraphics[width = 0.4\textwidth]{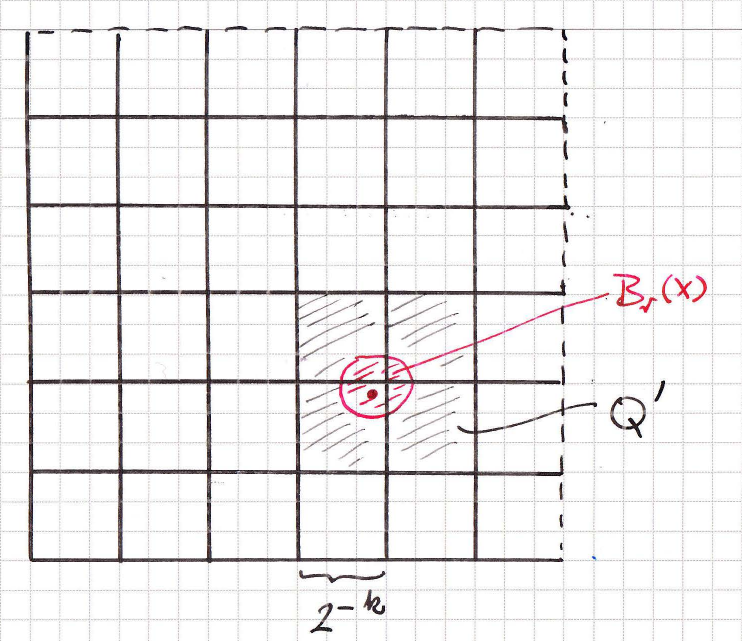}
    \caption{The picture shows a square from $\Gscr_{k-2}$ which is subdivided into four resp.\ sixteen squares from $\Gscr_{k-1}$ and $\Gscr_{k}$, respectively. A ball $\ball{r}{x}$ of radius $\frac 12 2^{-k-1} \leq r < \frac 12 2^{-k}$, $k=-1,0,1,2,\dots$, need not entirely fall into any single square of side-length $2^{-k}$, but it hits at most $2^d$ such squares \textup{(}shaded\textup{)}, each of them is a corner of a square from $\Gscr_{k-1}$ with side-length $2^{1-k}$, and all of them are in the square from $\Gscr_{k-2}$ with side-length $2^{2-k}$. Note that the pictures shows the worst possible case as $\ball{r}{x}$ may already fall into a square from $\Gscr_{k-1}$ \textup{(}or even $\Gscr_{k}$\textup{)}.}
\label{mar-grid}
\end{figure}

The latter problem can be fixed by extending $f$ and $\Gscr_k$, hence $f_k$ and $\Fscr_k$, periodically. This allows us to reduce everything to the situation where $\ball{r}{x}$, $\frac 12 2^{-k-1} \leq r <\frac 12 2^{-k}$, is within $C$, but fails to be inside a single $Q\in\Gscr_k$. In this case, $\ball{r}{x}$ intersects at most $2^d$ of the sets in $\Gscr_k$ (see Fig.~\ref{mar-grid}). Denote by $Q'\in\Gscr_{k-2}$ a cube of side length $4\cdot 2^{-k}$ such that $\ball{r}{x}\subseteq Q'$. Then,
\begin{align*}
    \frac 1{\lambda(\ball{r}{x})} \int_{\ball{r}{x}} |f|\,d\lambda
    \leq \frac{\lambda(Q')}{\lambda(\ball{r}{x})} \frac 1{\lambda(Q')} \int_{Q'}|f|\,d\lambda
    \leq c_d f^{\diamond}(x).
\end{align*}
The constant $c_d$ comes from the comparison of $\lambda(Q')$ and $\lambda(\ball{r}{x})$, and depends only on the space dimension as $\frac 12 2^{-k-1} \leq r <\frac 12 2^{-k}$; this constant is of the order $4^d$.

This proves $f^*(x)\leq c_d f^\diamond(x)$. A similar reasoning also gives the converse inequality $f^\diamond(x)\leq c_d' f^*(x)$.

In order to see the second part of the assertion, we note first that the martingale convergence theorem gives $\lim_{k\to\infty} f_k(x) = f(x)$ for almost all $x$. Note that this essentially proves Lebesgue's differentiation theorem. Therefore, $\|f\|_{L^p}\leq \|f^\diamond\|_{L^p}$ follows easily for all $p\in [1,\infty)$. The assertion for $p=1$ follows with Theorem~\ref{mar-11}, if we use $|f|$ instead of $f$.
\end{example}

\medskip
    A further Doob-type maximal result on martingales is discussed on connection with $\mathrm{BMO}$ spaces in Theorem~\ref{gen-35}.

\begin{example}[The \enquote{sharp} maximal function]\label{mar-17}
We continue Example~\ref{mar-15}. Recall that $C = [0,1)^d$ is the half-open unit cube in $\rd$, $\lambda$ is Lebesgue measure on $(C,\Bscr(C))$, and $\Fscr_k=\sigma(\Gscr_k)$ is the $\sigma$-algebra generated by the dyadic cubes $\Gscr_k$ in $C$ with side-length $2^{-k}$. For $f\in L^1(C,\lambda)$ we define the \emph{sharp maximal function} as \index{maximal function!sharp $(\#)$}
\begin{equation*}
    f^{\#}(x) \coloneqq \sup_{Q\ni x} \frac 1{\lambda(Q)} \int_Q |f - f_Q|\,d\lambda
    \quad\text{where}\quad
    f_Q \coloneqq \frac 1{\lambda(Q)} \int_Q |f|\,d\lambda;
\end{equation*}
the supremum ranges over all axis-parallel cubes $Q\subset C$ such that $x\in Q$. Compared with the Hardy--Littlewood maximal function $f^*$, we have now an additional centering with $f_Q$. This is useful if we want to deal with $\mathrm{BMO}$-spaces, see Remark~\ref{gen-37}; trivially $f\in\mathrm{BMO}$ if, and only if, $f^{\#}\in L^\infty$. It is straightforward to see that
\begin{gather*}
    f^{\#}(x) \leq 2 f^*(x)
    \quad\text{and, therefore,}\quad
    \|f^{\#}\|_{L^p} \leq \frac{2c_d p}{p-1}\, \|f\|_{L^p} \quad (1<p<\infty)
\end{gather*}
using the results from Example~\ref{mar-15}. Again, $p=1$ and $f^*, f^{\#}\in L^1$, requires $f\log^+|f| \in L^1$. But now we also have the converse estimate.

\begin{theorem*}[Fefferman--Stein \cite{fef-ste72}]
    Let $f\in L^r(C,\lambda)$ for some $1<r<\infty$. For every $r<p<\infty$ there is a constant $\kappa_{d,p}$ such that the sharp maximal function satisfies
    \begin{gather*}
        \|f\|_{L^p}\leq \|f^{\diamond}\|_{L^p} \leq \kappa_{d,p}\|f^{\#}\|_{L^p}.
    \end{gather*}
\end{theorem*}

The proof of this theorem relies on the good-$\lambda$-inequality \eqref{pre-e21}, which can be established in the following form
\begin{gather*}
    \lambda\left(\left\{ x\in C \mid f^{\diamond}(x) > 2\lambda, \: f^{\#}(x) < \gamma\lambda\right\}\right)
    \leq \gamma 2^d \lambda\left(\left\{ x\in C \mid f^{\diamond}(x) \geq \lambda\right\}\right),
\end{gather*}
for $\gamma\in (0,1)$, see e.g.\ Grafakos~\cite[Thm.~7.4.4]{grafakos09}. Take $\gamma < 2^{-p-d-1}$ and use Lemma~\ref{pre-14} to conclude that
\begin{gather*}
    \|f^{\diamond}\|_{L^p} \leq \frac {2\cdot 2^{1/p}}\gamma \|f^*(x)\|_{L^p}, \quad \gamma < 2^{-p-d-1},
\end{gather*}
and since $\|f\|_{L^p}\leq \|f^{\diamond}\|_{L^p}$ -- cf.\ the last paragraph in the previous Example~\ref{mar-15} -- the claim follows.

The sharp maximal function was introduced by Fefferman \& Stein~\cite{fef-ste72} in order to study interpolation \index{BMO}\index{interpolation} pairs when one end-point is the space $\mathrm{BMO}$. A detailed modern presentation is given by Grafakos~\cite[Ch.~7.4]{grafakos09}, see also Torchinsky \cite[Ch.~ VIII.1,2]{torchinsky86}. Recently, Kami\'nski \& Os\c{e}kowski~\cite{kam-ose23} have found an optimal bound for the constant $\kappa_{1,p}$ in dimension $d=1$.
\end{example}

\subsection{Khintchine and Burkholder--Davis--Gundy Inequalities}\label{mar-bdg}
Before we tend to the general case of the Burkholder--Davis--Gundy inequalities, we want to discuss an inequality which is frequently used in analysis. As it turns out it is related to the concentration of measure method, see Section~\ref{pre-con},  as well as to Doob's maximal inequality, cf.\ Theorem~\ref{mar-09}. The following inequality is Azuma's~\cite{azuma67} generalization of Hoeffding's inequality~\cite{hoeffding63} for sums of bounded, independent random variables.
\begin{lemma}[Azuma]\label{mar-31}\index{inequality!Azuma}
    Let $(X_n,\Fscr_n)_{n\in\nat_0}$ be a martingale such that the increments are bounded: $|X_n-X_{n-1}|\leq c_n$ for a sequence  $(c_n)_{n\in\nat}\subseteq[0,\infty)$. Set $d_n^2\coloneqq  c_1^2+\dots+c_n^2$. Then
    \begin{align}\label{mar-e30}
        \Pp(X_n-X_0 \geq \lambda)
        &\leq \exp\left[-\frac{\lambda^2}{2d_n^2}\right], \quad \lambda\geq 0,\; n\in\nat,\\
    \label{mar-e32}
        \Pp(|X_n-X_0| \geq \lambda)
        &\leq 2\exp\left[-\frac{\lambda^2}{2d_n^2}\right],\quad \lambda\geq 0,\; n\in\nat.
    \end{align}
\end{lemma}
\noindent
Further generalizations, e.g.\ if the $c_n$ are random variables, are due to Hitczenko \cite[Thm.~3.1, Lem.~4.3]{hitczenko90}.
\begin{proof}
    Since $(-X_n)_{n\in\nat_0}$ is also a martingale with bounded increments, it is enough to prove \eqref{mar-e30} and add the resulting inequalities for $(\pm X_n)_{n\in\nat_0}$ to get \eqref{mar-e32}.

    The function $e_\lambda(x)\coloneqq  e^{\lambda x}$, $x\in\real$, $\lambda\geq 0$, is convex. Thus, the graph of $e_\lambda$ is on the segment $[-c,c]$ below the line connecting the points $(-c; e_\lambda(-c))$ and $(c; e_\lambda(c))$, i.e.
    \begin{gather*}
        e^{\lambda x}
        \leq \frac{e^{\lambda c}-e^{-\lambda c}}{2c}\,x + \frac 12\left(e^{\lambda c}+e^{-\lambda c}\right),\quad
        x\in [-c,c].
    \end{gather*}
    Set $x=(X_{n}-X_{n-1})$ and $c=c_n$, and take on both sides conditional expectation $\Ee(\dots\mid\Fscr_{n-1})$. This gives
    \begin{align*}
        &\Ee\left(e^{\lambda(X_n-X_{n-1})}\mid\Fscr_{n-1}\right)\\
        &\quad\leq \frac{e^{\lambda c_n}-e^{- \lambda c_n}}{2c_n} \, \Ee\left( X_n-X_{n-1}\mid\Fscr_{n-1}\right) + \frac{e^{\lambda c_n}+e^{-\lambda c_n}}2\\
        &\quad= \frac{e^{\lambda c_n}+e^{-\lambda c_n}}2 \leq e^{\lambda^2c_n^2/2}.
    \end{align*}
    The tower property of conditional expectation and a pull out argument yield
    \begin{align*}
        \Ee \left(e^{\lambda(X_n-X_0)}\right)
        &= \Ee\left( \prod_{i=1}^n e^{\lambda(X_i-X_{i-1})}\right)\\
        &= \Ee\left( \prod_{i=1}^{n-1} e^{\lambda(X_i-X_{i-1})} \Ee\left[ e^{\lambda(X_n-X_{n-1})}\mid\Fscr_{n-1}\right]\right)\\
        &\leq \Ee\left( \prod_{i=1}^{n-1} e^{t(X_i-X_{i-1})} \right) e^{\lambda^2c_n^2/2}\\
        &\leq \cdots \leq e^{\lambda^2d_n^2/2}.
    \end{align*}
    Using the Markov inequality for $t,\xi\geq 0$ gives
    \begin{align*}
        \Pp\left(X_n-X_0 \geq \lambda\right)
        = \Pp\left(e^{\xi(X_n-X_0)} \geq e^{\lambda\xi}\right)
        \leq e^{-\lambda\xi}\Ee e^{\xi(X_n-X_0)}
        \leq e^{-\lambda\xi + \xi^2d_n^2/2}.
    \end{align*}
    Since the right hand side becomes minimal at $\xi = \lambda/d_n^2$, we get \eqref{mar-e30}.
\end{proof}

With the help of Azuma's inequality we can get \emph{Khintchine's inequalities}. The upper bound in Theorem~\ref{mar-33} was first proved in Khintchine~\cite{khintchine23}, the beautiful \enquote{duality argument} for the lower bound is due to Littlewood~\cite{littlewood30} and Paley \& Zygmund~\cite{pal-zyg30}. In analysis, the iid random variables appearing in Theorem~\ref{mar-33} are often represented by the Rademacher functions on the probability space $([0,1), \Bscr[0,1],\lambda)$, hence the alternative name \emph{Rademacher inequalities}.
\begin{theorem}[Khintchine]\label{mar-33} \index{inequality!Khintchine}
    Let $(c_n)_{n\in\nat}\subseteq\real$ be a sequence of real numbers and $(\xi_n)_{n\in\nat}$ iid Bernoulli random variables such that $\Pp(\xi_1=\pm 1)=\frac 12$. For all $0<p<\infty$
    \begin{gather}\label{mar-e34}
        \gamma_p \left(\sum_{i=1}^n c_i^2\right)^{p/2}
        \leq
        \Ee\left(\left|\sum_{i=1}^n {c_i\xi_i}\right|^p\right)
        \leq
        \Gamma_p \left(\sum_{i=1}^n c_i^2\right)^{p/2},\quad n\in\nat.
    \end{gather}
    The constants $0<\gamma_p \leq \Gamma_p<\infty$ depend only on $p$.
\end{theorem}
\begin{proof}
\emph{Upper estimate}.
    $S_0\coloneqq  0$, $S_n \coloneqq  \xi_1+\dots+\xi_n$ is a martingale, and so is the \emph{martingale transform}
    \begin{gather*}\label{mg-trafo}
        X_n \coloneqq  c\bullet S_n \coloneqq  \sum_{i=1}^n c_i (S_i-S_{i-1}) = \sum_{i=1}^n c_i \xi_i.
    \end{gather*}
    Azuma's inequality (Lemma~\ref{mar-31}) yields
    \begin{gather*}
        \Pp\left(|X_n| \geq \lambda\right)
        \leq 2 e^{-\lambda^2/2d_n^2},\quad d_n^2 = c_1^2+\dots+c_n^2,
    \end{gather*}
    and a direct calculation along the lines of Lemma~\ref{pre-12} shows for all $0<p<\infty$
    \begin{align*}
        \Ee\left(|X_n|^p\right)
        \leq 2p\int_0^{\;\:\infty} t^{p-1} e^{-t^2/2d_n^2}\,dt
        =  p 2^{p/2} \Gamma(p/2)\, \left(d_n^2\right)^{p/2},
    \end{align*}
    i.e.\ $\Gamma_p \coloneqq  p 2^{p/2} \Gamma(p/2)$.

\medskip\noindent\emph{Lower estimate}.
    The following \enquote{duality} trick is due to Littlewood~\cite{littlewood30}: Using independence, we see
    \begin{gather*}
        \Ee \left(X_n^2\right) = \Vv X_n = \sum_{i=1}^n \Vv(c_i\xi_i) = \sum_{i=1}^n c_i^2,
    \end{gather*}
    and the lower estimate can be rewritten as
    \begin{gather*}
        \gamma_p\|X_n\|_{L^2}^p\leq \|X_n\|_{L^p}^p.
    \end{gather*}
    This inequality is obvious if $p\geq 2$, since the $L^p(\Pp)$ norm is increasing in $p$. If $p\in (0,2)$, we write
    \begin{gather*}
        2= p\alpha + 4\beta
        \quad\text{with suitable}\quad
        \alpha+\beta=1,\; \alpha, \beta\in (0,1),
    \end{gather*}
    and use H\"{o}lder's inequality in conjunction with the upper estimate:
    \begin{align*}
        \Ee\left(|X_n|^2\right)
        = \Ee\left(|X_n|^{p\alpha + 4\beta}\right)
        &\leq \left\{\Ee\left(|X_n|^p\right)\vphantom{\Ee\left(|X_n|^4\right)}\right\}^\alpha  \left\{\Ee\left(|X_n|^4\right)\right\}^\beta\\
        &\leq \left\{\Ee\left(|X_n|^p\right)\vphantom{\Ee\left(|X_n|^4\right)}\right\}^\alpha  \Gamma_4^{\beta}\left\{\Ee\left(|X_n|^2\right)\right\}^{2\beta}.
    \end{align*}
    From this we get the lower estimate for $0<p<2$ with $\gamma_p = \Gamma_4^{1 - 2/p}$.
\end{proof}

\begin{remark}\label{mar-42}
The proof of the upper bound in \eqref{mar-e34} uses Azuma's inequality \eqref{mar-e32}. The following argument shows that both inequalities are equivalent -- up to some constant in the exponent.

The proof of Theorem~\ref{mar-33} yields the constant $\Gamma_p \coloneqq  p 2^{p/2} \Gamma(p/2)$ in the upper bound of the Khintchine inequality, which is not far off the known best constant $2^{p/2}\Gamma((p+1)/2)/\sqrt \pi$ (for $p>2$), see Haagerup~\cite{haagerup82}.

In order to simplify our calculations we use the following slightly rough estimate, which will, however, be good enough for the point we want to make. Since $y e^{-y} \leq e^{-1}$, we see for all $\lambda>0$ and $\alpha,\beta\in (0,1)$ with $\alpha+\beta=1$
\begin{align*}
    p\Gamma\left(\tfrac p2\right)
    = 2\Gamma\left(\tfrac p2+1\right)
    &= 2\int_0^\infty x^{p/2} e^{-x}\,dx\\
    &= 2 p \cdot p^{p/2} \int_0^\infty \left(ye^{-2\alpha y}\right)^{p/2} e^{-p\beta y}\,dy\\
    &\leq \frac 2\beta p^{p/2} (2\alpha e)^{-p/2}.
\end{align*}
Thus, assuming the upper bound in \eqref{mar-e34}, we see with Markov's inequality
\begin{align*}
    \Pp\left(\left|c\bullet S_n\right| > \lambda\right)
    \leq \frac 1{\lambda^{p}} \Ee\left[\left|(c\bullet S)_n\right|^p\right]
    \leq \frac 2\beta \left(\frac{d_n^2}{2\alpha e\lambda^2}\,p\right)^{p/2},
\end{align*}
where $d_n^2 = c_1^2 + \dots + c_n^2$. As a function of $p$, the right hand side becomes minimal if $\left({d_n^2}/{\alpha e\lambda^2}\right) p = 1/e$, i.e.\ $p = \alpha\lambda^2/d_n^2$. Thus,
\begin{align*}
    \Pp\left(\left|c\bullet S_n\right| > \lambda\right)
    \leq \frac 2\beta e^{-\alpha\lambda^2/(2d_n^2)},
    \quad \lambda >0,\: \alpha+\beta = 1,\: \alpha,\beta\in (0,1),
\end{align*}
which is essentially Azuma's inequality.
\end{remark}

\begin{remark}\label{mar-41}
    The following argument explains, why \eqref{mar-e34} can be seen as a maximal inequality. If $\sum_{i=1}^\infty c_i^2 < \infty$, we infer from \eqref{mar-e34} that the martingale $\sum_{i=1}^n c_i\xi_i$ is uniformly integrable, hence we get\footnote{We use the symbol $f\asymp g$ to indicate two-sided estimates $c f \leq g \leq C f$, with comparison constants $0<c<C<\infty$ depending only on $p$.}
    \begin{gather*}
        \left(\Ee\left[\left|\sum_{i=1}^\infty c_i\xi_i\right|^p\right]\right)^{1/p}
        \asymp \left(\sum_{i=1}^\infty c_i^2\right)^{1/2},\quad
        0<p<\infty.
    \end{gather*}
    On the other hand, by Doob's maximal inequality,
    \begin{gather*}
        \left(\Ee\left[\left|\sum_{i=1}^\infty c_i\xi_i\right|^p\right]\right)^{1/p}
        \asymp \left(\Ee\left[\sup_{n\in\nat}\left|\sum_{i=1}^n c_i\xi_i\right|^p\right]\right)^{1/p},\quad p>1.
    \end{gather*}
    Finally, since the $L^p$-norm of a probability measure is monotone in $p\in (0,\infty)$, we see that \eqref{mar-e34} implies
    \begin{gather*}\label{compen}
        \left(\sum_{i=1}^\infty c_i^2\right)^{1/2}
        \asymp \left(\Ee\left[\left|\sum_{i=1}^\infty c_i\xi_i\right|^p\right]\right)^{1/p}
        \asymp \left(\Ee\left[\sup_{n\in\nat} \left|\sum_{i=1}^n c_i\xi_i\right|^p\right]\right)^{1/p}.
    \end{gather*}
    Using the sequence $(c_1,c_2,\dots,c_n,0,0,\dots)$ brings us back to \eqref{mar-e34}.
\end{remark}

The Burkholder--Davis--Gundy (BDG) inequalities can be seen as a generalization of the Khintchine inequalities to martingales, which are more general than the random walk martingale $S_n = \xi_1+\dots+\xi_n$, and to martingale transforms $X_n \coloneqq  c\bullet S_n \coloneqq  \sum_{i=0}^n c_i(S_i-S_{i-1})$ with a previsible -- and not just deterministic -- sequence $(c_n)_{n\in\nat}$, $c_n = c_n(\omega)$. In order to appreciate this extension, it is helpful to interpret the term $\sum_{i=1}^n c_i^2$ appearing in Khintchine's inequalities as the \emph{quadratic variation} (or \emph{compensator} or \emph{square function}) of the martingale $X = c\bullet S$. Indeed, the compensator $[S]_n$ of the simple random walk martingale $S_n = \xi_1+\dots+\xi_n$ appearing in Theorem~\ref{mar-33} is  $[S]_n = n$, and so \index{square bracket}
\begin{gather*}
    [c\bullet S]_n = c^2\bullet [S]_n = \sum_{i=1}^n c_i^2([S]_i-[S]_{i-1}) = \sum_{i=1}^n c_i^2.
\end{gather*}
This point of view allows us to find a continuous-time extension of Khintchine's inequalities. We will distinguish between two cases: (i) $(X_t)_{t\in I}$ is a discrete martingale or it is a continuous martingale whose paths are c\`adl\`ag (right continuous \& finite left limits) with jumps, and (ii) $(X_t)_{t\geq 0}$ is a martingale in continuous time and with continuous paths. The latter situation is much easier to handle, since the previsible quadratic variation $\bracket{X} = (\bracket{X}_t)_{t\geq 0}$ and the adapted quadratic variation $\sbracket{X} = (\sbracket{X}_t)_{t\geq 0}$ coincide.

We will give two proofs of the BDG inequalities, first the standard proof based on It\^o's formula and then an alternative approach using the fact that continuous martingales are time-changed Brownian motions -- a further explanation as to why the continuous case is easier than the jump case or the discrete case.

\begin{theorem}[Burkholder--Davis--Gundy]\label{mar-43}\index{inequality!BDG (continuous case)}
    Let $(X_t,\Fscr_t)_{t\geq 0}$, $X_0=0$, be a martingale with continuous paths and \textup{(}predictable\textup{)} quadratic variation process $\left(\bracket{X}_t\right)_{t\geq 0}$. For any $p > 0$ we have
    \begin{gather}\label{mar-e40}
        c_p\Ee\left(\bracket{X}_\infty^{p/2}\right)
        \leq
        \Ee\left(\left(X^*\right)^p\vphantom{\bracket{X}_\infty^{p/2}}\right)
        \leq
        C_p\Ee\left(\bracket{X}_\infty^{p/2}\right).
    \end{gather}
    The constants $0<c_p\leq C_p<\infty$ depend only on $p\in (0,\infty)$.
\end{theorem}

\begin{remark}\label{mar-45}
a)
    Using a stopping argument $X_t \rightsquigarrow X^\tau_t \coloneqq  X_{t\wedge\tau}$ we can change \eqref{mar-e40} to a finite (deterministic or random) time horizon $[0,\tau]$. Moreover, \eqref{mar-e40} is always understood as taking values in $[0,\infty]$, i.e.\ if the right hand side is finite, then $X^*\in L^p$, hence $(X_t)_{t\geq 0}\subseteq L^p$. If $p\geq 1$ and $X^*\in L^p$, then $(X_t)_{t\geq 0}$ is uniformly integrable, and we get, as in Doob's inequality (Theorem~\ref{mar-09}), that $\Ee\left[|X_\infty|^p\right]$ and $\Ee\left[\left(X^*\right)^p\right]$ are comparable.

\medskip
b)
    There is a huge literature on the optimality of the constants in the BDG and related inequalities. For this we refer, in particular, to the monograph by Os\c{e}kowski \cite{osekowski12}. Recent developments can be found in  Kami\'nski \& Os\c{e}kowski \cite{kam-ose23}. An approach to BDG-inequalities via the Garsia--Rodemich--Rumsey lemma is due to Barlow \& Yor~\cite{bar-yor82} yielding more general BDG-type inequalities for continuous (semi-)martingales.
\end{remark}

\begin{proof}[Standard proof of Theorem~\ref{mar-43}: Getoor \& Sharpe~\cite{get-sha72}]\label{hidden-BDG-1}
    By stopping with the stopping time
    \begin{gather*}
        \tau_n \coloneqq  \inf\left\{s\geq 0 \mid |X_s|^2 + \bracket{X}_s \geq n\right\},
    \end{gather*}
    we can, without loss of generality, assume that both $X$ and $\bracket{X}$ are bounded;\footnote{As usual, the problem is the overshoot when crossing the level $n$. Since $t\mapsto X_t$ and $t\mapsto \bracket{X}_t$  are continuous, we trivially have $\left|X_{\tau_n} - X_{\tau_n-}\right|=0$ and $\bracket{X}_{\tau_n}-\bracket{X}_{\tau_n-}=0$, but this will change dramatically for a jump process.} in particular $X$ is uniformly integrable, and $X_\infty$ is well-defined.

\medskip
    Assume first that $p\geq 2$ and denote by $q = p/(p-1)$ the conjugate exponent. Combining the fact that $X^2-\bracket{X}$ is a martingale with Doob's maximal inequality (Theorem~\ref{mar-09}) proves \eqref{mar-e40} for $p=2$:
    \begin{gather*}
            \Ee\left[\bracket{X}_\infty\right]
            = \Ee\left[\left|X_\infty\right|^2\right]
            \leq \Ee\left[\left(X^*\right)^2\right]
            \leq 4 \sup_{s\geq 0} \Ee\left[\left|X_s\right|^2\right]
            = 4 \Ee\left[\bracket{X}_\infty\right].
    \end{gather*}
    Since $x\mapsto |x|^p$, $p>2$, is a $C^2$-function, It\^o's formula, yields
    \begin{gather*}
        |X_t|^p = p\int_0^t |X_s|^{p-1} \sgn(X_s)\,dX_s + \frac 12\,p(p-1)\int_0^t |X_s|^{p-2}\,d\bracket{X}_s.
    \end{gather*}
    Because of the boundedness of $|X_s|^{p-1}$, the first integral on the right hand side is a martingale, and we see with Doob's maximal inequality and the H\"{o}lder inequality for $p/(p-2)$ and $p/2$
    \begin{align*}
        \Ee\left[\sup_{s\geq 0}|X_s|^p\right]
        &\leq q^p\,\frac{p(p-1)}2\, \Ee\left[\int_0^\infty |X_s|^{p-2}\,d\bracket{X}_s\right]\\
        &\leq q^p\,\frac{p(p-1)}2 \Ee\left[\sup_{s\geq 0}|X_s|^{p-2} \, \bracket{X}_\infty\right]\\
        &\leq q^p\,\frac{p(p-1)}2 \left(\Ee\left[\sup_{s\geq 0}|X_s|^{p}\right]\right)^{1-2/p}
          \left(\Ee\left[\bracket{X}_\infty^{p/2}\right]\vphantom{\sup_{s\geq 0}}\right)^{2/p}.
    \end{align*}
    Dividing by $\left(\Ee\left[\sup_{s\geq 0}|X_s|^{p}\right]\right)^{1-2/p}$ gives the upper estimate in \eqref{mar-e40} with $C_p = \left[q^p\,p(p-1)/2\right]^{p/2}$.

\medskip
    In order to prove the lower estimate in \eqref{mar-e40} we define an auxiliary martingale
    $
        Y_t
        \coloneqq  \int_0^t \bracket{X}_s^{(p-2)/4}\,dX_s.
    $
    Observe that
    \begin{gather*}
        \smash[b]{\bracket{Y}_t
        = \int_0^t \bracket{X}_s^{(p-2)/2}\,d\bracket{X}_s
        = \frac 2p\,\bracket{X}_t^{p/2}}
    \intertext{and}
        \smash[t]{\Ee\left[\bracket{X}_t^{p/2}\right]
        = \frac p2\Ee\left[\bracket{Y}_t\vphantom{|Y_t|^2}\right]
        = \frac p2\Ee\left[|Y_t|^2\right].}
    \end{gather*}
    The It\^o formula for $f(x,y)=xy$ and the bivariate process $(X_t,\bracket{X}_t^{(p-2)/4})$ gives
    \begin{align*}
        X_t \bracket{X}_t^{(p-2)/4}
        &= \int_0^t \bracket{X}_s^{(p-2)/4}\,dX_s + \int_0^t X_s\,d\bracket{X}_s^{(p-2)/4}\\
        &= Y_t + \int_0^t X_s\,d\bracket{X}_s^{(p-2)/4}.
    \end{align*}
    If we rearrange this equality, we get
    $
        |Y_t| \leq 2\sup_{s\leq t} |X_s|\,\bracket{X}_t^{(p-2)/4}.
    $
    Finally, using H\"{o}lder's inequality with $p/2$ and $p/(p-2)$, shows
    \begin{align*}
        \frac 2p\,\Ee\left[\bracket{X}_t^{p/2}\right]
        = \Ee\left[|Y_t|^2\right]
        &\leq 4\Ee\left[\sup_{s\leq t}|X_s|^2\,\bracket{X}_t^{(p-2)/2}\right]\\
        &\leq 4 \left(\Ee\left[\sup_{s\leq t}|X_s|^p\right]\right)^{2/p}
           \left(\Ee\left[\bracket{X}_t^{p/2}\right]\vphantom{\sup_{s\leq t}}\right)^{1-2/p}.
    \end{align*}
    Divide both sides by $\left(\Ee\left[\bracket{X}_t^{p/2}\right]\right)^{1-2/p}$ and let $t\to\infty$ to get the lower bound in \eqref{mar-e40} with $c_p = (2p)^{-p/2}$.

\medskip
    Now we assume that $0<p<2$. In this case, we use the good-$\lambda$ argument from Lemma~\ref{pre-18}, setting $X = \left(X^*\right)^2$ and $Y=\bracket{X}_\infty$. Let $\tau \coloneqq  \inf\left\{s\geq 0 \mid \bracket{X}\geq\lambda\right\}$. We have
    \begin{align*}
        P \coloneqq  \Pp\left(\left(X^*\right)^2 \geq \lambda,\; \bracket{X}_\infty < \lambda\right)
        = \Pp\left(\left(X^*\right)^2 \geq \lambda,\; \tau = \infty\right)
        \leq \Pp\left(\left(X_\tau^*\right)^2 \geq \lambda\right),
    \end{align*}
    and with optional stopping, Doob's maximal inequality \eqref{mar-e08} and the definition of $\tau$, we find
    \begin{gather*}
        P \leq \frac 1\lambda \Ee \left(\left|X_{\tau}\right|^2\right)
        = \frac 1\lambda \Ee \left(\bracket{X}_{\tau}\right)
        = \frac 1\lambda \Ee \left(\bracket{X}_{\infty}\wedge \lambda\right).
    \end{gather*}
    Now we can apply Lemma~\ref{pre-18} and we get
    \begin{gather*}
        \Ee\left[\left(X^*\right)^{p}\right]
        \leq \frac{4-p}{2-p}\,\Ee\left[\bracket{X}^{p/2}_\infty\right].
    \end{gather*}
    This proves the upper bound in \eqref{mar-e40} for $0<p<2$.

    For the lower bound we using again Lemma~\ref{pre-18}, this time with $X=\bracket{X}_\infty$ and $Y = \left(X^*\right)^2$. Define $\sigma \coloneqq  \inf\{s> 0 \mid |X_s|^2 \geq \lambda\}$. Then
    \begin{gather*}
        P' \coloneqq  \Pp\left(\bracket{X}_\infty \geq \lambda,\; \left(X^*\right)^2 < \lambda\right)
        \leq \Pp\left(\bracket{X}_\infty \geq \lambda,\; \sigma=\infty\right)
        \leq \Pp\left(\bracket{X}_{\sigma} > \lambda\right).
    \end{gather*}
    Using the Markov inequality and the fact that $X^2-\bracket{X}$ is a martingale, we get
    \begin{gather*}
        P' \leq \frac 1\lambda \Ee\left[\bracket{X}_{\sigma}\right]
        \leq \frac 1\lambda \Ee\left[\left(X_\sigma^*\right)^2\right]
        \leq \frac 1\lambda \Ee\left[\left(X_\sigma^*\right)^2\wedge \lambda\right],
    \end{gather*}
    and Lemma~\ref{pre-18} gives
    \begin{gather*}
        \Ee\left[ \left(X^*\right)^{p}\right]
        \geq \frac{2- p}{4-p}\,\Ee\left[\bracket{X}^{p/2}_\infty\right],
    \end{gather*}
    which is the lower bound in \eqref{mar-e40} for $0<p<2$.
\end{proof}

\begin{proof}[Time change proof of Theorem~\ref{mar-43}]\label{hidden-BDG-2}\footnote{It seems that Burkholder \& Gundy~\cite[pp.~109--110 \& \S 7]{bur-gun70} were the first to point out this strategy.}\index{Brownian motion!time-change}%
    By the Dambis \& Dubins--Schwarz theorem, we can represent any continuous martingale $(X_t)_{t\in I}$ with $X_0=0$ as a time-changed Brownian motion, see e.g.\ Schilling~\cite[Thm.~19.17]{schilling-bm} and Revuz \& Yor~\cite[Thm.~V.1.6]{rev-yor99}; in fact, we have $X_t = B_{\bracket X_t}$, for some Brownian motion $(B_t)_{t\in I}$ which is, in general, not independent of $(X_t)_{t\in I}$. Therefore, it is enough to show \eqref{mar-e40} for Brownian motion $(B_t)_{t\geq 0}$ evaluated at some stopping time $\tau$. That is, we have to show that
    \begin{gather}\tag{$\ref{mar-e40}^\prime$}\label{mar-e40-bis}
        c_p\Ee\left[\tau^{p/2}\right] \leq \Ee\left[\left(B_\tau^*\right)^p\vphantom{\tau^{p/2}}\right] \leq C_p\Ee\left[\tau^{p/2}\right].
    \end{gather}

    Let $\sigma \coloneqq  \inf\left\{t\geq 0 \mid B_t^* > \lambda\right\}$. We have for all $\lambda>0$ and $\delta\in (0,1)$
    \begin{align*}
        \Pp\left( B_\tau^* > 2\lambda,\; \sqrt\tau < \delta\lambda\right)
        &= \Pp\left( B_\tau^* > 2\lambda,\; \sqrt\tau < \delta\lambda,\;\sigma<\tau\right)\\
        &\leq \Pp\left( \sup_{\sigma\leq t\leq \tau} |B_t-B_{\sigma}| > \lambda,\; \sqrt\tau < \delta\lambda,\;\sigma<\tau\right)\\
        &\leq \Pp\left( \sup_{\sigma\leq t\leq\smash{\delta^2\lambda^2}} |B_t-B_{\sigma}| > \lambda,\;\sigma<\tau\right)\\
        &= \Ee\left[\Pp^{B_\sigma}\left(\sup_{t\leq \smash{\delta^2\lambda^2}}|B_t-B_0| > \lambda\right) \I_{\{\sigma<\tau\}}\right]\\
        &\leq \sup_{x\in\real} \Pp^x\left(\sup_{t\leq \smash{\delta^2\lambda^2}}|B_t-x| > \lambda\right) \Pp(\sigma < \tau).
    \end{align*}
    In the penultimate step we use the strong Markov property of Brownian motion. Since Brownian motion is translation invariant and scales like $B_{\lambda^2 t}\sim \lambda B_{t}$, we get
    \begin{gather*}
        \Pp\left( B_\tau^* > 2\lambda,\; \sqrt\tau < \delta\lambda\right)
        \leq \Pp\left(B^*_{\delta^2} > 1\right)\Pp(\sigma < \tau)
        \leq \epsilon \Pp(B_\tau^* \geq \lambda).
    \end{gather*}
    Here we use that for a Brownian motion $\Pp\left(B_{\delta^2}^* > 1\right)\to 0$ as $\delta\to 0$. The upper estimate in \eqref{mar-e40-bis} now follows from Lemma~\ref{pre-14}.

    The lower estimate can be shown in a similar fashion, switching the roles of $\sqrt\tau$ and $B_\tau^*$. We have
    \begin{align*}
        \Pp\left(\sqrt\tau > 2\lambda,\; B^*_\tau < \delta\lambda\right)
        &\leq \Pp\left(\sqrt\tau > 2\lambda,\; \sup_{\smash{\lambda^2}\leq t\leq \smash{2\lambda^2}} |B_t - B_{\lambda^2}| < 2\delta\lambda\right)\\
        &\leq \Ee\left[\I_{\{\sqrt\tau > \lambda\}}  \Pp^{B_{\lambda^2}} \left(\sup_{t\leq \smash{\lambda^2}} |B_t - B_{0}| < 2\delta\lambda\right)\right]\\
        &\leq \Pp\left(\tau > \lambda^2\right) \sup_{x\in\real} \Pp^{x} \left(\sup_{t\leq \smash{\lambda^2}} |B_t - x| < 2\delta\lambda\right)\\
        &\leq \Pp\left(B_1^* < 2\delta\right) \Pp\left(\tau > \lambda^2\right) \\
        &\leq \epsilon \Pp\left(\tau > \lambda^2\right).
    \end{align*}
    using again the Markov property, translation invariance and Brownian scaling. A further application of Lemma~\ref{pre-14} gives the lower bound in \eqref{mar-e40}.
\end{proof}

\begin{remark}\label{mar-47}
    The time-change technique can also be used to get sharp \enquote{deviation} estimates for continuous martingales. This complements our earlier Remark~\ref{mar-42} on the role of Azuma's inequalities in relation to Khintchine's inequality.

    Let $(X_t)_{t\geq 0}$ be a continuous martingale and write it as $X_t-X_0 = B_{\bracket X_t}$. Assume that $\bracket X_T \leq d_T^2$ a.s.\ for some $T>0$ and a constant $d_T^2$. Using the scaling property of Brownian motion and the reflection principle we get
    \begin{align*}
        \Pp \left( (X- X_0)^*_T > \lambda\right)
        = \Pp\left(B^*_{\bracket X_T} > \lambda\right)
        &\leq \Pp\left(B^*_{d_T^2} > \lambda\right)\\
        &\leq \Pp\left(B^*_1 > \lambda/ d_T \right) \\
        &= 2\Pp\left(B_1 > \lambda/d_T\right).
    \end{align*}
    With the usual Gaussian tail estimate, we finally get
    \begin{gather*}
        \Pp \left( (X- X_0)^*_T > \lambda\right) \leq 2(2\pi)^{-1/2} \frac{d_T}{\lambda} e^{-\lambda^2/(2d_T^2)}.
    \end{gather*}
    This is sometimes called McKean's inequality, \index{inequality!McKean} see Krylov~\cite[Ch.~IV.\S2]{krylov95}.
\end{remark}

\begin{remark}\label{mar-48}\index{Brownian motion}%
    If $X_t = \int_0^t f(s) \, dB_s$ is a stochastic integral with respect to Brownian motion satisfying $\Ee\left[\bracket{X}_t^{p/2}\right] = \Ee\Big[\left( \int_0^t |f(s)|^2 \, ds\right)^{p/2}\Big] <\infty$ for some $p>1$, then Novikov~\cite{novikov71} showed that
	\begin{gather}\label{mar-e41}
		c_p  \Ee\left[\bracket{X}_t^{p/2}\right]
            \leq \Ee\left[|X_t|^p\right]
            \leq C_p \Ee\left[\bracket{X}_t^{p/2}\right]
	\end{gather}
	for some constants $c_p,C_p>0$. Hence, by the BDG inequalities,
	\begin{equation*}
		\Ee\left[\bracket{X}_t^{p/2}\right]
        \asymp \sup_{s \leq t} \Ee\left[|X_s|^p\right]
        \asymp \Ee \Bigl[ \sup_{s \leq t} |X_s|^p\Bigr].
	\end{equation*}
    While the upper bound in \eqref{mar-e41} holds for any $p>0$, cf.\ Novikov~\cite{novikov71}, the lower bound breaks down for $p=1$. This can be seen by considering the martingale $X_t \coloneqq  B_{t \wedge \tau}$ with $\tau \coloneqq  \inf\{t \geq 0\mid B_t=1\}$. On the one hand, we have
    \begin{gather*}
        \Ee\left[|X_t|\right]
        = \Ee\left[2B_{t\wedge\tau}^+ - B_{t\wedge\tau}\right]
        = 2\Ee\left[B_{t\wedge\tau}^+\right]
        \leq 2.
    \end{gather*}
    On the other hand, $\Ee\left[\bracket{X}_t^{1/2}\right] = \Ee\left[ (t \wedge \tau)^{1/2}\right] \to \Ee\left[\tau^{1/2}\right] = \infty$ as $t \to \infty$,\footnote{A simple proof that $\Ee\left[\tau^a\right]=\infty$ for all $a\geq \frac 12$ (for random walks) can be found in Schilling~\cite[Bsp.~11.10, p.~115--116]{schilling-maps}. Alternatively, one can use a direct calculation and the fact that $\tau \sim (2\pi s^3)^{-1/2} e^{-1/2s}\,ds$, see e.g.\ Schilling~\cite[Thm.~6.10]{schilling-bm}}
    cf.\ Novikov~\cite{novikov13}, so that the lower bound cannot hold for $p=1$.
\end{remark}

\paragraph{The BDG inequalities for jump processes.}
In the final part of this section, we will discuss the BDG inequalities for processes with jumps. Let us first explain why the presence of jumps creates difficulties.

Our first proof of the BDG inequalities for processes (see p.~\pageref{hidden-BDG-1} \emph{et seq.})  with continuous sample paths heavily relied on a stopping argument: we could assume without loss of generality that the process $X$ and its quadratic variation $\bracket X$ are bounded by considering the stopped process $X^{\tau}$ for the stopping time $\tau = \{s \geq 0 \mid |X_s|^2 + \bracket{X}_s \geq n\}$. This reasoning is no longer applicable if $t \mapsto X_t$ has jumps. While $|X_s|^2+\langle X \rangle_s<n$ is still true for $s<\tau$, we do not have any bound on $|X_{\tau}|^2$ because the jumps create an overshoot which cannot be controlled.\footnote{We note, in passing, that also the time-change proof of the BDG inequalities, see \pageref{hidden-BDG-2} \emph{et.\ seq.}, breaks down. Although there are quite general embedding results for right continuous (semi-)martingales into Brownian motion, see Monroe~\cite{monroe72,monroe78} and Dubins~\cite{dubins68} or Schilling~\cite[Thm.~14.7]{schilling-bm} for the discrete-time case, the stopping times appearing in the embedding frequently have unbounded moments of order $p/2$.} The same kind of problem arises for discrete-time processes $(X_n)_{n \in \nat}$, and, in fact, the BDG inequalities for discrete-time processes are essentially as hard to derive as the BDG inequalities for jump processes in continuous time.

From now on we will focus on continuous-time processes with jumps; an exposition of the discrete-time setting can be found in Schilling~\cite{schilling-maps}.

\begin{example}\label{mar-49}
    Among the simplest examples of martingales with jumps are the compound Poisson processes. For a probability distribution $\mu$ on $\real \setminus \{0\}$, let $H_1,H_2,\ldots$ be iid random variables with $H_i \sim \mu$, and let $(N_t)_{t \geq 0}$ be an independent Poisson process with intensity $\Ee N_1 = \Vv N_1 = 1$. Then the sample paths of the \emph{compound Poisson process}
    \begin{equation*}
    	J_t\coloneqq  \sum_{i=1}^{N_t} H_i,\quad t \geq 0,
    \end{equation*}
    are piecewise constant with jumps occurring whenever the value of $N_t$ increases by $1$. The compound Poisson process $(J_t)_{t \geq 0}$ has stationary and independent increments, i.e.\ it is a L\'evy process, see Section~\ref{lp-back}, Scholium~\ref{pre-70}, and also Schilling~\cite[Chapter 3]{crm}. Using the independence and stationarity of the increments, it is not difficult to show that $X_t \coloneqq  J_t- b t$ with $b=\Ee(J_1) = \int y \, \mu(dy)$ is a martingale with quadratic variation $\langle X \rangle_t = t \int y^2 \, \mu(dy)$; of course, we assume that the first and second moments exist. A naive application of Theorem~\ref{mar-43} would give
    \begin{gather}\label{mar-e42}
    	\Ee\left( \sup_{s \leq t} |X_s|^p \right) \asymp t^{p/2} \left( \int y^2 \, \mu(dy) \right)^{p/2},
        \quad t \geq 0,\;p>0,
    \end{gather}
    but this estimate fails to hold. In fact, it can  happen that the left hand side is infinite while the right hand side is finite. To see this, let us consider $p=4$. The moment $\Ee(|X_t|^4)$ exists if, and only if, the characteristic function $\phi(\xi)=\Ee e^{i \xi X_t}$ is four times differentiable at $\xi=0$. Since $\phi(\xi) = \exp(-t \psi(\xi))$ with $\psi(\xi) = \int_{y \neq 0} (1-e^{iy \xi}-iy \xi) \, \mu(dy)$, we find that
    \begin{align} \begin{aligned}
	\Ee(|X_t|^4)<\infty
	&\iff \text{$\psi$ is four times differentiable at $\xi=0$} \\
	&\iff \int_{|y| \geq 1} |y|^4 \, \mu(dy)<\infty,
	\end{aligned} \label{mar-e44}
    \end{align}
    see e.g.\ K\"{u}hn~\cite[Lem.~4.3]{ltp-moments} for a detailed proof or, indeed Theorem~\ref{lp-53} further down. This already indicates that \eqref{mar-e42} cannot hold for $p=4$. In fact, a somewhat lengthy, but easy, calculation yields
    \begin{equation*}
    	\Ee(X_t^4)
    	= \frac{d^4}{d\xi^4} \phi(\xi) \bigg|_{\xi=0}
    	= 3t^2 \left( \int y^2 \, \mu(dy) \right)^2 + t \int y^4 \, \mu(dy).
    \end{equation*}
    While the first term on the right hand side looks familiar from \eqref{mar-e42}, the second one is new. As \eqref{mar-e44} shows, this second term is actually needed to ensure that $\Ee(X_t^4)$ is finite.
\end{example}

The following version of the BDG inequalities for (possibly discontinuous) martingales is due to Dzhaparidze \& Valkeila~\cite[Lem.~2.1]{dzh90}.

\begin{theorem}\label{mar-50}\index{inequality!BDG (jump case)}
    Let $(X_t)_{t \geq 0}$ be a real-valued martingale with c\`adl\`ag sample paths and $X_0=0$. For every $p \geq 2$, there are constants $0<c_p \leq C_p<\infty$ such that
	\begin{gather}\label{mar-e48}
		c_p \Ee\left[\bracket{X}^{p/2}_{t} + |(\Delta X)^*_{t}|^p\right]
		\leq \Ee\left[(X_{t}^*)^p\right]
		\leq C_p \Ee\left[\langle X \rangle_t^{p/2} + |(\Delta X)_t^*|^p\right];
	\end{gather}
	here $(\Delta X)_t = X_t-X_{t-}$ is the jump height at time $t$.
\end{theorem}

Using a standard stopping argument, Theorem~\ref{mar-50} can be extended to stopping times $\tau$.

There are two further variants of the BDG inequalities for discontinuous martingales. We begin with Novikov's~\cite{novikov75} result, which applies to processes which can be written as a stochastic integral with respect to a Poisson random measure; this includes, for instance, L\'evy processes and solutions to stochastic differential equations driven by L\'evy processes, see Sections~\ref{lp} and \ref{fel}. In contrast to Theorem~\ref{mar-50}, the result by Novikov gives moment estimates for any $p>0$.

\begin{theorem}[Novikov] \label{mar-51} \index{inequality!BDG (jump case)}
	 Let $(X_t)_{t \geq 0}$ be a one-dimensional stochastic process of the form
	 \begin{equation*}
	 	X_t = \int_0^t \int_{y \neq 0} F(s,y) \, \widetilde{N}(dy,ds), \quad t \geq 0,
	 \end{equation*}
    for a predictable stochastic process $F$ and a compensated Poisson random measure $\widetilde{N} = N - \widehat{N}$ with compensator $\widehat{N}(dy,ds) = \nu(dy) \, ds$ for a measure $\nu$ satisfying $\int_{y \neq 0} \min\{1,|y|^2\} \, \nu(dy)<\infty$.
    \begin{enumerate}
	\item\label{mar-51-a}
        If $\Ee\left[\int_0^T \int_{y \neq 0} |F(s,y)| \, \nu(dy) \, ds\right] < \infty$ for some $T>0$, then
		\begin{gather}
			\Ee \left( \sup_{t \leq T} |X_t|^p \right)
			\leq C_{p,\alpha} \Ee \left[ \left( \smash{\int_0^T} \int_{y \neq 0} |F(s,y)|^{\alpha} \, \nu(dy) \, ds \right)^{p/\alpha} \right]
			\label{mar-e50}
		\end{gather}
		for any $\alpha \in [1,2]$ and $p \in [0,\alpha] \subseteq [0,2]$.
		\item\label{mar-51-b} If $\Ee(\int_0^T \int_{y \neq 0} |F(s,y)|^2 \, \nu(dy) \, ds) < \infty$ for some $T>0$, then
		\begin{align}
			\label{mar-e52} \begin{aligned}
			\Ee \left( \sup_{t \leq T} |X_t|^p \right)
			&\leq c_p \Ee \left[ \left( \smash{\int_0^T} \int_{y \neq 0} |F(s,y)|^2 \, \nu(dy) \, ds \right)^{p/2} \right] \\
			&\qquad \mbox{}+ c_p \Ee \left[ \int_0^T \int_{y \neq 0} |F(s,y)|^p \, \nu(dy) \, ds \right]
			\end{aligned}
		\end{align}
		for all $p \geq 2$.
	\end{enumerate}
\end{theorem}

The estimate \eqref{mar-e52} for $p \geq 2$ can be derived from Theorem~\ref{mar-50} or by applying It\^o's formula to $x \mapsto x^p$, see e.g.\ Kunita~\cite[Thm.~2.11]{kunita04}. For $p \in (0,2)$, one can show that it suffices to consider $p=\alpha$ and then establish the estimate for this particular case using a change of variable formula which is, so to speak, the substitute for It\^o's formula, cf.\ Novikov~\cite{novikov75}.

\medskip

The other possibility to get BDG inequalities for a jump process $X=(X_t)_{t\geq 0}$ is to replace the predictable compensator $\bracket X$ by the \emph{square bracket} process \index{square bracket}
\begin{gather}\label{mar-e70}
    \sbracket{X}_t \coloneqq  \lim_{n\to\infty}  \left(X_0^2 + \smash[t]{\sum_{i=1}^{k(n)}} (X_{t\wedge\tau_i^n} - X_{t\wedge\tau_{i-1}^n})^2\right)
\end{gather}
(the limit is taken in probability) where $0 = \tau_0^n \leq \tau_1^n\leq \dots\leq \tau_{k(n)}^n\leq t$ are stopping times such that $\max_{1\leq i\leq k(n)} |\tau_i^n - \tau_{i-1}^n|\to 0$ as $n\to\infty$. A full discussion is given in Protter~\cite[Ch.~II.6]{protter}; here we just mention that $\sbracket{X}_t$ is an adapted increasing process such that $\sbracket{X}_t-\bracket X_t$ is a martingale. Moreover, one has the following relation to stochastic integration with jumps:
\begin{gather}\label{mar-e72}
    \Delta\sbracket{X}_t = (\Delta X)_t^2
    \et
    \sbracket{X}_t - \sbracket{X}_0 = X_t^2 - X_0^2 - 2\int_0^t X_{s-}\,dX_s.
\end{gather}

Example~\ref{mar-72} below shows why we must restrict ourselves to $p\geq 1$ in the following \enquote{square-bracket} version of the BDG inequalities.

\begin{theorem}\label{mar-71} \index{inequality!BDG (jump case)}
    Let $(X_t)_{t\geq 0}$ be a real-valued martingale with c\`adl\`ag sample paths and $X_0=0$. For every $p\geq 1$ there are universal constants $0<c_p\leq C_p < \infty$ such that
    \begin{gather}\label{mar-e74}
        c_p \Ee\left[ \sbracket{X}_t^{p/2}\right]
        \leq
        \Ee\left[ \left(X^*_t\right)^{p}\right]
        \leq
        C_p \Ee\left[ \sbracket{X}_t^{p/2}\right],
        \quad t\geq 0.
    \end{gather}
\end{theorem}
It goes without saying that, as in all cases discussed before, \eqref{mar-e74} is stable under stopping and under the limit as $t\to\infty$. Theorem~\ref{mar-71} goes back to the seminal paper \cite{bur-dav-gun72} by  Burkholder, Davis \& Gundy, see also Burkholder~\cite{burkholder73} for a simplified version in discrete time. The first continuous-time proof, including the hitherto unpublished \enquote{Davies} case $p=1$, seems to be due to Meyer~\cite[App.~II]{meyer72}. Full proofs of Theorem~\ref{mar-71} are fairly difficult to find in textbooks, notable exceptions are Cohen \& Elliott~\cite[Thm.~11.5.5]{coh-ell15}, Liptser \& Shiryaev~\cite[pp.~70--83]{lip-shi89} or Schilling~\cite[Ch.~10]{schilling-maps} in the discrete-time setting; usually Dellacherie \& Meyer~\cite[p.~304, \S VII.92]{dm-2},~\cite[pp.~328--329, XIII.14]{dm-5}, Meyer~\cite[App.~II, III]{meyer72} and~\cite[pp.~350--353]{meyer-cours} and  Lenglart \emph{et al.}~\cite[\S 2]{len-lep-pra80} serve as standard references.
 Marinelli \& R\"ockner~\cite{mar-roe16} present a proof of Theorem~\ref{mar-71} along the lines of the stochastic calculus proof of the BDG inequalities for jump processes which also works for Hilbert-space valued random variables. As one would expect, the key ingredient is the careful (conditional) domination of the jumps. If the underlying martingale is given by a stochastic integral, the BDG inequalities for $p=2$ are just It\^o's isometry; Dirksen~\cite{dirksen14} discusses generalizations of It\^o's isometry to $L^p$-spaces with $1<p<\infty$.

\medskip
Burkholder~\cite[\S 10--12]{burkholder73} points out that it is the impact of big jumps which causes the breakdown of the BDG inequalities if $p\in (0,1)$. This paper also contains an interesting discussion on how to tame the jumps. The following example is due to Marcinkiewicz \& Zygmund~\cite{mar-zyg38}, see also Burkholder \& Gundy~\cite{bur-gun70}; we follow the presentation in Schilling~\cite[Bsp.~10.5]{schilling-maps}.

\begin{example}[Marcinkiewicz \& Zygmund]\label{mar-72} \index{inequality!BDG, counterexample}
    In order to illustrate the breakdown of the BDG inequalities, it is enough to consider a discrete-time martingale $(X_n)_{n\in\nat_0}$; the analogue of the jumps are the increments $\xi_n \coloneqq  X_n - X_{n-1}$.

    Let $\eta$ be a random variable with law $\Pp(\eta = 1)=1-1/m$ and $\Pp(\eta = 1-m) = 1/m$ for some fixed $m\geq 2$. We construct sequences of independent random variables $(\xi_n)_{n\in\nat}$ such that the inequalities \eqref{mar-e74} fail for the \enquote{random walk} martingale $X_0\coloneqq 0$, $X_n \coloneqq  \xi_1+\dots+\xi_n$ if $0<p<1$.

    \subparagraph{Breakdown of the lower bound.}
    Let $\xi_1\sim \eta$, $\xi_2\sim-\eta$ and $\xi_n = 0$ for $n\geq 3$ be independent random variables. For $n\geq 2$
    \begin{align*}
        \Ee\left(|X_n|^p\right)
        = \Ee\left(|\xi_1+\xi_2|^p\right)
        &= \Ee\left(|\xi_1+\xi_2|^p\I_{\{\xi_1\neq -\xi_2\}}\right)\\
        &= 2\Ee\left(|\xi_1+\xi_2|^p \I_{\{\xi_1=1\}\cap\{\xi_2 = m-1\}}\right)\\
        &\leq 2 m^{p} \,\Pp(\xi_2=m-1) = 2m^{p-1}.
    \end{align*}
    However, we always have $\xi_1^2+\xi_2^2 \geq 2$, and so
    \begin{gather*}
        \Ee\left(\sbracket X_\infty^{p/2}\right)
        = \Ee\left(\left(\xi_1^2+\xi_2^2\right)^{p/2}\right)
        \geq 1.
    \end{gather*}
    If $m\gg 1$ is large enough, the lower bound in \eqref{mar-e74} is violated.

    \subparagraph{Breakdown of the upper bound.}
    For $k,n\in\nat$ let $\xi_1\sim \eta,\dots, \xi_n\sim\eta$ and $\xi_{n+k} = 0$ be independent random variables. If $m\geq 2n$, then
    \begin{gather*}
            |X_{n+k}|
            = \left|\sum_{i=1}^n \xi_i\right|
            \geq
            \begin{cases}
            n, &\forall i\in \{1,\dots, n\} : \xi_i = 1,\\[5pt]
            \overbracket[.6pt]{|\xi_{j}| - \sum\nolimits_{i\neq j}|\xi_i|}^{\geq (m-1)-(n-1)}\geq n,
            &\exists j\in \{1,\dots, n\} : \xi_{j} =1-m.
            \end{cases}
    \end{gather*}
    Moreover, $\sbracket X_{n+k} = \sum_{i=1}^n\xi_i^2$ satisfies
    \begin{gather*}
            \Pp\left(\sbracket X_{n+k} = n\right) = \left(1-\tfrac 1m\right)^n
            \et
            \Pp\left(n < \sbracket X_{n+k} \leq nm^2\right) = 1-\left(1-\tfrac 1m\right)^n.
    \end{gather*}
    On the one hand, we have $\Ee\left(|X_{n+k}|^p\right)\geq n^p$, and on the other hand
    \begin{gather*}
            \Ee\left(\sbracket X_{n+k}^{p/2}\right)
            \leq n^{p/2}\left(1-\tfrac 1m\right)^n +  n^{p/2} m^p\left(1-\left(1-\tfrac 1m\right)^n\right)
            \xrightarrow[]{m\to\infty} n^{p/2}.
    \end{gather*}
    Thus, the upper bound in \eqref{mar-e74} cannot hold if $m,n\gg 1$.
\end{example}

\begin{remark}\label{mar-73}
    Let us briefly indicate an extension of the BDG inequalities which is contained in Lenglart \emph{et al.}~\cite{len-lep-pra80} and Meyer~\cite{dm-5}. Rather than considering the moment $F(X_t^*)$ for the polynomially growing function $F(x) = |x|^p$, $p\geq 1$, one may consider convex functions $F\colon [0,\infty)\to[0,\infty)$ which are of \emph{moderate growth}, i.e.\ functions that satisfy \index{moderate growth}
    \begin{gather*}
        \exists c>0 \::\: \sup_{x>0} \frac{F(cx)}{F(x)} < \infty.
    \end{gather*}
    All concave functions are moderate. If $F$ is convex with right derivative $f = F'_+$, then the above condition is equivalent to
    \begin{gather*}
        \alpha = \sup_{x>0} \frac{xf(x)}{F(x)} < \infty
        \quad\text{which implies that}\quad
        \forall c>1\::\: \sup_{x>0} \frac{F(cx)}{F(x)} \leq c^\alpha.
    \end{gather*}
    The BDG inequalities for convex moderate functions read
    \begin{gather*}
        \Ee\left[F(X_t^*)\right] \asymp \Ee\left[F\left(\sqrt{\sbracket{X}_t}\right)\right],\quad t\geq 0.
    \end{gather*}
    The proof relies,
     essentially, on the good-$\lambda$-inequality and a clever use of the layer-cake formula for moderate functions.
\end{remark}

Using the optimization trick from Remark~\ref{mar-42} it is easy to deduce exponential tail estimates from the BDG inequalities. The restriction $\alpha\lambda^2 \geq d^2(t)$ comes from the fact that the BDG inequalities hold, in general, only for $p\geq 1$. The following result is due to Kallenberg \& Sztencel~\cite[Thm.~5.1, 5.3]{kal-szt91}.\footnote{Note that Kallenberg \& Sztencel~\cite{kal-szt91} treat high-dimensional martingales and aim to get uniform constants for all dimensions. Thus, the proofs are much more complicated, and the resulting estimates are \enquote{only} sub-Gaussian or require further symmetry assumptions to get Gaussian tails.}
\begin{corollary}\label{mar-77}
    Let $(X_t)_{t\geq 0}$ be a real-valued c\`adl\`ag martingale such that the square bracket process is bounded, i.e.\ $\sbracket X_t \leq d^2(t)$ for some increasing function $d(t)>0$. Then
    \begin{gather}\label{mar-e78}
        \Pp\left( X_t^* > \lambda \right) \leq \frac 2\beta e^{-\alpha \lambda^2/d^2(t)},\quad \lambda > \frac{d^2(t)}{\alpha},\: \alpha+\beta = 1,\: \alpha,\beta\in (0,1).
    \end{gather}
\end{corollary}

\section{L\'evy Processes}\label{lp}

Brownian motion is arguably one of the most important stochastic processes in continuous time. The class of L\'evy processes shares many features with Brownian motion, and it is reasonable to treat Brownian motion in the wider context of L\'evy processes. In many cases, \enquote{Brownian proofs} extend to the L\'evy setting almost literally, and in such cases we sometimes point out \enquote{Brownian} references, leaving the obvious adaptations to the reader. Below we often work in dimension $d=1$. Let us point out, however, that most results below remain valid for $d$-dimensional processes at the expense of larger, usually dimension-dependent, constants.

\subsection{Background Knowledge on L\'evy Processes}\label{lp-back}

\begin{definition}\label{lp-03} \index{Levy process@L\'evy process}
    A \emph{L\'evy process} is a stochastic process $(X_t)_{t\geq 0}$ taking values in $\rd$ and such that its increments are stationary \eqref{Lstat}, independent \eqref{Lindep} and the trajectories are continuous in probability \eqref{Lcont}:
\begin{align}
    \forall s\leq t                             &:&& X_t-X_s \sim X_{t-s} \tag{L1}\label{Lstat}\\
    \forall n\in\nat,\: t_0\leq \dots \leq t_n  &:&& (X_{t_k}-X_{t_{k-1}})_{k=1}^n\text{\ \ independent} \tag{L2}\label{Lindep}\\
    \forall \epsilon>0                          &:&& \lim_{h\downarrow 0} \Pp\left(|X_h| >\epsilon\right) = 0 \tag{L3}\label{Lcont}
\end{align}
\end{definition}
Our standard reference for L\'evy processes is Sato's monograph~\cite{sato99}, a short introduction is given in Schilling~\cite{crm}.

Note that \eqref{Lstat} implies that $X_0\sim \delta_0$, i.e.\ $\Pp(X_0 = 0)=1$; this is often separately stated as a further defining condition. Let us add a few further remarks.

\begin{remark}[role of Brownian motion]\label{lp-05}\index{Brownian motion}%
    Among all L\'evy processes, Brownian motion is special for
    (i) its law is Gaussian,
    (ii) its paths $t\mapsto B_t$ are (a.s.) continuous, and
    (iii) its infinitesimal generator is a second-order differential operator with constant coefficients.
    In fact, the properties (i)--(iii) are equivalent. The implication
    (i)$\Rightarrow$(ii) is the Kolmogorov--Chentsov theorem (e.g.\ Revuz \& Yor~\cite[Thm.~I.(2,1)]{rev-yor99}, Schilling~\cite[Thm.~10.1]{schilling-bm} or Theorem~\ref{gen-12}),
    (ii)$\Rightarrow$(i) is due to L\'evy (e.g.\ It\^o~\cite[Sec.~4.\S2]{ito-tata}, Schilling~\cite[Thm.~8.4]{crm}, it is essentially a uniform version of the central limit theorem),
    for (ii)$\Rightarrow$(iii) see e.g.\ Schilling~\cite[Thm.~7.38]{schilling-bm},
    and (iii)$\Rightarrow$(i) is a standard fact on the fundamental solution of 2nd order differential operators -- just use the Fourier transform.
\end{remark}

\begin{remark}[filtration and measurability issues]\label{lp-07}
    In the definition of a L\'evy process
    we have implicitly used the natural filtration $\Fscr_t^X \coloneqq  \sigma(X_s, s\leq t)$ generated by the process itself. It is not hard to see that \eqref{Lindep} is equivalent to
    \begin{gather}
    \tag{L2$^\prime$}\label{Lindep-bis}
    \forall s\leq t \::\: X_t-X_s \text{\ \ is independent of\ \ } \Fscr_s^X.
    \end{gather}
    If we want to regularize the sample paths, it is important to add null sets to the filtration. To do so, we complete the underlying probability space (if needed -- but we do not change notation) and denote the family of $\Pp$-null sets in the complete space by $\Nscr$. The filtration is then enlarged by setting $\Fscr_t \coloneqq  \sigma(\Fscr_t^X, \Nscr)$. Since we are adding only null sets and sets of full measure, it is clear that \eqref{Lindep-bis}, hence \eqref{Lindep}, remains valid; thus $(X_t,\Fscr_t)_{t\geq 0}$ is indeed a \emph{L\'evy process with filtration}. One can show that this enlargement is right continuous, i.e.\ $\Fscr_t = \Fscr_{t+}$, see e.g.\ Protter~\cite[Thm.~I.31]{protter}, Revuz \& Yor~\cite[Thm.~III.(2.10)]{rev-yor99} or Schilling~\cite[Ch.~6.7]{schilling-bm}; thus, $(X_t)_{t\geq 0}$ can be defined on a filtered probability space $(\Omega,\Ascr,\Fscr_t,\Pp)$ which fulfils the \emph{usual conditions}.

    In this setting, we can construct a modification of $(X_t,\Fscr_t)_{t\geq 0}$ such that for all $\omega$ one has $X_0(\omega)=0$ and \eqref{Lcont} is (up to a modification) equivalent to
    \begin{gather}
    \tag{L3$^\prime$}\label{Lcont-bis}
    \forall \omega\in\Omega\::\: t\mapsto X_t \text{\ \ is right continuous with finite left limits (c\`adl\`ag)}.
    \end{gather}
    We refer to Revuz \& Yor~\cite[Thm.~III.(2.7)]{rev-yor99} for the proof.
\end{remark}

\begin{remark}[strong Markov property]\label{lp-08} \index{Levy process@L\'evy process!strong Markov property}
    The independent and stationary increments property implies that the process $Y_t \coloneqq  X_{t+s}-X_s$ is, for each fixed $s\geq 0$, again a L\'evy process, which is independent of $\Fscr_s$, and that $(Y_t)_{t\geq 0}$ and $(X_t)_{t\geq 0}$ have the same finite dimensional distributions.

    The independence of $(Y_t)_{t\geq 0}$ from the past $\Fscr_s$ allows us to realize a L\'evy process as a Markov process. To do so, we define probability measures $\Pp^x$, $x\in\rd$, on $(\Omega,\Fscr_\infty)$ as projective limit of the finite dimensional distributions
    \begin{gather*}
        \Pp^x\left(X_{t_k}\in B_k, k=1,\dots, n\right)\coloneqq  \Pp\left(X_{t_k}+x\in B_k, k=1,\dots, n\right),\\
        0\leq t_1<\dots < t_n,\:n\in\nat,\:B_1,\dots,B_n\in\Bscr(\rd).
    \end{gather*}
    Then $\Pp^x$ and $\Pp^\mu(\dots) \coloneqq  \int \Pp^x(\dots)\,\mu(dx)$ are the laws of the process $(X_t)_{t\geq 0}$ with initial distribution $X_0\sim\delta_x$ and $X_0\sim\mu$, respectively; we frequently use $\Pp$ instead of $\Pp^0$.

    In the last two paragraphs we may replace the fixed time $s\geq 0$ by an $\Fscr_t$-stopping time $\sigma$:
    \begin{gather}\label{lp-e06}\begin{gathered}
        (X_{t+\sigma}-X_\sigma)_{t\geq 0} \text{\ \ is, conditional on $\{\sigma<\infty\}$, a L\'evy process},\\
        \text{which is independent of $\Fscr_\sigma$ and distributed like $(X_t)_{t\geq 0}$}.
    \end{gathered}\end{gather}
    The property \eqref{lp-e06} is often called the \emph{strong Markov property} of a L\'evy process and, using standard arguments, one can see that it implies the usual formulation of the strong Markov property
    \begin{gather*}
        \Ee^x\left[f(X_{t+\sigma}) \mid \Fscr_{\sigma} \right]
        = \Ee^{X_\sigma}\left[f(X_t)\right]
        \quad\text{$\Pp^x$-a.s.\ on $\{\sigma<\infty\}$ for $f\in B_b(\rd)$}
    \end{gather*}
    see e.g.\ Schilling~\cite[Ch.~6.2]{schilling-bm}.
\end{remark}

An elegant way to characterize L\'evy processes can be based on the conditional characteristic function, see Schilling~\cite[Thm.~3.1]{crm}.
\begin{theorem}\label{lp-09} \index{Levy process@L\'evy process!characterization}
    A stochastic process $(X_t,\Fscr_t)_{t\geq 0}$ with values in $\rd$ and c\`adl\`ag paths is a L\'evy process with filtration $(\Fscr_t)_{t\geq 0}$ if, and only if,
    \begin{gather}\label{lp-e08}
        \Ee\left[e^{i\xi\cdot (X_t-X_s)} \:\middle|\: \Fscr_s\right]
        = e^{-(t-s)\psi(\xi)},\quad \xi\in\rd,\; s\leq t,
    \end{gather} \index{characteristic exponent}\index{Levy--Khintchine formula@L\'evy--Khintchine formula}\index{Levy triplet@L\'evy triplet}
    with a \emph{characteristic exponent} $\psi\colon \rd\to\comp$ which is uniquely determined by the \emph{L\'evy--Khintchine formula}
    \begin{gather}\label{lp-e10}\begin{aligned}
        \psi(\xi)
        &= -ib\cdot\xi + \frac 12\xi\cdot Q\xi
        + \int_{|y|\geq 1} \left(1-e^{i\xi\cdot y}\right) \nu(dy)\\
        &\qquad\mbox{} + \int_{0<|y|<1} \left(1-e^{i\xi\cdot y} + i\xi\cdot y\right) \nu(dy).
    \end{aligned}\end{gather}
    The \emph{L\'evy triplet} $(b,Q,\nu)$ appearing in \eqref{lp-e10} is unique and it comprises a vector $b\in\rd$, a positive semidefinite matrix $Q\in\real^{d\times d}$ and a Radon measure $\nu$ on $\rd\setminus\{0\}$ such that $\int_{y\neq 0} |y|^2\wedge 1\,\nu(dy)<\infty$.
\end{theorem}

If we take expectations on both sides of the formula \eqref{lp-e08}, we see immediately that the random variables $X_t$ are infinitely divisible, see also the discussion at the beginning of Section~\ref{pre-mom}. In fact, and this can also be deduced from \eqref{lp-e08} and the structure of infinitely divisible laws, each infinitely divisible random variable $Z$ can be \enquote{embedded} into a L\'evy process $(X_t)_{t\geq 0}$ such that $X_1\sim Z$. For further details we refer to Sato~\cite[Ch.~2.\S7--8]{sato99} and Schilling~\cite[Ch.~2, Ch.~7]{crm}.

Since a L\'evy process is a translation invariant Markov process, the transition semigroup is given by $P_t f(x) = \Ee^x f(X_t) = \Ee f(X_t+x)$. The next theorem explains the structure of the infinitesimal generator $A$ of the semigroup $P_t = e^{tA}$ and the structure of the sample paths. Proofs can be found in Sato~\cite[\S19--20 and Thm.~31.5]{sato99} and Schilling~\cite[Ch.~6--7]{crm}.

We will use $\widehat f(\xi) \coloneqq  (2\pi)^{-d}\int_{\rd} e^{-i\xi\cdot y} f(y)\,dy$ to denote the Fourier transform of the function $f$.
\begin{theorem}\label{lp-11}
    Let $(X_t,\Fscr_t)_{t\geq 0}$ be a L\'evy process with values in $\rd$ and characteristic exponent $\psi$ given by \eqref{lp-e10}. The infinitesimal generator of the process is a pseudo differential operator
    \begin{align}\label{lp-e12}
        Af(x) &= -\psi(D)f(x) = -\int_{\rd} \psi(\xi) \widehat f(\xi)\,d\xi,\quad f\in C_c^\infty(\rd),
    \intertext{which can also be written as an integro-differential operator}\label{lp-e14}
        Af(x) &= b\cdot\nabla f(x) + \frac{1}{2} \tr\left( Q\cdot \nabla^2 f(x) \right)
	           + \int_{|y|\geq 1} \left(f(x+y) -f(x)\right) \nu(dy)\\
                &\quad\notag
	           \mbox{}+ \int_{0<|y|<1} \left(f(x+y) -f(x)- y\cdot \nabla f(x)\right) \nu(dy),\quad f\in C_c^2(\rd).
    \end{align}
    The sample paths of a L\'evy process enjoy the following \emph{L\'evy-It\^o decomposition}
    \begin{align}\label{lp-e16}
        X_t
        = bt + \sqrt Q W_t
        + \sum_{\substack{0<s\leq t\\\mathclap{|\Delta X_s|\geq 1}}} \Delta X_s
        + \lim_{\epsilon\downarrow 0}\Bigg[\quad\sum_{\substack{0<s\leq t\\\mathclap{\epsilon\leq |\Delta X_s|<1}}} \Delta X_s - t\!\!\!\int\limits_{\epsilon\leq|y|<1}\!\!\! y\,\nu(dy)\Bigg]
    \end{align}
    where the four terms are L\'evy processes, which are independent and whose L\'evy exponents are given by the four terms in the representation \eqref{lp-e10} of $\psi$. The limit appearing in \eqref{lp-e16} is a limit in probability.
\end{theorem}

\subsection{Reflection-type Maximal Inequalities}\label{lp-ref}

The reflection principle is one of the first results one comes across when studying Brownian motion or random walks. It is, of course, a maximal (in-)equality and its proof is closely related to maximal inequalities for L\'evy processes.

Let $(X_t)_{t\geq 0}$ be a real-valued stochastic process. As before, $X_t^* \coloneqq  \sup_{s\leq t}|X_s|$ denotes the absolute maximum, and we  write $M_t \coloneqq  \sup_{s\leq t} X_t$ for the running maximum. By $\tau_a \coloneqq  \inf\left\{s>0 \mid X_s \geq a\right\}$ we denote the first passage time at the level $a$ (also: the first hitting time of the closed set $[a,\infty)$).

\begin{scholium}[classical reflection principle]\label{lp-31} \index{reflection principle}\index{Brownian motion}%
Let us briefly review the proof of the \emph{reflection principle} for a one-dimensional Brownian motion $(B_t)_{t\geq 0}$. For any $a>B_0=0$ and $t>0$ we have $\{\tau_a \leq t\} = \{M_t \geq a\}$ and, using the strong Markov property,
\begin{align*}
    \Pp^0\left(M_t \geq a\right)
    &= \Pp^0\left(M_t \geq a,\: B_t< a\right) + \Pp^0\left(M_t \geq a,\: B_t \geq a\right)\\
    &= \Pp^0\left(M_t \geq a,\: B_t< a\right) + \Pp^0\left(B_t \geq a\right)\\
    &= \Pp^0\left(\tau_a \leq t,\: B_t< a\right) + \Pp^0\left(B_t \geq a\right)\\
    &= \int_{\tau_a \leq t} \Pp^{B_{\tau_a}(\omega)}\left(B_{t-\tau_a(\omega)} < a\right) \Pp^0(d\omega) + \Pp^0\left(B_t \geq a\right).
\end{align*}
Since Brownian motion is invariant under translations, and $B_{\tau_a}=a$, we get
\begin{gather}\label{lp-e30}\begin{aligned}
    \Pp^{B_{\tau_a}(\omega)}\left(B_{t-\tau_a(\omega)} < a\right)
    &= \Pp^{0}\left(B_{t-\tau_a(\omega)} < a-B_{\tau_a}(\omega)\right)\\
    &= \Pp^{0}\left(B_{t-\tau_a(\omega)} < 0\right)
    = \frac 12
\end{aligned}\end{gather}
and, using again $\{\tau_a \leq t\} = \{M_t \geq a\}$, we see that
\begin{gather}\label{lp-e32}
    \Pp^0\left(M_t \geq a\right)
    = 2\Pp^0\left(B_t \geq a\right)
    = \Pp^0\left(|B_t| \geq a\right),\quad a>0.
\end{gather}
In the second equality we use the symmetry of Brownian motion: $B_t\sim -B_t$.
\end{scholium}

Recall that a stochastic process $(X_t)_{t\geq 0}$ is called \emph{self-similar} if for any $a>0$ there is some $b>0$ such that the processes $(X_{at})_{t\geq 0}$ and $(bX_t)_{t\geq 0}$ have the same finite dimensional distributions. In particular, $\Pp(X_t \geq 0) = \Pp(X_1 \geq 0)=p_0$ does not depend on $t>0$. A L\'evy process $(X_t)_{t\geq 0}$ is said to be \emph{strictly stable}, if each random variable $X_t$ is stable, i.e.\ if $X_t \sim b(t)X_1$ for a suitable $b(t)>0$. For L\'evy processes the notions of \emph{self-similarity} and \emph{strict stability} coincide and, in this case, $b(t) = t^{1/\alpha}$ for some \enquote{stability index} $\alpha\in (0,2]$, see Schilling~\cite[Ch.~3]{crm} and Sato~\cite[Ch.~3.\S13]{sato99}. The following observation is from Meerschaert \emph{et al.}~\cite[Thm.~6.1]{mss15}.
\begin{lemma}\label{lp-33} \index{reflection principle}
    Let $(X_t)_{t\geq 0}$ be a strictly stable \textup{(}hence, self-similar\textup{)} L\'evy process which is spectrally negative, i.e.\ it does not have positive jumps. Set $p_0 \coloneqq  \Pp(X_1 \geq 0)$. Then
    \begin{gather}\label{lp-e34}
        \Pp^0\left(M_t \geq a\right) = \frac{1}{1-p_0}\Pp^0\left(X_t \geq a\right),\quad a>0.
    \end{gather}
\end{lemma}
\begin{proof}
    We can follow the proof of the reflection principle for Brownian motion: First $\{M_t \geq a\}=\{\tau_a \leq t\}$  and $X_{\tau_a}=a$ follows from the fact that $(X_t)_{t\geq 0}$ does not have positive jumps,  i.e.\ it can cross any level $a$, coming from below, only by \enquote{creeping}. Thus we only have to pay attention to \eqref{lp-e30}. Due to self-similarity, we have now $\Pp(X_{t-\tau_a(\omega)}  <  0) = \Pp(X_1 < 0) = 1-p_0$ instead of $\frac 12$.\footnote{Let us point out that, without self-similarity, \eqref{lp-e34} remains true, if we use $\Pp^0(X_t \geq 0)$ instead of $1-p_0 = \Pp^0(X_1 \geq 0)$.}
\end{proof}

The proof of the reflection principle (Scholium~\ref{lp-31}) can also be tweaked to give L\'evy's inequalities. These inequalities were originally shown by L\'evy~\cite[Ch.~VI.\S44]{levy37}, see also Lo\`eve~\cite[Vol.~1, Sec.~18.C]{loeve77}, when studying the convergence of sums of independent random variables. Here we present a variant for symmetric L\'evy processes, i.e.\ a L\'evy process with symmetric probability distributions: $X_t\sim -X_t$ for all $t>0$.
\begin{lemma}\label{lp-35} \index{reflection principle}
    Let $(X_t)_{t\geq 0}$ be a symmetric real-valued L\'evy process. Then
    \begin{gather}\label{lp-e36}
        \Pp^0\left(M_t > a\right) \leq 2\Pp^0\left(X_t \geq a\right)
    \et
        \Pp^0\left(X_t^* > a\right) \leq 2\Pp^0\left(|X_t| \geq a\right).
    \end{gather}
\end{lemma}

\begin{proof}
    Since we still have $\{M_t > a\} \subseteq \{\tau_a \leq t\}$, we can follow the proof in Scholium~\ref{lp-31} up to \eqref{lp-e30}. Since $X_{\tau_a}\geq a$, this identity now becomes an inequality:
    \begin{gather*}
        \Pp^{X_{\tau_a}(\omega)}\left(X_{t-\tau_a(\omega)} \leq a\right)
        = \Pp^{0}\left(X_{t-\tau_a(\omega)} \leq \smash[b]{\underbracket[.6pt]{a-X_{\tau_a}(\omega)}_{\leq 0}}\right)
        \leq \frac 12
    \end{gather*}
    and we get the first part of \eqref{lp-e36}. Using symmetry, we have
    \begin{align*}
        \Pp^0\left(X_t^* > a\right)
        \leq \Pp^0\left(\sup_{s\leq t} X_s > a\right) + \Pp^0\left(\sup_{s\leq t} (-X_s) > a\right),
    \end{align*}
    and $2\Pp^0\left(X_t \geq a\right) = \Pp^0\left(|X_t| \geq a\right)$; this reduces things to the first estimate.
\end{proof}

\begin{remark}\label{lp-37}
    Lemma~\ref{lp-35} holds, of course, for Brownian motion. In this case, the first inequality in \eqref{lp-e36} is actually an equality. To get an exact expression for $\Pp^0\left(B_t^* > a\right)$ one needs more care. Using L\'evy's \enquote{triple law}
    (see Schilling~\cite[Thm.~6.19 \& p.~207]{schilling-bm} -- this is also a \enquote{reflection} result) one can show that
    \begin{gather*}
        \Pp^0(B_t^* > a)
        = 1- \frac 4\pi\sum_{k=0}^\infty \frac{(-1)^k}{2k+1}\exp\left(-\frac{\pi^2(2k+1)^2}{8a^2}\right).
    \end{gather*}
\end{remark}

It is possible, as L\'evy did, to avoid the symmetry assumption in Lemma~\ref{lp-35} by a centring argument, e.g.\ at the mean value or the median of $X_t$. There is, however, a much more powerful (but little known) argument due to Etemadi~\cite{etemadi85} and, independently, Kwapie\'n \& Woyczy\'nski~\cite[Prop.~1.1.1, Prop.~8.2.1]{kwa-woy92}; we refer to see Billingsley~\cite[Thm.~22.5]{billingsley95} or Schilling~\cite[Thm.~33.1]{schilling-dc} for a textbook presentation of the discrete case. The following proof in the continuous-time setting combines Etemadi's original idea with the proof of Scholium~\ref{lp-31}.
\begin{theorem}[Etemadi]\label{lp-39} \index{inequality!Etemadi}
    Let $(X_t)_{t\geq 0}$ be a real-valued L\'evy process. Then, for all $a>0$, $t>0$
    \begin{gather}\label{lp-e40}
        \Pp^0\left(M_t > 3a\right)
        \leq \Pp^0\left(|X_t|\geq a\right) + \sup_{s\leq t}\Pp^0\left(X_s > a\right),
    \intertext{as well as}\label{lp-e42}
        \Pp^0\left(X_t^* > 3a\right)
        \leq 3\sup_{s\leq t}\Pp^0\left(|X_s| \geq a\right).
    \end{gather}
\end{theorem}
\begin{proof}
    Set $\tau = \tau_{3a} = \inf\left\{s>0 \mid X_s > 3a\right\}$ and use $\{\tau_{3a}\leq t\} \supseteq \{M_t > 3a\}$ along with the strong Markov property to get
    \begin{align*}
        \Pp^0\left(M_t > 3a\right)
        &\leq \Pp^0\left(\tau\leq t,\: X_t < a\right) + \Pp^0\left(X_t\geq a\right)\\
        &= \int_{\tau\leq t} \Pp^{X_{\tau}(\omega)}\left(X_{t-\tau(\omega)} < a\right) \Pp^0(d\omega) + \Pp^0\left(X_t\geq a\right)
    \end{align*}
    Since $X_{\tau} \geq 3a$, we can use translation invariance to get
    \begin{align*}
        \Pp^{X_{\tau}(\omega)}\left(X_{t-\tau(\omega)} < a\right)
        &= \Pp^{0}\left(X_{t-\tau(\omega)}< a-X_{\tau}(\omega) \right)\\
        &\leq \Pp^{0}\left(X_{t-\tau(\omega)} < -2a\right)\\
        &\leq \sup_{s\leq t}\Pp^{0}\left(X_{t-s} < -2a\right).
    \end{align*}
    Since $X_{t-s}\sim X_t - X_s$ and $\{X_t-X_s < -2a\}\subseteq \{X_t < -a\}\cup\{-X_s < -a\}$ we see
    \begin{align*}
        \Pp^{X_{\tau}(\omega)}\left(X_{t-\tau(\omega)} < a\right)
        \leq \Pp^0\left(X_t < -a\right) + \sup_{s\leq t}\Pp^{0}\left(X_s > a\right).
    \end{align*}
    Inserting this estimate into the integral formula which we got from the strong Markov property yields \eqref{lp-e40}.

    The inequality \eqref{lp-e42} follows along the same lines, we just have to use the stopping time $\sigma_{3a}=\inf\left\{s>0 \mid |X_s|>3a\right\}$ instead of $\tau_{3a}$. The slightly worse factor $3$ comes from the triangle inequality when dealing with $|X_{t-s}|\sim |X_t - X_s|$ in the last step of the argument.
\end{proof}

\begin{remark}\label{lp-41} \index{inequality!Ottaviani--Skorokhod}
    A close inspection of the proof of Theorem~\ref{lp-39} shows that it actually yields estimates of the type
    \begin{gather*}
        \Pp^0\left(X_t^* > 3a\right)
        \leq \frac{\Pp^0\left(|X_t|\geq a\right)}{1-\sup\limits_{s\leq t}\Pp^0\left(|X_{t-s}| > 2a\right)}
        \leq \frac{\Pp^0\left(|X_t|\geq a\right)}{\Pp^0\left(|X_t|\leq  a\right)-\sup\limits_{s\leq t}\Pp^0\left(|X_s| > a\right)}.
    \end{gather*}
    Such inequalities -- often stated with $a+2b$ instead of $3a$, causing some obvious changes -- are also called \emph{Ottaviani--Skorohkod inequalities} (Billingsley~\cite[p.~296]{billingsley95}) or \emph{generalized Kolmogorov inequalities} (Skorokhod~\cite[I.\S1.1]{skorokhod91}). See also the notes in Billingsley where the potency of the Etemadi and the Ottaviani--Skorokhod inequalities is compared: as one might expect from our proofs, they do have similar power.
\end{remark}

In a series of papers starting in the 1980s Etemadi demonstrated how the maximal estimate \eqref{lp-e42} can be used to replace Kolmogorov's maximal estimate for centred partial sums, resulting in shorter and often more general proofs of classical results. We will use now Etemadi's estimate and some of our considerations on censored moments and concentration functions from Sections~\ref{pre-mom}, \ref{pre-con} to give a short proof of Pruitt's maximal estimates for L\'evy processes, Pruitt~\cite{pruitt81}. Pruitt used the truncated moments \index{censored moment} $K,G$ and $M$ from \eqref{pre-e75} and \eqref{pre-e85}, the connection with the characteristic exponent $\psi$ was pointed out in Schilling~\cite{rs-growth}, see also Theorem~\ref{pre-78}.
\begin{theorem}\label{lp-43}\index{inequality!maximal tail}\index{maximum function $\psi^*, \vert\psi\vert^*$}
    Let $(X_t)_{t\geq 0}$ be a real-valued L\'evy process with characteristic exponent $\psi$ given by \eqref{lp-e10} and denote by $|\psi|^*(\ell) = \sup_{|\xi|\leq\ell}|\psi(\xi)|$ the associated maximum function as in \eqref{pre-e73}. Then, there is an absolute constant $c>0$ such that for all $a>0$, $t>0$
    \begin{align}
    \label{lp-e46}
        \Pp^0\left(X_t^* > a\right) \leq c t|\psi|^*(a^{-1})
        \et
        \Pp^0\left(X_t^* \leq a\right) \leq \frac{c}{t |\psi|^*(a^{-1})}.
    \end{align}
\end{theorem}
\begin{proof}
    Fix $a>0$ and $t>0$. For the first estimate we begin with Etemadi's inequality:
    \begin{align*}
        \Pp\left(X_t^* > a\right)
        \leq 3 \sup_{s\leq t} \Pp\left(|X_s| > \tfrac a3\right)
        = 3 \sup_{s\leq t} \Pp\left(|X_s|\wedge a > \tfrac a3\right).
    \end{align*}
    By the Markov inequality
    \begin{align*}
        \Pp\left(|X_s|\wedge a > \tfrac a3\right)
        \leq \frac 9{a^2} \Ee\left(|X_s|^2\wedge a^2\right)
        \leq  18 \, \Ee \left[\frac{a^{-2}X_s^2}{1+a^{-2}X_s^2}\right].
    \end{align*}
    Using Lemma~\ref{pre-72} and Tonelli's theorem we see that
    \begin{align*}
        \Pp\left(|X_s| \wedge a > \tfrac a3\right)
        &\leq  18 \int \Ee \left[1-\cos \left(\tfrac \xi a X_s\right)\right] g(\xi)\,d\xi\\
        &= 18 \int \RE\left(1-e^{-s\psi\left(a^{-1}\xi\right)}\right) g(\xi)\,d\xi\\
        &\leq 18 \,s\int \left|\psi\left(a^{-1}\xi\right)\right| g(\xi)\,d\xi.
    \end{align*}
    The first inequality in \eqref{lp-e46} follows now easily from Lemma~\ref{pre-74}, see also the argument used in the proof of Theorem~\ref{pre-76}.

    We will now turn to the second inequality. If we content ourselves with $\psi^*(a^{-1})$ on the right hand side, then the argument is rather straightforward:
    \begin{align*}
        \Pp\left(X_t^* \leq a\right)
        &\leq \Pp\left(|X_{t/2}|\leq a,\: |X_t| \leq a\right)\\
        &\leq \Pp\left(|X_{t/2}|\leq a,\: |X_t - X_{t/2}| \leq 2a\right)\\
        &\leq \Pp\left(|X_{t/2}|\leq 2a\right)^2;
    \end{align*}
    in the last step we use the independent and stationary increments property of a L\'evy process. Now we can apply Esseen's estimate for L\'evy's concentration function, Theorem~\ref{pre-50} and Remark~\ref{pre-54}, to get \index{concentration function!L\'evy}
    \begin{align*}
        \Pp\left(X_t^* \leq a\right)
        \leq \Pp\left(|X_{t/2}|\leq 2a\right)^2
        \leq Q_{X_{t/2}}(4a)^2
        \leq \frac{C}{t D(\kolmo,4a)}
        \leq \frac{C'}{t \psi^*((4a)^{-1})};
    \end{align*}
    if we use Lemma~\ref{pre-74}, we can compare $\psi^*((4a)^{-1})$ and $\psi^*(a^{-1})$ to get the upper bound with $\psi^*(a^{-1})$. For many practical purposes this is enough and, quite often one has also the sector condition (Lemma~\ref{pre-92}) which allows one to exchange $\psi^*$ for $|\psi|^*$.

\bigskip
    Let us give the additional argument (due to Pruitt~\cite{pruitt81}) to get $|\psi|^*(a^{-1})$. This relies on estimates in terms of the truncated moments $K(\kolmo;a)$, $G(\nu;a)$ and, crucially, $M(b,\nu;a)$. First we note that we have already seen that
    \begin{gather*}
        \Pp\left(X_t^* \leq a\right)
        \leq \frac{C}{t D(\kolmo,4a)}
        \leq \frac{C'}{t} \left[\frac 1{K(\kolmo,4a)} \wedge \frac 1{G(\kolmo,4a)}\right].
    \end{gather*}
    To proceed, we decompose the L\'evy process into two independent processes $X_t = Y_t + J_t$, where $J_t = J_t(4a) = \sum_{s\leq t} \Delta X_s \I_{(4a,\infty)}(|\Delta X_s|)$ is the compound Poisson process consisting of all jumps larger than $4a$, and $Y_t = X_t-J_t$. The processes $(Y_t)_{t\geq 0}$ and $(J_t)_{t\geq 0}$ are independent L\'evy processes, the characteristic exponents of $Y_1$ and $J_1$ are those given in Scholium~\ref{pre-70} with $\ell=4a$. On the set $\{X_t^*\leq a\}$, there cannot happen any jumps of size $>2a$, in particular $J_t^*=0$, and we see
    \begin{align*}
        \Pp\left(X_t^*\leq a\right)
        = \Pp\left(X_t^*\leq a,\: J_t^*=0\right)
        = \Pp\left(Y_t^*\leq a,\: J_t^*=0\right)
        \leq \Pp\left(Y_t^*\leq a\right).
    \end{align*}
    Assuming $t M(b,\nu;4a)\geq 1$ or, equivalently, $|\Ee Y_t|\geq 4a$, see Scholium~\ref{pre-70}, we have
    \begin{align*}
        \Pp\left(X_t^*\leq a\right)
        \leq \Pp\left(Y_t^*\leq a\right)
        \leq \Pp\left( |Y_t| \leq 2a\right)
        \leq \Pp\left( |Y_t- \Ee Y_t| \geq \tfrac 12 |\Ee Y_t|\right).
    \end{align*}
    The last estimate uses $|Y_t - \Ee Y_t|\geq |\Ee Y_t|-|Y_t|\geq |\Ee Y_t|-2a\geq\frac 12|\Ee Y_t|$ on the set $\{|Y_t|\leq 2a\}$ and $t M(b,\nu;4a)\geq 1$. The Markov inequality and the moment formulae stated in Scholium~\ref{pre-70} yield
    \begin{gather*}
        \Pp\left(X_t^*\leq a\right)
        \leq \frac{4\Vv(Y_t)}{|\Ee Y_t|^2}
        = \frac{4\,K(\kolmo;4a)}{t M^2(b,\nu;4a)}.
    \end{gather*}
    Finally, if $tM(b,\nu;4a)\leq 1$, we trivially have
    \begin{gather*}
        \Pp(X_t^*\leq a) \leq 1 \leq \frac{1}{tM(b,\nu;4a)}.
    \end{gather*}
    Summing up we get, up to a constant $Ct^{-1}$, upper bounds of the form
    \begin{gather*}
        \frac{1}{K(\kolmo;4a)}\wedge\frac{1}{G(\nu;4a)}\wedge\frac{K(\kolmo;4a)}{M(b,\nu;4a)^2}
        \quad\text{or}\quad
        \frac{1}{K(\kolmo;4a)}\wedge\frac{1}{G(\nu;4a)}\wedge\frac{1}{M(b,\nu;4a)}.
    \end{gather*}
    Since $K(\kolmo;4a)\vee G(\nu;4a) \vee M(b,\nu;4a) \asymp |\psi|^*((4a)^{-1})\asymp |\psi|^*((a^{-1})$, see Theorem~\ref{pre-78} and Lemma~\ref{pre-74}, the estimate follows in the second case.

    The first case needs a further distinction of cases: If $K(\kolmo;4a)\leq c M(b,\nu;4a)$, we have ${K(\kolmo;4a)}/{M(b,\nu;4a)^2} \leq c/M(b,\nu;4a)$, and we are back in Case 2. If $K(\kolmo;4a) > c M(b,\nu;4a)$, we have $\Pp(X_t^*\leq a) \leq K(\kolmo;4a)^{-1} \leq c M(b,\nu;4a)^{-1}$, which puts us again in Case 2. This finishes the argument.
\end{proof}

    We can apply the technique used in the proof of Theorem~\ref{lp-43} to get maximal tail estimates for stochastic integrals driven by a L\'evy process. The following result recovers a result by Joulin~\cite{joulin07} and Gin\'e \& Marcus~\cite{gin-mar83} on stochastic integrals driven by stable L\'evy processes.

\begin{corollary}\label{lp-44} \index{inequality!maximal tail}\index{Levy-driven SDE@L\'evy-driven SDE}
    Let $(X_t)_{t\geq 0}$ be a L\'evy process with filtration $(\Fscr_t)_{t\geq 0}$ and $(H_t)_{t\geq 0}$ be an $\Fscr_t$-adapted c\`adl\`ag process such that $\int_0^t \Ee\left(H_s^2\right) ds<\infty$ for all $t\geq 0$. There exists a constant $c\in (0,\infty)$ such that for every $a,\ell>0$ and $t\geq 0$ the following tail estimate for the stochastic integral $H\bullet X_t = \int_0^t H_{s-}\,dX_s$ holds:
    \begin{gather}\label{lp-e47}
        \Pp\left( (H\bullet X)^*_t > \ell\right)
        \leq c \left(t + \frac{a^2}{\ell^2}t + \frac{a^2}{\ell^2}\int_0^t \Ee \left(H_s^2\right) ds\right)  |\psi|^*(a^{-1}).
    \end{gather}
\end{corollary}

\begin{proof}
    Fix $a>0$ and write $X_t = \left(Y_t - \Ee Y_t\right) + \left(J_t + \Ee Y_t\right)$ where $J_t = J_t(a) = \sum_{s\leq t} \Delta X_s \I_{(2a,\infty)}(|\Delta X_s|)$ is the L\'evy process containing all jumps of $X_t$ whose absolute value is larger than $2a$. The process $\bar Y_t \coloneqq  Y_t - \Ee Y_t$ is both a L\'evy process and a square-integrable martingale.

    In view of our assumptions, the stochastic integral $H\bullet X_t = \int_0^t H_{s-}\,dX_s$ is well-defined. We will now estimate the tail of $(H\bullet X)^*_t$:
    \begin{align*}
        \Pp\left( (H\bullet X)^*_t > \ell\right)
        &\leq \Pp\left( (H\bullet X)^*_t > \ell,\: X_t^* \leq a\right) +  \Pp\left( X_t^* > a\right)\\
        &\leq \Pp\left( (H\bullet Y)^*_t > \ell,\: X_t^* \leq a\right) +  c t |\psi|^*(a^{-1}).
    \end{align*}
    The last estimate uses that $X_s=Y_s$ for all $s\leq t$ on the set $\{X_t^*\leq a\}$, since the latter assumption excludes all jumps larger than $2a$; we have also used Theorem~\ref{lp-43} to estimate the tail of $X_t^*$. Denote the first expression on the right-hand side by $P$.
    Now we use the Markov inequality and the elementary estimate $(A+B)^2\leq 2A^2+2B^2$ to get for any $\ell > 0$
    \begin{align*}
        P &\leq \Pp\left( (H\bullet Y)^*_t > \ell\right)\\
        &\leq \frac 2{\ell^2} \Ee\left( \left[(H\bullet \bar Y)^*_t\right]^2 \right) + \frac{2}{\ell^2} \left(\Ee Y_1\right)^2 \Ee \left( \smash[t]{\left[\int_0^t H_s\,ds\right]^2}\right).
    \end{align*}
    Applying Doob's maximal inequality (Theorem~\ref{mar-09} with $p=2$), It\^o's isometry -- recall that $\sbracket{\bar Y}_t = \Ee\left[\bar Y_t^2\right] = t \Ee\left[\bar Y_1^2\right]$ -- and Jensen's inequality, yields
    \begin{align*}
        P&\leq \frac{8}{\ell^2} \Ee\left( \int_0^t H^2_s \,d\sbracket{\bar Y}_s \right) + \frac{2t}{\ell^2} \left(\Ee Y_1\right)^2 \Ee \left(\int_0^t H_s^2\,ds\right)\\
        &= \frac{8}{\ell^2} \Ee\left( \int_0^t H^2_s\,ds\right) \Ee\left(\bar Y_1^2\right)  + \frac{2t}{\ell^2} \left(\Ee Y_1\right)^2 \Ee \left(\int_0^t H_s^2\, ds\right)\\
        &= \frac{8}{\ell^2} \left( \Ee\left(\bar Y_1^2\right)  +  t \left(\Ee Y_1\right)^2\right) \int_0^t \Ee \left(H_s^2\right) ds.
    \end{align*}

    We can now apply Theorem~\ref{pre-78} to bound the expressions $\ell^{-2}\Ee\left(\bar Y_1^2\right)$ and $\ell^{-2} \left(\Ee Y_1\right)^2$ by
    $\ell^{-2} a^{2}|\psi|^*((2a)^{-1})$. Finally, an application of Lemma~\ref{pre-74} completes the proof, yielding \eqref{lp-e47}.
\end{proof}

\begin{remark}\label{lp-44-a}
 There are quite a few variations possible in Corollary~\ref{lp-44}.

\smallskip a)
    If $(X_t)_{t\geq 0}$ is symmetric, $\Ee Y_1=0$, and the term $a^2 \ell^{-2} t$ inside the parentheses on the right-hand side of \eqref{lp-e47} can be omitted.

\smallskip b)
    If $(X_t)_{t\geq 0}$ is a symmetric $\alpha$-stable L\'evy process, then $|\psi|^*(a^{-1}) = a^{-\alpha}$; in particular, \eqref{lp-e47} becomes estimate of Joulin~\cite[Eq.~(2.11)]{joulin07} and Gin\'e \& Marcus~\cite{gin-mar83}, if we use $a \coloneqq  \ell \big/ \int_0^t \Ee(H_s)\,ds$.

\smallskip c)
    We can obtain $L^p$-estimates, if we use the Markov inequality with an $L^p$-norm (rather than an $L^2$-norm) when estimating the probability $P$, and combine this with the Burkholder--Davis--Gundy inequalities (Theorem~\ref{mar-71}). Further variants can be found in Joulin~\cite{joulin07}.
\end{remark}

A typical application of the maximal estimates of Theorem~\ref{lp-43} are Pruitt's estimates \cite{pruitt81} for the speed of a L\'evy process.  Recall that we have for the stopping time $\sigma_r \coloneqq  \inf\left\{t>0 \mid |X_t|>r\right\}$, $r>0$, and $t>0$ the inclusions
\begin{gather*}
     \left\{X_t^* < r\right\}\subseteq \left\{\sigma_r > t\right\} \subseteq \left\{X_t^* \leq r\right\} \subseteq \left\{\sigma_r \geq t\right\}.
\end{gather*}

\begin{corollary}\label{lp-45} \index{inequality!mean exit time}
    Let $(X_t)_{t\geq 0}$ be a real-valued L\'evy process and denote by $\sigma_r$ the first exit time from the open ball \textup{(}interval\textup{)} $\ball{r}{0}$. Then, there are absolute constants $c,c'>0$ such that
    \begin{align}\label{lp-e48}
        \frac{c}{|\psi|^*(r^{-1})}
        \leq \Ee^0\left[\sigma_r\right]
        \leq \frac{c'}{|\psi|^*(r^{-1})},
        \quad r>0.
    \end{align}
\end{corollary}
\begin{proof}
    We begin with a slight improvement of the second maximal estimate in \eqref{lp-e46}. The idea for this is already present in the proof of Theorem~\ref{lp-43}. We use the fact that $(X_s)_{s\leq t/2}$ and $(X_{u+t/2}-X_{t/2})_{u\geq 0}$ are iid L\'evy processes to get
    \begin{align*}
        \Pp^0\left(X_t^* \leq a\right)
        &= \Pp^0\left(X_{t/2}^* \leq a,\; \sup_{t/2\leq s\leq t} |X_s| \leq a\right)\\
        &\leq \Pp^0\left(X_{t/2}^* \leq a,\; \sup_{u\leq t/2} |X_{u+t/2} - X_{t/2}| \leq 2a\right)\\
        &\leq \Pp^0\left(X_{t/2}^*\leq 2a\right)^2\\
        &\leq \frac{c}{t^2 \left(|\psi|^*((2a)^{-1})\right)^2}.
    \end{align*}
    The argument $(2a)^{-1}$ can be replaced by $a^{-1}$ using Lemma~\ref{pre-74}. By the layer cake formula we have
    \begin{gather*}
        \Ee^0\left[\sigma_r\right]
        = \int_0^\infty \Pp^0\left(\sigma_r > t\right) dt
        \leq \int_0^\infty \Pp^0\left( X_t^* \leq r\right) dt.
    \end{gather*}
    For the upper estimate we use the (just improved) second inequality of \eqref{lp-e46}
    \begin{align*}
        \Ee^0\left[\sigma_r\right]
        &\leq x + \int_x^\infty \Pp^0\left( X_t^* \leq r\right) dt\\
        &\leq x + c\int_x^\infty \frac{dt}{t^2 \left(|\psi|^*(r^{-1})\right)^2}\\
        &= x + \frac 1x \frac{c}{\left(|\psi|^*(r^{-1})\right)^2}.
    \end{align*}
    Minimizing in $x$ gives the upper bound in \eqref{lp-e48}. The lower bound follows in a similar fashion, now with the (original version of the) first maximal inequality in \eqref{lp-e46}:
    \begin{align*}
        \Ee^0\left[\sigma_r\right]
        = \int_0^\infty \Pp^0\left(\sigma_r \geq t\right) dt
        &\geq \int_0^x \Pp^0\left( X_t^* \leq r\right) dt\\
        &= x - \int_0^x \Pp^0\left( X_t^* > r\right) dt\\
        &\geq x - c\int_0^x t \, |\psi|^*(r^{-1})\,dt\\
        &= x - \frac c2 x^2 \, |\psi|^*(r^{-1}).
    \end{align*}
    Again we can optimize in $x$ to get the lower bound.
\end{proof}

    Here is a typical application of the estimates in Corollary~\ref{lp-45}, which is needed to derive Harnack inequalities for L\'evy processes, see Song \& Vondra\v{c}ek~\cite{son-von04}.

\begin{corollary}\label{lp-47}\index{inequality!mean exit time}
    Let $(X_t^z)_{t\geq 0}$ be a L\'evy process with starting point $X_0 = z$, and denote by $\Pp^z$ the corresponding probability measure and by
    $\sigma^z_{\ball{r}{x}}$ the first exit time from the open ball $\ball{r}{x}$ with radius $r>0$ and centre $x$. There is a constant $C\in (0,\infty)$ such that
    \begin{gather}\label{lp-e49}
        \sup_{z\in\ball{r}{x}} \Ee^z \left[\sigma^z_{\ball{r}{x}}\right]
        \leq C \inf_{y\in\ball{r/2}{x}} \Ee^y\left[\sigma^y_{\ball{r}{x}}\right],
        \quad x\in\real.
    \end{gather}
\end{corollary}

\begin{proof}
    Write $X_t \coloneqq X_t^0$ for the L\'evy process with starting point $X_0=0$. For a L\'evy process, $(X_t^z)_{t\geq 0}$ is obtained by a simple shift: $X_t^z =  X_t + z$. Since $\ball{r}{x}\subseteq\ball{2r}{z}$ for all $z\in\ball{r}{x}$ and $\ball{r/2}{y}\subseteq\ball{r}{x}$ for all  $y\in\ball{r/2}{x}$, we see that for any $x\in\real$
    \begin{gather*}
        \forall y\in\ball{r/2}{x},\;\forall z\in\ball{r}{x}\::\quad
        \sigma^y_{\ball{r/2}{y}} \leq \sigma^y_{\ball{r}{x}} \et \sigma^z_{\ball{r}{x}} \leq \sigma^z_{\ball{2r}{z}}.
    \end{gather*}
    Because of $\Ee^y\left[\smash[b]{\sigma^y_{\ball{s}{y}}}\right] = \Ee^0\left[\smash[b]{\sigma^0_{\ball{s}{0}}}\right]$, we can apply Corollary~\ref{lp-45} to get
    \begin{gather*}
        \sup_{z\in\ball{r}{x}} \Ee^z \left[\sigma^z_{\ball{r}{x}}\right]
        \leq \sup_{z\in\ball{r}{x}} \Ee^z \left[\sigma^z_{\ball{2r}{z}}\right]
        \leq \frac{c}{|\psi|^*\left(\tfrac{1}{2r}\right)}
    \intertext{as well as}
        \frac{c'}{|\psi|^*\left(\tfrac{2}{r}\right)}
        \leq \inf_{y\in\ball{r/2}{x}} \Ee^y\left[\sigma^y_{\ball{r/2}{y}}\right]
        \leq \inf_{y\in\ball{r/2}{x}} \Ee^y\left[\sigma^y_{\ball{r}{x}}\right].
    \end{gather*}
    Using Lemma~\ref{pre-74}, we can compare $|\psi|^*\left(\tfrac 2r\right) \asymp |\psi|^*\left(\tfrac 1{2r}\right)$, leading to \eqref{lp-e49}.
\end{proof}

\subsection{Generalized Moments of L\'evy Processes}\label{lp-gen}

Traditionally expressions of the form $\Ee\left[g(X_t)\right]$ are called \emph{generalized moments}. The minimum requirement on $g\colon \rd\to [0,\infty)$ is measurability, but to get satisfactory existence (finiteness) results, we need that $g$ is \emph{locally bounded} and \emph{submultiplicative}, i.e.
\begin{gather}\label{lp-e52}
    \forall r>0\::\:\sup_{|x|\leq r} g(x) < \infty
    \et
    \forall x,y\in\rd\::\: g(x+y)\leq g(x)g(y).
\end{gather}

\begin{scholium}[on submultiplicative functions]\label{lp-51}\index{submultiplicative function}
    Sometimes the submultiplicativity property is stated as $g(x+y)\leq cg(x)g(y)$ for some constant $c\geq 1$; changing $g\rightsquigarrow c\cdot g$ reduces things to the present definition.

    \subparagraph{Regularity.} By a mollifier argument, see Berger \emph{et al.}~\cite{bks22}, we can always modify $g$ in such a way, that it is smooth and its derivatives $|\partial^\alpha g(x)|$ are dominated by $c_\alpha g(x)$ for a suitable constant $c_\alpha\in (0,\infty)$. Fix $g$. Let $\phi=\phi_g\in C^\infty(\rd)$ such that $\phi\geq 0$, $\supp\phi=\cball{1}{0}$ and $\int \phi(z)\,dz = \sup_{|x|\leq 1}g(x)=c_g$. From
    \begin{gather*}
        g*\phi(x+y)
        = \int g(x+y-z)\phi(z)\,dz
        \leq \int g(y-z)\phi(z)\,dz \cdot g(x)
    \intertext{in combination with}
        c_g g(x)
        = \int g(x-z+z)\phi(z)\,dz
        \leq \sup_{|z|\leq 1} g(z) \int g(x-z)\phi(z)\,dz
        = c_g g*\phi(x)
    \end{gather*}
    we get that $g*\phi$ is submultiplicative and, taking $y=0$, that $g*\phi\asymp g$. It is now easy to see that $\partial^\alpha(g*\phi) = g*(\partial^\alpha\phi)$ can be bounded by $\|g\cdot \partial^\alpha\phi\|_{L^1} \cdot g(x)$.

    Taking a convolution does not only regularize the function $g$, but it gives additional structure, which allows one to get such kind of bounds for the derivatives. For example, in dimension one $g(x) = e^x (2+ \sin (x^2))$ is smooth and does \emph{not} satisfy the above bounds, see Martin~\cite{stack-2015}.

    \subparagraph{Boundedness.}
    Since $\gamma(r)\coloneqq \sup_{|x|\leq r}g(x)$ satisfies
    \begin{gather*}
        \gamma(r+s) = \sup_{|z|\leq r+s}g(z) = \sup_{|x|\leq r, |y|\leq s}g(x+y)\leq \gamma(r)\gamma(s),
    \end{gather*}
    it is submultiplicative and increasing. Without loss of generality, $\gamma$ is smooth, so integrating up the estimate $\gamma'(r)\leq c\gamma(r)$ yields $\log\gamma(r)\leq c_0r+c_1$. Thus, $g(x)$ grows at most exponentially. For an alternative proof see Sato~\cite[Lem.~25.5]{sato99}.

    \subparagraph{Examples.}
    If $g$ is submultiplicative, so are -- up to a multiplicative constant -- $x\mapsto g(x)+1$ and $x\mapsto\max\{g(x),1\}$ and $g^\alpha(x)$ for any $\alpha>0$. Typical examples of submultiplicative functions are $x\mapsto \max\{|x|,1\}$, $x\mapsto \exp\left[c|x|^\beta\right]$, $\beta\in (0,1]$, $\log(|x|+e)$, $\log(\log(x+e^e))$ and all combinations with coordinate projections, see Sato~\cite[Prop.~25.4]{sato99}.
\end{scholium}
Since we are only interested in estimating $\Ee\left[g(X_t)\right]$, we may always switch between $g$, $\phi*g$ and $\phi*g + 1$, i.e.\ assume that $g$ is as good as we need it to be.

A classical result in the theory of L\'evy processes, due to Kruglov~\cite{kruglov70} (see also Sato~\cite[Thm.~25.3]{sato99}), shows that the generalized moment of a L\'evy process with triplet $(b,Q,\nu)$ satisfies
\begin{gather*}
    \Ee\left[g(X_t)\right] < \infty
    \iff
    \int_{|y|\geq 1} g(y)\,\nu(dy) < \infty.
\end{gather*}
The right hand side does not depend on $t$, and so we may read the left hand side to be true for \emph{some} or for \emph{all} $t>0$. In the following theorem we will see that this equivalence remains true for $\sup_{s\leq t}g(X_s)$, i.e.\ we have a maximal inequality. The following statement is taken from Berger \emph{et al.}~\cite[Thm.~3]{bks22}, but most of the equivalences were known earlier, see Remark~\ref{lp-55} for a discussion.

\begin{theorem}\label{lp-53}\index{generalized moment}\index{submultiplicative function}
    Let $(X_t)_{t\geq 0}$ be a $d$-dimensional L\'evy process with characteristic exponent $\psi$ and triplet $(b,Q,\nu)$. Moreover, $g\colon \rd\to[0,\infty)$ is a locally bounded submultiplicative function, and $\Tcal$ is the family of all stopping times for the process $(X_t)_{t\geq 0}$. Then, the maximal moment is finite
    \begin{enumerate}[\upshape(M1)]\itemsep6pt
    \item\label{lp-53-a}
        $\Ee\left[\sup_{s\leq t}g(X_s)\right]<\infty$ for some \textup{(}hence, all\textup{)} $t>0$,
    \end{enumerate}
    if, and only if, one of the following \textup{(}a fortiori equivalent\textup{)} conditions \textup{(M\ref{lp-53-b})--(S\ref{lp-53-i})} is satisfied.

    \noindent
    Siebert's moment conditions:
    \begin{enumerate}[\upshape(M1)]\addtocounter{enumi}{1}\itemsep6pt
    \item\label{lp-53-b}
        $\Ee\left[g(X_t)\right]<\infty$ for some \textup{(}hence, all\textup{)} $t>0$;
    \item\label{lp-53-c}
        $\left\{g(X_{\sigma\wedge t})\right\}_{\sigma\in\Tcal}$ is uniformly integrable for some \textup{(}hence, all\textup{)} $t>0$, i.e.
        \begin{equation*}
	        \lim_{R\to\infty} \sup_{\sigma\in\Tcal}\int_{g(X_\sigma)>R}g(X_\sigma)\, d\Pp=0;
	    \end{equation*}
    \item\label{lp-53-d}
        $\sup_{\sigma\in\Tcal}\Ee \left[g(X_{\sigma\wedge t})\right]<\infty$ for some \textup{(}hence, all\textup{)} $t>0$;
    \item\label{lp-53-e}
        \textup{[}this requires, in addition, that $g(x)=g(|x|)$ and $r\mapsto g(r)$ is increasing\textup{]}\\[5pt]
        $\Ee\left[g\left(\sup_{s\leq t}|X_t|\right)\right]<\infty$ for some \textup{(}hence, all\textup{)} $t>0$.
    \end{enumerate}
    Kruglov's condition in terms of the triplet:
    \begin{enumerate}[\upshape(T1)]\addtocounter{enumi}{5}\itemsep6pt
    \item\label{lp-53-f}
        $\int_{|y|\geq 1} g(y)\,\nu(dy) < \infty$.
    \end{enumerate}
    Hulanicki's conditions in terms of the semigroup and the generator:
    \begin{enumerate}[\upshape(S1)]\addtocounter{enumi}{6}\itemsep6pt
    \item\label{lp-53-g}
        The adjoint semigroup $P^*_tf(x)\coloneqq \Ee \left[f(x-X_t)\right]$ is a strongly continuous semigroup on the weighted Lebesgue space $L^1(\rd;g(x)\,dx)$;
    \item\label{lp-53-h}
    	The adjoint infinitesimal generator $(A^* \phi)(x) \coloneqq  (A \phi)(-x)$ satisfies $A^* \phi \in L^1(\real^d,g(x) \, dx)$ for all $\phi \in C_c^{\infty}(\real^d)$;
	\item\label{lp-53-i}
		There is a non-negative $0\not\equiv\phi \in C_c^{\infty}(\rd)$ such that $A^* \phi \in L^1(\rd;g(x)\,dx)$.
    \end{enumerate}
    If one \textup{(}hence, all\textup{)} of these conditions is satisfied, then there is some constant $c>0$ such that
    \begin{gather*}
    	\Ee^x\left[\sup_{s\leq t} g(X_t)\right] \leq g(x) e^{c t}, \quad t > 0.
    \end{gather*}
\end{theorem}
\begin{proof}
Without loss of generality we can assume that $g$ is of the form $g*\phi$ and $g\geq 1$, see Scholium~\ref{lp-51}.

\primolist
    We begin by showing that $S(t)\coloneqq \Ee\left[\sup_{s\leq t}g(X_s)\right]<\infty$ for \emph{some} $t$ implies $S(t)<\infty$ for \emph{all} $t$. Since $t\mapsto S(t)$ is increasing, the claim follows if we can show that $S(2t)<\infty$ if $S(t)<\infty$. We have used a very similar argument in the proof of Corollary~\ref{lp-45}. Since $(X_s)_{s\leq t}$ and $(X_{u+t}-X_t)_{u\geq 0}$ are iid L\'evy processes, we have
    \begin{align*}
        S(2t)
        &\leq S(t) + \Ee\left[\sup_{u\leq t}g(X_{u+t}-X_t + X_t)\right]\\
        &\leq S(t) + \Ee\left[\sup_{u\leq t}g(X_{u+t}-X_t) g(X_t)\right]\\
        &\leq S(t) + \Ee\bigl[g(X_t)\bigr] \cdot\Ee\left[\sup_{u\leq t}g(X_{u+t}-X_t)\right]\\
        &\leq S(t) + S(t)^2.
    \end{align*}

\secundolist
    If $\Ee\left[g(X_t)\right]$ is finite for \emph{some} $t>0$, then it is finite for \emph{all} $s<t$, too. Indeed, using submultiplicativity in the form $g(y)\geq g(x)/g(x-y)$ and the independent increments property of a L\'evy process, we get
    \begin{gather*}
        \infty > \Ee\left[g(X_t)\right]
        \geq \Ee\left[ \frac{g(X_s)}{g(X_s-X_t)}\right]
        = \Ee\left[g(X_s)\right] \cdot \underbracket[.6pt]{\Ee\left[\frac 1{g(X_s-X_t)}\right]}_{\in (0, 1]}.
    \end{gather*}
    Since $1\leq g(x) <\infty$, we have $0 < 1/g(x) \leq 1$, and the assertion follows.
\tertiolist
    Assume that (M\ref{lp-53-b}) holds for \emph{some} $t$, hence for all sufficiently small $t$. We show that (M\ref{lp-53-a}) holds for \emph{some} $T>0$. For $a,b> 0$ we define a stopping time
    \begin{gather*}
        \sigma \coloneqq  \sigma_{a,b} \coloneqq  \inf\left\{s\mid g(X_s) > g(0) e^{a+b}\right\}
    \end{gather*}
    and observe that the submultiplicativity of $g$ gives
    \begin{align*}
        \Pp\left( g(X_t) > e^a\right)
        &\geq \Pp\left( g(X_t) > e^a,\; g(X_\sigma)\geq  g(0)e^{a+b},\; \sigma \leq t\right)\\
        &\geq \Pp\left( g(X_\sigma - X_t) < g(0)e^b,\; g(X_\sigma)\geq  g(0)e^{a+b},\; \sigma \leq t\right).
    \end{align*}
    The strong Markov property yields (for all $t\leq T$, $T$ will be determined in the following step)
    \begin{align*}
        \Pp\left( g(X_t) > e^a\right)
        &\geq \int_{\sigma \leq t} \Pp\left( g(-X_{t-\sigma(\omega)}) < g(0)e^b\right) \Pp(d\omega)\\
        &\geq \inf_{r\leq t} \Pp\left( g(-X_{r}) < g(0)e^b\right)\cdot \Pp(\sigma \leq t)\\
        &\geq \frac 12 \Pp\left(\sup\nolimits_{r \leq t} g(X_r) > cg(0) e^{a+b}\right).
    \end{align*}
    In the last estimate we use $\left\{\sup\nolimits_{s\leq t} g(X_s)>  g(0) e^{a+b}\right\}\subseteq\{\sigma \leq t\} $. The factor $\frac 12$ comes from the fact that $g$ is locally bounded and $\lim_{t\to 0}\Pp(|X_t|>\epsilon)=0$ (continuity in probability), which shows that there is some $T>0$ such that
    \begin{gather*}
        \Pp\left(g(-X_t)< g(0)e^b\right) \geq \frac 12
        \quad\text{for all $t\leq T$}.
    \end{gather*}
    Setting $e^\gamma \coloneqq  g(0)e^{b}$ we have shown
    \begin{gather*}
        \Pp\left(\sup\nolimits_{s\leq T} g(X_s) > e^{a+\gamma}\right) \leq 2\Pp(g(X_T)>e^a).
    \end{gather*}
    Now a simple variant of the layer-cake argument yields
    \begin{align*}
        \Ee\left[\sup_{s\leq T}g(X_s)\right]
        \leq e^\gamma + 2e^\gamma \Ee\left[g(X_T)\right].
    \end{align*}
    Using Step~\secundo, we see that (M\ref{lp-53-a}) holds for this particular $T>0$, then by Step~\primo\ for all $T>0$. The latter trivially implies (M\ref{lp-53-b}) for all $t>0$. Thus, we have established that (M\ref{lp-53-a})$\Leftrightarrow$(M\ref{lp-53-b}).

\quartolist
    The uniform integrability of $\left\{g(X_{\sigma\wedge t})\right\}_{\sigma\in\Tcal}$ (for some or all $t>0$) follows immediately from the uniform bound $g(X_{\sigma\wedge t}) \leq \sup_{s\leq t}g(X_{s}) \in L^1(\Pp)$ (for some or all $t>0$).

    Since uniformly integrable families are $L^1$-bounded, it is clear that (M\ref{lp-53-c}) implies (M\ref{lp-53-d}) -- both for some $t>0$ or for all $t>0$. Finally, (M\ref{lp-53-d}) trivially gives (M\ref{lp-53-b}).

    Summing up, we have shown (M\ref{lp-53-a})$\Rightarrow$(M\ref{lp-53-c})$\Rightarrow$(M\ref{lp-53-d})$\Rightarrow$(M\ref{lp-53-b})$\Rightarrow$(M\ref{lp-53-a}).

\quintolist
    Assume, in addition, that $g(x)=g(|x|)$ and that $g(r)$ is increasing. By monotonicity, $\sup_{s\leq t}g\left(|X_s|\right)\leq g\left(\sup_{s\leq t}|X_s|\right)$. On the other hand, by right continuity, there is a random time $\tau\in [0,t]$ such that $\left|\sup_{s\leq t} \left|X_s(\omega)\right| - \left|X_{\tau(\omega)}(\omega)\right| \right|\leq 1$.  Consequently,
    \begin{gather*}
        g\left(\sup_{s\leq t}|X_s|\right)
        \leq  g\left(\sup_{s\leq t}|X_s|-|X_\tau|\right) g\left(|X_\tau|\right)
        \leq  \sup_{|y|\leq 1}g\left(|y|\right) \sup_{s\leq t}g\left(|X_s|\right).
    \end{gather*}
    The equivalence (M\ref{lp-53-a})$\Leftrightarrow$(M\ref{lp-53-e}) now follows from the estimates $\sup_{s\leq t}g\left(|X_s|\right)\asymp g\left(\sup_{s\leq t}|X_s|\right)$.

\sextolist
    Now we show (T\ref{lp-53-f})$\Rightarrow$(M\ref{lp-53-b}). For this, we use that $g$ is smooth (and has a convolution structure) so that we can insert $g$ into the formula \eqref{lp-e14} of the generator $A$ of $(X_t)_{t\geq 0}$. The differential operator part can be estimated using $\left|\partial^\alpha g(x)\right|\leq c_\alpha g(x)$, see Scholium~\ref{lp-51}. For the second integral expression we first apply Taylor's formula, $|g(x+y)-g(x) - y\cdot\nabla g(x)| \leq c|y|^2 |\nabla^2 g(x+\theta y)|$ with $\theta\in (0,1)$ and for all $|y| < 1$, and use again $|\partial^\alpha g(x+\theta y)|\leq c_\alpha g(x+\theta y) \leq c_\alpha g(x)\sup_{|z|\leq 1}g(z)$. For the first integral term of \eqref{lp-e14} we have $g(x+y)-g(x)\leq g(x)(g(y)-1)$. Altogether,
    \begin{gather*}
        Ag(x) \leq c\bigl(b,Q,\nu\big|_{\ball{1}{0}}\bigr)\cdot g(x) + \int_{|y|\geq 1} \left(g(y)-1)\right)\nu(dy) \cdot g(x).
    \end{gather*}
    The last integral is finite due to (T\ref{lp-53-f}). By It\^o's formula, $M_t = g(X_t)-g(x) - \int_0^t Ag(X_s)\,ds$ is a local martingale. If $(\sigma_n)_{n\in\nat}$ is a sequence of stopping times such that $M_{t\wedge\sigma_n}$ is a martingale, we get
    \begin{align*}
        \Ee^x \left[g(X_{t\wedge\sigma_n})\right] - g(x)
        = \int_0^t \Ee^x\left[Ag(X_{s\wedge\sigma_n})\right] ds
        \leq c\int_0^t \Ee^x\left[g(X_{s\wedge\sigma_n})\right] ds.
    \end{align*}
    Now we can apply Gronwall's lemma to conclude
    \begin{gather*}
        \Ee^{x} \left[g(X_{t\wedge\sigma_n})\right] \leq g(x) e^{ct}.
    \end{gather*}
    Letting $n\to\infty$, Fatou's lemma proves (M\ref{lp-53-b}) as well as the exponential bound for the generalized moment.

\septimolist
    In order to see (M\ref{lp-53-a})$\Rightarrow$(T\ref{lp-53-f}) we write $X_t = Y_t + J_{t}$ where $J_t$ is the compound Poisson process comprising all jumps whose absolute height exceeds $1$; we have used this (in dimension $1$) in the proof of Theorem~\ref{lp-43}, see also Scholium~\ref{pre-70}. In particular, $J_t = H_1+\dots+H_{N_t}$, where the jump heights $H_k$ are iid random variables with law $\nu(dy \,\cap\, \coball{1}{0})/\nu(\coball{1}{0})$ and $(N_t)_{t\geq 0}$ is an independent Poisson process with intensity $\nu(\coball{1}{0})$.

    Now we can argue as in Step~\secundo, using the subadditivity of $g$ and the independence of $Y_t$ and $J_t$:
    \begin{gather*}
        \infty > \Ee\left[g(X_t)\right]
        =\Ee\left[g(Y_t+J_t)\right]
        \geq \Ee\left[ \frac{g(J_t)}{g(-Y_t)}\right]
        = \Ee\left[g(J_t)\right] \cdot \smash[b]{\underbracket[.6pt]{\Ee\left[\frac{1}{g(-Y_t)}\right]}_{\mathclap{\in(0,1]\text{\ as\ } g(y)\geq 1}}}.
    \end{gather*}
    Moreover, because of the structure of $J_t$, we see
    \begin{align*}
        \Ee\left[g(J_t)\right]
        &= \Ee\left[\sum_{n=0}^\infty g(H_1+\dots+H_n)\Pp(N_t=n)\right]\\
        &\geq \Ee\left[g(H_1)\right] \Pp\left(N_t=1\right)\\
        &= e^{-t\nu(\coball{1}{0})} \smash[b]{\int_{|y|\geq 1} g(y)\frac{\nu(dy)}{\nu(\coball{1}{0})}},
    \end{align*}
    and (T\ref{lp-53-f}) follows.

\medskip
This completes the proof of the equivalence of  (M\ref{lp-53-a})--(T\ref{lp-53-f}). For Hulanicki's conditions, we need to show (M\ref{lp-53-c})$\Rightarrow$(S\ref{lp-53-g})$\Rightarrow$(M\ref{lp-53-b}) and (T\ref{lp-53-f})$\Rightarrow$(S\ref{lp-53-h})$\Rightarrow$(S\ref{lp-53-i})$\Rightarrow$(M\ref{lp-53-b}). Since this requires semigroup methods, we omit the proofs, and refer to Berger \emph{et al.}~\cite{bks22}.
\end{proof}

\begin{remark}\label{lp-55}
    Let us be a bit more precise in naming the above conditions. Siebert~\cite{siebert82} showed the equivalence (M\ref{lp-53-a})$\Leftrightarrow$(M\ref{lp-53-b}). Since it is not a long way from there to the conditions (M\ref{lp-53-c}), (M\ref{lp-53-d}) in this group (as observed by Berger \emph{et al.}~\cite{bks22}), we feel that the name-tag is justified. The last condition (M\ref{lp-53-e}) in this group is inspired by Kwapie\'n \& Woyczy\'nski~\cite[Prop.~8.4.1]{kwa-woy92} and Sato~\cite[Thm.~25.18]{sato99}, but our almost trivial new proof reveals that it properly belongs here. We have already mentioned Kruglov~\cite{kruglov70} for (M\ref{lp-53-b})$\Leftrightarrow$(T\ref{lp-53-f}), while the equivalence of (M\ref{lp-53-b}), (S\ref{lp-53-h}) and (S\ref{lp-53-i}) is due to Hulanicki \cite{hulanicki}; morally, (S\ref{lp-53-g}) also belongs here, it is from Berger \emph{et al.}~\cite{bks22}.
\end{remark}

The estimates of Theorem~\ref{lp-53} (partially) extend to larger classes of processes, notably to additive processes (see the following remark) and L\'evy-type processes, see K\"{u}hn~\cite{ltp-moments}.

\begin{remark}\label{lp-57}
    A \enquote{L\'evy process without stationary increments}, i.e.\ a stochastic process with independent increments \eqref{Lindep} or \eqref{Lindep-bis} and stochastic continuity \eqref{Lcont} is called an \emph{additive process}\footnote{In fact, L\'evy himself considered mostly this class of processes and It\^o calls them -- correctly -- L\'evy processes, contrary to today's usage.}. Essentially, the L\'evy--Khintchine formula \eqref{lp-e10} and Theorem~\ref{lp-11} remain valid with $t$-dependent L\'evy triplets. This means, in particular, that the equivalence of (M\ref{lp-53-a})--(T\ref{lp-53-f}) in Theorem~\ref{lp-53} remain valid -- even with almost the same proof, see Berger \emph{et al.}~\cite[Thm.~11]{bks22}; the \enquote{if we have an estimate for some $t$, we have it for all $t$} assertion, however, breaks down for additive processes.
\end{remark}

\section{L\'evy-type Processes}\label{fel}

In this section, we study maximal inequalities for so-called \emph{L\'evy-type processes}. Intuitively,
\begin{quote}\normalsize
a L\'evy-type process behaves locally like a L\'evy process, which means that the L\'evy triplet may change, depending on the current position of the process. Once started at a point $x_0 \in \real^d$, the process behaves for an \enquote{infinitesimally small} time $\Delta t$ like a L\'evy process with L\'evy triplet $(b(x_0),Q(x_0),\nu(x_0,dy))$ depending on the starting point $x_0$, then for an \enquote{infinitesimally small} time like a L\'evy process with L\'evy triplet $(b(X_{\Delta t}),Q(X_{\Delta t}),\nu(X_{\Delta t},dy))$, and so on.
\end{quote}
The corresponding family $(b(x),Q(x),\nu(x,dy))$, $x \in \real^d$, of L\'evy triplets characterizes many distributional properties and sample path properties of the L\'evy-type process. Throughout, we assume that $x\mapsto (b(x),Q(x),\nu(x,B))$ is measurable for every Borel set $B\subset\rd\setminus\{0\}$.

\subsection{Getting Started: Basics on L\'evy-type Processes}

Let $(X_t)_{t \geq 0}$ be a strong Markov process with respect to a family of probability measures $(\Pp^x)_{x \in \rd}$, and assume that $(X_t)_{t \geq 0}$ has c\`adl\`ag sample paths. Recall that the underlying filtered probability space $(\Omega,\Ascr,\Pp,\Fscr_t)$ satisfies the usual conditions, i.e.\ $(\Omega,\Ascr,\Pp)$ is complete, $\Fscr_0$ contains all $\Pp$-null sets and $(\Fscr_t)_{t\geq 0}$ is right-continuous. We call $(X_t)_{t \geq 0}$ a \emph{L\'evy-type process} \index{Levy-type process@L\'evy-type process}\index{Levy triplet@L\'evy triplet}\index{infinitesimal characteristics} if for every $x \in \rd$ and $f \in C_c^{\infty}(\rd)$, the process
\begin{equation*}
	M_t^f \coloneqq  f(X_t)-f(X_0) - \int_0^t Af(X_s) \, ds
\end{equation*}
is a martingale with respect to $\Pp^x$ for some \emph{L\'evy-type operator} $A$, i.e.\ an operator of the form
\begin{align}
	\label{fel-eq2}
	\begin{aligned}
	Af(x) &= b(x) \cdot \nabla f(x) + \frac{1}{2} \tr\left( Q(x) \cdot \nabla^2 f(x) \right) \\
	&\quad + \int_{\rd \setminus \{0\}} \left(f(x+y) -f(x)- \nabla f(x) \cdot y \I_{(0,1)}(|y|) \right) \, \nu(x,dy);
	\end{aligned}
\end{align}
here, $(b(x),Q(x),\nu(x,dy))$ is a family of L\'evy triplets, the so-called \emph{\textup{(}infinitesimal\textup{)} characteristics}. As before, we assume that $x\mapsto (b(x), Q(x), \nu(x,B))$ is measurable for every Borel set $B\subseteq \rd\setminus\{0\}$. In other words, the L\'evy-type process $(X_t)_{t \geq 0}$ is a solution to the $(A,C_c^{\infty}(\rd))$-martingale problem for the L\'evy-type operator $A$; see Ethier \& Kurtz~\cite{ethier-kurtz} for further information on the martingale problem and its connection to Markov processes. Note that, by Taylor's formula, the pointwise expression \eqref{fel-eq2} makes sense for any $f \in C_b^2(\rd)$. Using standard rules for the Fourier transform, it is not difficult to see that $A|_{C_c^{\infty}(\rd)}$ can be written as a pseudo-differential operator
\begin{equation*}
	Af(x) = - \int_{\rd} q(x,\xi) e^{ix \cdot \xi} \widehat{f}(\xi) \, d\xi,
	\quad f \in C_c^{\infty}(\rd),\; x \in \rd,
\end{equation*}
where $\widehat{f}(\xi) = (2\pi)^{-d} \int_{\rd} f(x) e^{-ix \cdot \xi} \, dx$ is the Fourier transform of $f$ and
\begin{equation}\label{fel-eq4}
    q(x,\xi)
    = - ib(x) \cdot \xi + \frac{1}{2} \xi \cdot Q(x) \xi + \int_{y\neq 0} \left(1-e^{iy \cdot \xi} + iy \cdot \xi \I_{(0,1)}(|y|) \right) \, \nu(x,dy)
\end{equation}
is the \emph{symbol}. \index{symbol of L\'evy-type process}\index{Levy--Khintchine formula@L\'evy--Khintchine formula} For fixed $x \in \rd$, the mapping $\xi \mapsto q(x,\xi)$ is the characteristic exponent of a L\'evy process. \index{characteristic exponent} In particular, $\xi \mapsto q(x,\xi)$ is continuous and negative definite (in the sense of Schoenberg), cf.\ Sections~\ref{pre-mom} and \ref{lp-back}. This implies that $\xi \mapsto \sqrt{|q(x,\xi)|}$ is subadditive, i.e.\
\begin{gather}\label{fel-eq5}
	\sqrt{|q(x,\xi+\eta)|}
	\leq \sqrt{|q(x,\xi)|} + \sqrt{|q(x,\eta)|}, \quad x,\xi,\eta \in \rd,
\end{gather}
see e.g.\ Schilling~\cite[Thm.~6.2]{crm} or the proof of Lemma~\ref{pre-74}. In the following, we will always assume that the symbol $(x,\xi)\mapsto q(x,\xi)$ is locally bounded. In terms of the characteristics this means that
\begin{gather}\label{fel-eq6}
    \sup_{x \in K} \left(|b(x)|+|Q(x)| + \smash[b]{\int_{y\neq 0}} \min\left\{1,|y|^2\right\}  \nu(x,dy)\right)<\infty
\end{gather}
for any compact set $K \subseteq \rd$. In particular, the local boundedness of $q$ entails that $\|Af\|_{\infty}<\infty$ for any $f \in C_c^{\infty}(\rd)$. Let us give some important examples of L\'evy-type processes.

\begin{example} \label{fel-3}
\begin{enumerate}\itemsep6pt
    \item\label{fel-3-1} \index{Levy process@L\'evy process}
        \emph{L\'evy processes:} If $(X_t)_{t \geq 0}$ is a L\'evy process with characteristic exponent $\psi$ and L\'evy triplet $(b,Q,\nu)$, then $(X_t)_{t \geq 0}$ is a L\'evy-type process with symbol $q(x,\xi)=\psi(\xi)$ and characteristics $(b,Q,\nu)$. In fact, there is a one-to-one correspondence between \emph{L\'evy triplets}, \emph{characteristic exponents} (\emph{symbols}) and \emph{L\'evy processes}.

	\item\label{fel-3-2} \index{Feller process}
    \emph{Feller processes:} Let $(X_t)_{t \geq 0}$ be a Feller process. If the domain of its (strong) infinitesimal generator contains the test functions $C_c^{\infty}(\rd)$, then a theorem by Cour\-r\`ege and van Waldenfels shows that $(X_t)_{t \geq 0}$ is a L\'evy-type process, cf.\ B\"{o}ttcher \emph{et al.}~\cite{matters3}. For a thorough study of Feller processes and their connection to L\'evy-type operators we refer to the monographs B\"{o}ttcher \emph{et al.}~\cite{matters3}, Hoh~\cite{hoh98} and Jacob~\cite{jacob1-3}.

    Although a Feller process always gives an $x$-dependent symbol $q(x,\xi)$ and an $x$-dependent L\'evy triplet $(b(x),Q(x),\nu(x,dy))$, the converse does not hold: there are triplets and symbols which are not connected with a \emph{Feller} process, cf.\ B\"{o}ttcher \emph{et al.}~\cite[Ex.~2.26]{matters3} and van Casteren~\cite[p.~156]{vancas11}. All of these examples still lead to (deterministic) processes; we are not aware of examples of symbols of the form \eqref{fel-eq4} which are not connected with a (stochastic or deterministic) process at all.

	\item\label{fel-3-3} \index{Levy-driven SDE@L\'evy-driven SDE}
        \emph{Solutions of L\'evy-driven SDEs:} Let $(L_t)_{t \geq 0}$ be a L\'evy process with characteristic exponent $\psi$ and $\sigma$ a continuous bounded function. If the stochastic differential equation (SDE)
	\begin{equation*}
		dX_t = \sigma(X_{t-}) \, dL_t, \quad X_0 = x,
	\end{equation*}
    has a unique weak solution, then $(X_t)_{t \geq 0}$ is a L\'evy-type process with symbol $q(x,\xi) = \psi(\sigma(x)^{\top} \xi)$, where $\psi$ is the characteristic exponent of $(L_t)_{t \geq 0}$, cf.\ Schilling \& Schnurr~\cite{sch-schnurr10}. The assumption on the boundedness of $\sigma$ and the uniqueness of the solution can be relaxed, cf.\ K\"{u}hn~\cite{sde,markovian-selection}.

    \item\label{fel-3-4} \index{pseudo Poisson process}
    \emph{\textup{(}Compound\textup{)} pseudo Poisson processes:} William Feller~\cite[Ch.~X.1]{feller-2} calls L\'evy-type processes with a triplet of the form $(0,0,\nu(x,dy))$ such that $\lambda_0 \coloneqq  \sup_{x\in\rd} \nu\left(x,\rd\setminus\{0\}\right) < \infty$ (compound) pseudo Poisson processes. The condition on the L\'evy kernel means that we have only finite jump activity such that the sample paths of the process are piecewise constant with exponential holding times between any two jumps. It is not difficult to construct such processes in a pathwise manner, see Feller \emph{loc.\ cit.} or Ethier \& Kurtz~\cite[Ch.~4.\S2]{ethier-kurtz}. Another way to understand (and, indeed to construct) such processes is to start with an arbitrary time-homogeneous Markov chain $(Y_n)_{n\in\nat}$ with transition kernel $p(x,dy) \coloneqq  \lambda_0^{-1}\nu(x,dy)$, and to time-change it using an independent Poisson process $(N_t)_{t\geq 0}$ with intensity $\lambda_0$, i.e.\ $X_t \coloneqq  Y_{N_t}$.

    \item\label{fel-3-5} \index{stable-like process}
    \emph{Stable-like processes:} A Feller process whose generator is of the form $(-\Delta)^{\alpha(x)}$ for some continuous function $\alpha \colon \rd\to[0,2]$ is a stable-like process. Intuitively this is a stable process which changes its stability behaviour \enquote{on the go}. This class of processes was proposed by Bass~\cite{bass88} who also proved the existence in dimension $d=1$. Rigorous results in higher dimensions are due to Hoh~\cite{hoh00}, Jacob \& Leopold~\cite{jac-leo93} and using a parametrix construction e.g.\ in Knopova \emph{et al.}~\cite{kno-kul-schi21}, Kolokoltsov~\cite{kol11} or K\"{u}hn~\cite{matters6}, see also the references therein.

    \item\label{fel-3-6} \index{Dirichlet form}
    \emph{Processes connected with Dirichlet forms:} In Schilling \& Uemura~\cite{schilling-uemura} criteria are given in terms of Dirichlet forms such that a negative definite symbol $q(x,\xi)$ gives rise to a Dirichlet form and, thus, a stochastic process. This process need not be a Feller process. In Schilling \& Wang~\cite{schilling-wang} this discussion is extended to not necessarily symmetric semi-Dirichlet forms.
\end{enumerate}
\end{example}
The symbol plays a key role for the probability estimates which we are going to derive. It is therefore natural to ask how the symbol can be calculated given the stochastic process. If $(X_t)_{t \geq 0}$ is a L\'evy process, then the L\'evy--Khintchine formula yields \index{symbol of L\'evy-type process}
\begin{equation*}
	- \psi(\xi)
	= \lim_{t \to 0} \frac{e^{-t \psi(\xi)}-1}{t}
	= \lim_{t \to 0} \frac{\Ee^x \left[e^{i \xi (X_t-x)}\right] - 1}{t}.
\end{equation*}
Although the L\'evy--Khintchine formula breaks down for L\'evy-type processes, i.e.\ $\Ee^x e^{i \xi (X_t-x)}\neq e^{-t q(x,\xi)}$ in general, the identity
\begin{gather}\label{fel-eq7}
	- q(x,\xi)
	= \lim_{t \to 0} \frac{\Ee^x e^{i \xi (X_t-x)}-1}{t}
\end{gather}
remains valid for a wide class of L\'evy-type processes.

\begin{proposition} \label{fel-4}
    Let $(X_t)_{t \geq 0}$ be a L\'evy-type process with symbol $q(x,\xi)$ and characteristics $(b(x),Q(x),\nu(x,dy))$. Assume that $q$ has bounded coefficients, i.e.\
	\begin{equation*}
		|q(x,\xi)| \leq C (1+|\xi|^2), \quad x,\xi \in \rd,
	\end{equation*}
	for some constant $C>0$.\footnote{Equivalently, $\sup_{x \in \rd} (|b(x)|+|Q(x)|+\int \min \{1,|y|^2\} \, \nu(x,dy))<\infty$.}
    If $x \mapsto q(x,\xi)$ is continuous, then \eqref{fel-eq7} holds.
\end{proposition}

The key step in the proof is the identity
\begin{gather}\label{fel-eq8}
	\Ee^x \left(e_{\xi}(X_t)\right) - e_{\xi}(x)
    = \Ee^x \left( \int_0^t (A e_{\xi})(X_s) \, ds \right)
\end{gather}
for the function $e_{\xi}(x) \coloneqq  e^{i\xi \cdot x}$. The definition of a L\'evy-type process entails that
\begin{equation*}
	\Ee^x(e_{\xi} \chi_k)(X_t) - (e_{\xi} \chi_k)(x) = \Ee^x \left( \int_0^t (A(e_{\xi}\chi_k))(X_s) \, ds \right)
\end{equation*}
for $\chi_k(x) \coloneqq  \chi(x/k)$ and $\chi \in C_c^{\infty}(\rd)$ with $\I_{\ball{1}{0}} \leq \chi \leq \I_{\ball{2}{0}}$. In order to get \eqref{fel-eq8}, a limiting argument is then needed to pass to the limit $k \to \infty$, see B\"{o}ttcher \emph{et al.}~\cite[Proof Thm.~2.36]{matters3} for details. Once we have \eqref{fel-eq8} at our disposal, we note that $(Ae_{\xi})(x) = - e_{\xi}(x) q(x,\xi)$. Using the boundedness of the coefficients and the continuity of $x \mapsto q(x,\xi)$, the dominated convergence theorem yields the assertion.

\begin{remark}\label{fel-4b}
    There is a variant of Proposition~\ref{fel-4} for L\'evy-type processes with unbounded coefficients. If the symbol $q$ is locally bounded and continuous with respect to the $x$-variable, then
	\begin{gather}
		-q(x,\xi) = \lim_{t \to 0} \frac{\Ee^x \left[e^{i \xi (X_{t \wedge \tau_r^x}-x)}\right]-1}{t},
		\label{fel-eq9}
	\end{gather}
    where $\tau_r^x  = \inf\left\{t \geq 0 \mid |X_t - x| > r\right\}$ is the first exit time from the closed ball $\cball{r}{x}$. The proof is similar to the one above, combined with a standard stopping argument, see also B\"{o}ttcher \emph{et al.}~\cite[Cor.~2.39]{matters3}. This technique was developed by Schnurr \cite[Ch.~4.3]{schnurr-diss} and Schilling \& Schnurr \cite{sch-schnurr10}.
\end{remark}

\subsection{Maximal Tail Estimates for L\'evy-type Processes}

Let $(X_t)_{t \geq 0}$ be a L\'evy-type process. We will now study probability estimates for the running supremum $\sup_{s \leq t} |X_s-X_0|$, that is, we are interested in lower and upper bounds for the probability $\Pp^x \left( \sup_{s \leq t} |X_s-x| \geq r \right)$. If we denote by
\begin{gather}\label{fel-eq10}
    \tau_r^x = \inf\left\{t \geq 0 \mid X_t \notin \cball{r}{x}\right\}
\end{gather}
the first exit time from the closed ball $\cball{r}{x}$, then
\begin{gather}\label{fel-eq11}
	\{\tau_r^x<t\}
	\subseteq \left\{ \sup_{s \leq t} |X_s-x|>r \right\}
	\subseteq \{\tau_r^x \leq t\}
	\subseteq \left\{ \sup_{s \leq t} |X_s -x| \geq r \right\},
\end{gather}
and so the results can be read equivalently as estimates for $\Pp^x(\tau_r^x \leq t)$. Let us point out that \eqref{fel-eq11} remains valid if we replace $t$ by a random time.

\begin{theorem} \label{fel-5} \index{inequality!maximal tail}
	Let $(X_t)_{t \geq 0}$ be a L\'evy-type process with symbol $q(x,\xi)$. There is  a constant $C>0$ such that
	\begin{gather}\label{fel-eq13}
		\Pp^x \left( \sup_{s \leq \sigma} |X_s-x| \geq r \right)
		\leq C \,\Ee^x\left[\sigma\right] \sup_{|y-x| \leq r} \sup_{|\xi| \leq r^{-1}} |q(y,\xi)|
	\end{gather}
	for all $x \in \rd$, $r>0$ and any a.s.\ finite stopping time $\sigma$, i.e.\ $\Pp^x(\sigma<\infty)=1$. In particular,
	\begin{gather}\label{fel-eq15}
		\Pp^x \left( \sup_{s \leq t} |X_s-x| \geq r \right)
		\leq C\,t \sup_{|y-x| \leq r} \sup_{|\xi| \leq r^{-1}} |q(y,\xi)|.
	\end{gather}
\end{theorem}

We give a direct proof using martingale methods which goes back to Schilling \cite{rs-growth}. This approach will allow us to establish a localized version of Theorem~\ref{fel-5}. A much shorter proof based on L\'evy's truncation inequality (see e.g.\ the upper bound in Lemma~\ref{pre-48}, or Schilling~\cite[Thm.~30.1]{schilling-dc} for a textbook version) can be found in K\"{u}hn~\cite[Thm.~1.29]{matters6}. Using similar methods as \cite{rs-growth}, Schnurr~\cite[Prop.~3.10]{schnurr13} has established an analogous estimate for homogeneous diffusions with jumps (in the sense of Jacod \& Shiryaev \cite[Def.~III.2.18]{jac-shir}). The proof of Schnurr is based on the the triplet $b$, $Q$ and $\nu$ of \eqref{fel-eq4} and it completely avoids the \emph{generator}; this means that the Markovianity of the process is not needed.

\begin{proof}[Proof of Theorem~\ref{fel-5}]
	The second assertion \eqref{fel-eq15} follows  from \eqref{fel-eq13}.

    \primolist
    We first establish \eqref{fel-eq13} with \enquote{$>r$} instead of \enquote{$\geq r$} on the left hand side. In view of \eqref{fel-eq11}, it suffices to show that
	\begin{equation*}
		\Pp^x\left(\tau_r^x \leq \sigma\right)
        \leq C\, \Ee^x\left[\sigma\right] \sup_{|y-x| \leq r} \sup_{|\xi| \leq r^{-1}} |q(y,\xi)|.
	\end{equation*}
    Because of monotone convergence, we need to consider only bounded stopping times, i.e.\ $\Pp^x(\sigma \leq M)=1$ for some $M>0$. Fix $x \in \rd$, $r>0$, and pick $u \in C_c^{\infty}(\rd)$ such that $u(0)=1$, $0 \leq u \leq 1$ and $\supp u \subseteq \ball{1}{0}$. If we set $u_r^x(y)\coloneqq  u((y-x)/r)$, then $M_t = u_r^x(X_t)-\int_0^t Au_r^x(X_s) \, ds$ is a martingale, and so the optional sampling theorem yields $\Ee^x M_0 = \Ee^x M_{\sigma \wedge \tau_r^x}$, i.e.\
	\begin{equation*}
		\Ee^x \left[ u_r^x(X_{\sigma \wedge \tau_r^x})\right] -1
		= \Ee^x \left( \int_{[0,\sigma \wedge \tau_r^x)} Au_r^x(X_s) \, ds \right).
	\end{equation*}
	Thus,
	\begin{align*}
		\Pp^x\left(\tau_r^x \leq \sigma\right)
		\leq 1 - \Ee^x \left[u_r^x(X_{\sigma})\right]
		&= -  \Ee^x \left( \int_{[0,\sigma \wedge \tau_r^x)} Au_r^x(X_s) \, ds \right) \\
		&= - \Ee^x \left( \int_{[0,\sigma \wedge \tau_r^x)} \I_{\{|X_s-x|<r\}} Au_r^x(X_s) \, ds \right).
	\end{align*}
	Since we have
	\begin{align*}
		-Au_r^x(y)
		 = \int_{\rd} e^{iy \cdot \xi} q(y,\xi) \widehat{u_r^x}(\xi) \, d\xi
		 &= e^{-ix \cdot \xi}\cdot r^d \int_{\rd} e^{iy \cdot \xi} q(y,\xi) \widehat{u}(r\xi) \, d\xi \\
		 &= e^{-ix \cdot \xi} \int_{\rd} e^{iy \cdot \xi} q(y,r^{-1} \xi) \widehat{u}(\xi) \, d\xi
	\end{align*}
	for all $y \in \rd$, we find that
	\begin{align}\label{fel-eq17}
		\Pp^x(\tau_r^x \leq \sigma)
        \leq \Ee^x \left[ \int_{[0,\sigma \wedge \tau_r^x)} \int_{\rd} |q(X_s ,r^{-1}\xi)| \, |\widehat{u}(\xi)| \, d\xi \, \I_{\{|X_s-x|<r\}} \, ds \right].
	\end{align}
    Using the growth estimate $\sup_{x \in K} |q(x,\xi)| \leq c_{K} (1+|\xi|^2)$ for $q$ with the constant $c_K \coloneqq  2 \sup_{x \in K} \sup_{|\eta| \leq 1} |q(x,\eta)|$ for any compact set $K \subseteq \rd$, cf.\ B\"{o}ttcher \emph{et al.}~\cite[Thm.~2.31]{matters3}, we get
	\begin{align*}
		 \int\limits_{\rd} \sup_{|y-x| \leq r} |q(y,r^{-1}\xi)| \, |\widehat{u}(\xi)| \, d\xi
		 &\leq 2\sup_{|y-x|\leq r} \sup_{|\eta| \leq 1/r} |q(y,\eta)| \int\limits_{\rd} \left(1+|\xi|^2 \right) |\widehat{u}(\xi)| \, d\xi \\
		 &\leq C \sup_{|y-x|\leq r} \sup_{|\eta| \leq 1/r} |q(y,\eta)|
	\end{align*}
	for the absolute constant $C \coloneqq  2\int_{\rd} \left(1+|\xi|^2 \right) |\widehat{u}(\xi)| \, d\xi < \infty$. In summary,
	\begin{gather}
		\Pp^x\left(\tau_r^x \leq \sigma\right)
		\leq C \Ee^x\left[\sigma \wedge \tau_r^x\right] \sup_{|y-x|\leq r} \sup_{|\eta| \leq 1/r} |q(y,\eta)|.
		\label{fel-eq19}
	\end{gather}

    \secundolist
    From
	\begin{equation*}
		\Pp^x \left( \sup_{s \leq t} |X_s-x| \geq r \right)
		= \lim_{n \to \infty} \Pp^x \left( \sup_{s \leq t} |X_s-x|> r - \frac{1}{n} \right)
	\end{equation*}
	and Step~\primo\ we get
	\begin{align*}
		\Pp^x \left( \sup_{s \leq t} |X_s-x| \geq r \right)
		&\leq C \Ee^x\left[\sigma\right] \lim_{n \to \infty} \sup_{|y-x| \leq r-1/n} \sup_{|\xi| \leq 1/(r-1/n)} |q(y,\xi)| \\
		&\leq C \Ee^x\left[\sigma\right] \sup_{|y-x| \leq r} \sup_{|\xi| \leq 2/r} |q(y,\xi)|.
	\end{align*}
	Since $\xi \mapsto \sqrt{|q(y,\xi|}$ is subadditive, cf.\ \eqref{fel-eq5}, we have $|q(y,2\xi)| \leq 4 |q(y,\xi)|$, and so
	\begin{equation*}
		\Pp^x \left( \sup_{s \leq t} |X_s-x| \geq r \right)
		\leq 4C \Ee^x\left[\sigma\right] \sup_{|y-x| \leq r} \sup_{|\xi| \leq 1/r} |q(y,\xi)|.
    \qedhere
	\end{equation*}
\end{proof}

As an immediate consequence, we obtain the following lower bound for the expectation of stopping times.  This is also an alternative way to derive the lower bound in \eqref{lp-e48} for L\'evy processes.

\begin{corollary} \label{fel-7}  \index{inequality!mean exit time}
	Let $(X_t)_{t \geq 0}$ be a L\'evy-type process with symbol $q(x,\xi)$. There is an absolute constant $C>0$ such that
	\begin{equation*}
		\Ee^x\left[\sigma\right]
		\geq C \sup_{r>0} \frac{\Pp^x\left(\tau_r^x\leq \sigma < \infty\right)}{\sup_{|y-x| \leq r} \sup_{|\xi| \leq 1/r} |q(y,\xi)|}
	\end{equation*}
	for all $x \in \rd$ and any stopping time $\sigma$. In particular,
	\begin{equation*}
		\Ee^x\left[\tau_r^x \right]
		\geq C \frac{\Pp^x\left(\tau_r^x<\infty\right)}{\sup_{|y-x| \leq r} \sup_{|\xi| \leq 1/r} |q(y,\xi)|}
	\end{equation*}
\end{corollary}

\begin{proof}
	From \eqref{fel-eq19} with $\sigma \rightsquigarrow \sigma \wedge t$, we see that
	\begin{align*}
		\Pp^x\left(\tau_r^x \leq \sigma \wedge t\right)
		\leq C \Ee^x\left[t \wedge \sigma \wedge \tau_r^x\right] \sup_{|y-x|\leq r} \sup_{|\eta| \leq 1/r} |q(y,\eta)|.
	\end{align*}
	Letting $t \to \infty$ yields
	\begin{equation*}
		\Pp^x\left(\tau_r^x \leq \sigma< \infty\right)
		\leq C \Ee^x\left[\sigma \right] \sup_{|y-x|\leq r} \sup_{|\eta| \leq 1/r} |q(y,\eta)|
	\end{equation*}
    for any $r>0$. Rearranging the terms gives the first assertion, and then the second one is immediate if we take $\sigma\coloneqq \tau_r^x$.
\end{proof}

Next we present a localized version of the maximal inequality, which deals with the small-time asymptotics of $\Pp^x \left( \sup_{s \leq t} |X_s-x|>r \right)$. It has been established in K\"{u}hn \& Schilling~\cite{ihke} for Feller processes, but the proof remains valid for the wider class of L\'evy-type processes.

\begin{theorem} \label{fel-9}
	Let $(X_t)_{t \geq 0}$ be a L\'evy-type process with symbol $q(x,\xi)$. Assume that $x \mapsto q(x,\xi)$ is continuous for any $\xi \in \rd$. There is an absolute constant $C>0$ such that
	\begin{gather}
		\limsup_{t \to 0} \frac{1}{t} \Pp^x \left( \sup_{s \leq t} |X_s-x| \geq r \right)
		\leq C \sup_{|\xi| \leq 1/r} |q(x,\xi)|
		\label{fel-eq23}
	\end{gather}
	for any $x \in \rd$ and $r>0$.
\end{theorem}

\begin{remark}
    The symbol $q$ is continuous with respect to the $x$-variable if, and only if, the pseudo-differential operator $A$ with symbol $q$ satisfies $Af \in C(\rd)$ for any $f \in C_c^{\infty}(\rd)$. There is also a characterization in terms of the triplets
    $(b(x),Q(x),\nu(x,dy))$, see K\"{u}hn \& Kunze~\cite[Thm.~B.1]{kunze21}. If $(X_t)_{t \geq 0}$ is a Feller process and the domain of the (strong) infinitesimal generator contains $C_c^{\infty}(\rd)$, then $x \mapsto q(x,\xi)$ is continuous if, and only if, $x\mapsto q(x,0)$ is continuous; if we have a conservative Feller process, then this is always satisfied since $q(x,0)=0$, cf.\ Schilling~\cite[Thm.~4.4]{rs-cons} or B\"{o}ttcher \emph{et al.}~\cite[Thm.~2.30]{matters3}.
\end{remark}

\begin{proof}[Proof of Theorem~\ref{fel-9}]
	Let $u$ be the function from the proof of Theorem~\ref{fel-5}. From \eqref{fel-eq17} with $\sigma \equiv t$, we see that
	\begin{align*}
		\Pp^x(\tau_r^x \leq t)
		\leq t \,\Ee^x \left( \int_{\rd} \sup_{s<t \wedge \tau_r^x} |q(X_s,r^{-1}\xi)| \, |\widehat{u}(\xi)| \, d\xi \, ds \right).
	\end{align*}	
	Since $X_s \in \cball{r}{x}$ for all $s<t \wedge \tau_r^x$, there is a constant $C=C(r,x)$ such that
	\begin{equation*}
		\sup_{s<t \wedge \tau_r^x} |q(X_s,r^{-1}\xi)|
		\leq C (1+|\xi|^2) |\widehat{u}(\xi)| \in L^1(d\xi)
	\end{equation*}
	for all $t \geq 0$. Moreover, the continuity of $x \mapsto q(x,\xi)$ entails that
	\begin{equation*}
		\lim_{t \to 0} \sup_{s<t \wedge \tau_r^x} |q(X_s,r^{-1}\xi)| \, |\widehat{u}(\xi)|
		= |q(x,r^{-1}\xi)| \, |\widehat{u}(\xi)|.
	\end{equation*}
    Applying dominated convergence yields
	\begin{equation*}
		\limsup_{t \to 0} \frac{1}{t} \Pp^x(\tau_r^x \leq t)
		\leq \int_{\rd} |q(x,r^{-1}\xi)| \, |\widehat{u}(\xi)| \, d\xi.
	\end{equation*}
	Taking into account \eqref{fel-eq11} and the estimate
	\begin{equation*}
		|q(x,r^{-1}\xi)| \leq 2 \sup_{|\eta| \leq 1/r} |q(x,\eta)| (1+|\xi|^2)
	\end{equation*}
	gives
	\begin{equation*}
		\limsup_{t \to 0} \frac{1}{t} \Pp^x \left( \sup_{s \leq t} |X_s-x| > r \right)
		\leq C \sup_{|\xi| \leq 1/r} |q(x,\xi)|.
	\end{equation*}

    In order to get \eqref{fel-eq23} with \enquote{$\geq r$}, we can almost literally follow the strategy used in Step~\secundo\ of the proof of Theorem~\ref{fel-5}.
\end{proof}

\begin{example} \label{fel-10} \index{stable-like process}
    Let $(X_t)_{t \geq 0}$ be a process of variable order, i.e.\ a L\'evy-type process with symbol $q(x,\xi) = |\xi|^{\alpha(x)}$ for a continuous function $\alpha\colon \rd \to (0,2]$, cf.\ Example~\ref{fel-3} No.~\ref{fel-3-5}. Then Theorem~\ref{fel-5} and Theorem~\ref{fel-9} show that
	\begin{gather*}
		\Pp^x \left( \sup_{s \leq t} |X_s-x| \geq r \right) \leq C\,t\, r^{-\alpha_*(x,r)}
	\intertext{and}
		\limsup_{t \to 0} \frac{1}{t} \Pp^x \left( \sup_{s \leq t} |X_s-x| \geq r \right)
		\leq C  r^{-\alpha(x)}
	\end{gather*}
	for $r > 0$ and some absolute constant $C>0$, where $\alpha_*(x,r) = \min_{|z-x| \leq r} \alpha(z)$.
\end{example}

\begin{example} \label{fel-11} \index{Levy-driven SDE@L\'evy-driven SDE}
	Let $(X_t)_{t \geq 0}$ be the unique weak solution to an SDE
	\begin{equation*}
		dX_t = \sigma(X_{t-}) \, dL_t,
	\end{equation*}
	for a one-dimensional L\'evy process $(L_t)_{t \geq 0}$ and $\sigma \in C_b(\real)$. The symbol of $(X_t)_{t \geq 0}$ is $q(x,\xi) = \psi(\sigma(x)\xi)$ where $\psi$ is the characteristic exponent of $(L_t)_{t \geq 0}$, cf.\ Example~\ref{fel-3}, and so Theorems~\ref{fel-5} and \ref{fel-9} yield
	\begin{gather*}
		\Pp^x \left( \sup_{s \leq t} |X_s-x| \geq r \right)
		\leq C t \sup_{|y-x| \leq r} \sup_{|\xi| \leq r^{-1}} |\psi(\sigma(y)\xi)|
	\intertext{and}
		\limsup_{t \to 0} \frac{1}{t} \Pp^x \left( \sup_{s \leq t} |X_s-x| \geq r \right)
		\leq C  |\psi(\sigma(x)\xi)|.
	\end{gather*}	
	If $\sigma$ satisfies a growth estimate $|\sigma(x)| \leq M (1+|x|^{\gamma})$ for some constants $\gamma \geq 0$, $M>0$, then we get
	\begin{equation*}
		\Pp^x \left( \sup_{s \leq t} |X_s-x| \geq r \right)
		\leq C't \sup_{|\eta| \leq r^{\gamma-1}} |\psi(\eta)|
        = |\psi|^*(r^{\gamma-1})
	\end{equation*}
	for $r \gg 1$.
\end{example}

So far, we have established upper bounds for the tails of the maximum. Next we turn to lower bounds. Lower bounds for $\Pp^x(\sup_{s \leq t}|X_s-x| > r)$ are of course equivalent to upper bounds for the probability $\Pp^x(\sup_{s \leq t} |X_s-x|\leq r)$ of the complementary event. Since we know from \eqref{fel-eq11} that
\begin{equation*}
	\Pp^x \left( \tau_r^x \geq t \right)
	\geq \Pp^x \left( \sup_{s \leq t} |X_s-x| \leq r \right)
	\geq \Pp^x(\tau_r^x > t)
	\geq \Pp^x \left( \sup_{s \leq t} |X_s-x| < r\right),
\end{equation*}
this means that the estimates can also be stated as tail estimates for the first exit time $\tau_r^x$.   \par
We start with a bound in terms of the characteristics $(b(x),Q(x),\nu(x,dy))$. The basic idea is quite simple: the process $(X_t)_{t \geq 0}$ leaves the ball $\ball{r}{x}$ centred at the starting point $x$ as soon as a jump of modulus larger than $2r$ occurs. Thus we can estimate the exit probability by considering only the compound pseudo Poisson process (see Example~\ref{fel-3} No.~\ref{fel-3-4}) accounting for the jumps of $(X_t)_{t\geq 0}$ exceeding absolute size $2r$. The likelihood that a jump of size $>2r$ happens has a lower bound which is proportional to $\inf_z \nu\left(z,\left\{y \in \rd \mid |y|>2r\right\}\right)$.

\begin{proposition} \label{fel-13} \index{inequality!maximal tail}
	Let $(X_t)_{t \geq 0}$ be a L\'evy-type process with infinitesimal characteristics $(b(x),Q(x),\nu(x,dy))$, $x \in \rd$. If $\sigma$ is a finite stopping-time, then
	\begin{gather}\label{fel-eq26}
		\Pp^x \left( \sup_{s \leq \sigma} |X_s-x| > r \right) \geq \Ee^x\left[\tau_r^x \wedge \sigma\right] \cdot G(x,2r)
	\intertext{for all $x \in \rd$ and $r>0$, where\footnotemark}\label{fel-eq27}
		G(x,r) \coloneqq  \inf_{|z-x| \leq r} \nu\left(z, \left\{y \in \rd \mid |y| >r\right\}\right).
	\intertext{In particular,}\label{fel-eq28}
		\Pp^x \left( \sup_{s \leq t} |X_s-x|>r \right) \geq \frac{t\,G(x,2r)}{1+t\, G(x,2r)}.
	\end{gather}
\footnotetext{Our notation is chosen to be in-line with the truncated and censored second moments which were introduced on page~\ref{pre-e75}, Eq.~\eqref{pre-e75}.}
\end{proposition}

\noindent
Note that \eqref{fel-eq28} is equivalent to
\begin{gather}\label{fel-eq29}
	\Pp^x \left( \sup_{s \leq t} |X_s-x| \leq r \right) \leq \frac{1}{1+t\, G(x,2r)};
\end{gather}
this bound will be improved in Corollary~\ref{fel-15} below.

\begin{proof}[Proof of Proposition~\ref{fel-13}]
    Fix $x \in \rd$, $\epsilon>0$ and $r>0$. By monotone convergence, it suffices to prove the assertion for bounded stopping times $\sigma$. Pick a function $u \in C_c^{\infty}(\rd)$ such that $\I_{\ball{r+\epsilon}{x}} \leq u \leq \I_{\ball{r+2\epsilon}{x}}$. Denote by $\tau_r^x$ the first exit time from the ball $\cball{r}{x}$. As $u(X_t)=1$ for $t<\tau_r^x$, we find from Dynkin's formula that
	\begin{align*}
		\Pp^x\left(\tau_r^x>\sigma\right)
		&= \Ee^x \left[u(X_\sigma) \I_{\{\tau_r^x>t\}}\right]\\
		&\leq \Ee^x \left[u(X_{\sigma \wedge \tau_r^x})\right]
		= 1+ \Ee^x \left[ \int_{(0,\sigma \wedge \tau_r^x)} Au(X_s) \, ds \right].
	\end{align*}
	 For $z \in \ball{r+\epsilon}{x}$ we have $u(z)=1$, $\nabla u(z)=0$ and $\nabla^2 u(z)=0$, so that
	\begin{equation*}
		Au(z) = \int_{y\neq 0} \left(u(z+y)-1\right)  \nu(z,dy), \quad z \in \ball{r+\epsilon}{x}.
	\end{equation*}
	Since the integrand is negative and $u = 0$ outside the ball $\ball{r+2\epsilon}{x}$, it follows that
	\begin{equation*}
		Au(z)
		\leq \int_{|(z+y)-x| \geq r+2\epsilon} (u(z+y)-1) \, \nu(z,dy)
		\leq - \int_{|y| \geq 2r+3\epsilon} \nu(z,dy)
	\end{equation*}
	for all $z \in \ball{r+\epsilon}{x}$. By definition, $X_s \in \cball{r}{x}$ for $s<\tau_r^x$, and so we may use the bound for $z=X_s$, $s<\tau_r^x$, to find that
	\begin{align*}
		\Pp^x\left(\tau_r^x>\sigma\right)
		&\leq 1- \Ee^x \left[ \int_{(0,\sigma \wedge \tau_r^x)} \nu\left(X_s,\left\{y \in \rd \mid  |y| \geq 2r+3\epsilon\right\}\right)  ds \right] \\
		&\xrightarrow[]{\;\epsilon \downarrow 0\;}
            1- \Ee^x \left[ \int_{(0,\sigma \wedge \tau_r^x)} \nu\left(X_s,\left\{y \in \rd \mid  |y| \geq 2r\right\}\right)  ds \right].
	\end{align*}
	Hence,
	\begin{align}\label{fel-eq31}
		\Pp^x\left(\tau_r^x>\sigma\right)
		&\leq 1- \Ee^x\left[\tau_r^x \wedge \sigma\right]  \inf_{|z-x| \leq r} \nu\left(z, \left\{y \in \rd \mid  |y| >2r\right\}\right).
    \\\notag
		&\leq 1-  \Ee^x\left[\tau_r^x \wedge \sigma\right] G(x,2r).
	\end{align}
    Since $\{\tau_r^x \geq \sigma\} = \bigcap_n \{\tau_r^x>\sigma-1/n\}$, a standard continuity-of-measure argument shows that the same estimate holds for $\Pp^x(\tau_r^x \geq \sigma)$. In particular, by \eqref{fel-eq11},
	\begin{equation*}
		\Pp^x \left( \sup_{s \leq \sigma} |X_s-x| > r \right)
        \geq \Ee^x\left[\tau_r^x \wedge \sigma\right]  G(x,2r).
	\end{equation*}
    For $\sigma \equiv t$ we use in \eqref{fel-eq31} the elementary estimate $\Ee^x \left[t \wedge \tau_r^x\right] \geq  t \Pp^x(\tau_r^x \geq t)$ to see
	\begin{align*}
		\Pp^x\left(\tau_r^x \geq t\right)
		&\leq 1- t \Pp^x\left(\tau_r^x \geq t\right)   G(x,2r),
	\end{align*}
	i.e.
	\begin{equation*}
		\Pp^x \left( \sup_{s \leq t} |X_s-x| \leq r \right)
		\leq \Pp^x\left(\tau_r^x \geq t\right)
		\leq \frac{1}{1+t\, G(x,2r)}.
	\end{equation*}
	In other words,
	\begin{equation*}
		\Pp^x \left( \sup_{s \leq t} |X_s-x|>r \right) \geq \frac{t\,G(x,2r)}{1+t \, G(x,2r)}.
    \qedhere
	\end{equation*}
\end{proof}

Let $N_t \coloneqq  \# \left\{s \leq t \mid |\Delta X_s|>2r\right\}$ be the process counting the jumps having modulus larger than $2r$. Clearly, $\left\{\sup_{s \leq t} |X_s-x| \leq r\right\} \subseteq \left\{N_t=0\right\}$. For the particular case that $(X_t)_{t \geq 0}$ is a L\'evy process, it is known that $(N_t)_{t \geq 0}$ is a Poisson process with intensity $\lambda = \nu\left(\left\{y \in \rd \mid  |y|>2r\right\}\right)$, where $\nu$ denotes the L\'evy measure, see e.g.\ Schilling~\cite[Ch.~9]{crm}. Consequently,
\begin{align*}
    \Pp^x \left( \sup_{s \leq t} |X_s-x| \leq r\right)
	&\leq \Pp^x(N_t=0)\\
	&= e^{-\lambda t}
	= \exp \left[- t \nu\left(\left\{y \in \rd \mid  |y|>2r\right\} \right)\right]
\end{align*}
for a L\'evy process $(X_t)_{t \geq 0}$. If $(X_t)_{t \geq 0}$ is a general L\'evy-type process, $(N_t)_{t \geq 0}$ is no longer a Poisson process (but a pseudo Poisson process), but the following result shows that we can still expect a similar estimate in terms of the jump intensity.

\begin{corollary} \label{fel-15} \index{inequality!maximal tail}
	Let $(X_t)_{t \geq 0}$ be a L\'evy-type process. Then
	\begin{gather}\label{fel-eq35}
		\Pp^x \left( \sup_{s \leq t} |X_s-x| \leq r \right) \leq \exp(-t \,G(x,5r)), \quad x \in \rd,\,r>0,\,t>0
	\intertext{and}\label{fel-eq36}
		\Ee^x \left[\tau_r^x\right] \leq \frac{1}{G(x,2r)}, \quad x \in \rd,\,r>0,
	\end{gather}
	for the function $G$ defined in \eqref{fel-eq27}.
\end{corollary}

\begin{proof}
	For fixed $x \in \rd$, $r>0$, $n \in \nat$ and $t>0$, set $t_i \coloneqq  \frac{i}{n} t$, then
	\begin{align*}
		&\Pp^x \left( \sup_{s \leq t} |X_s-x| \leq r \right)\\
		&\quad\leq \Pp^x \left( \bigcap_{i=0}^{n-1} \big\{ |X_{t_i}-x| \leq r, \sup_{s \leq t_{i+1}-t_i} |X_{t_i+s}-X_{t_i}| \leq 2r \big\} \right).
	\end{align*}
	Applying the Markov property and using that, by Proposition~\ref{fel-13},
	\begin{equation*}
		\Pp^y \left( \sup_{s \leq \rho} |X_s-y| \leq 2r \right)
		\leq \frac{1}{1+\rho\, G(y,4r)}
		\leq \frac{1}{1+\rho\, G(x,5r)}
	\end{equation*}
	for any $y \in \cball{r}{x}$ and $\rho>0$, we get
	\begin{equation*}
		\Pp^x \left( \sup_{s \leq t} |X_s-x| \leq r \right)
		\leq \left( \frac{1}{1+\frac{t}{n}\, G(x,5r)} \right)^n
		\xrightarrow[]{\;n \to \infty\;} \exp \left[-t\,G(x,5r) \right],
	\end{equation*}
	which proves \eqref{fel-eq35}. For \eqref{fel-eq36}, we note that by \eqref{fel-eq31} with $\sigma \equiv t$
	\begin{equation*}
		\Ee^x\left[\tau_r^x \wedge t\right] \leq \frac{1-\Pp^x(\tau_r^x > t)}{G(x,2r)} \leq \frac{1}{G(x,2r)}.
	\end{equation*}
	Letting $t \to \infty$ using Fatou's lemma gives \eqref{fel-eq36}.
\end{proof}

The rate in \eqref{fel-eq35} is optimal for a wide class of jump process. However, the estimate uses only the tails of the L\'evy measure, and therefore not the full information is taken into account, as the following example shows. Recall that we have already seen a similar phenomenon in connection with L\'evy processes, see Theorem~\ref{lp-43} and its proof.

\begin{example} \label{fel-17} \index{stable L\'evy process}
    Let $(X_t)_{t \geq 0}$ be an $\alpha$-stable L\'evy process for $\alpha \in (0,1)$. Taylor~\cite{taylor67} has shown that the bounds for $\Pp\left(\sup_{s\leq t}|X_s| \leq r\right)$ depend substantially on the support of the L\'evy measure $\nu$: if the spectral measure (i.e.\ the spherical part of $\nu$) is concentrated in a hemisphere $\left\{y \in \mathds{S}^{d-1} \mid  y_i \geq 0\right\}$ for some $i \in \{1,\ldots,d\}$, then
	\begin{gather}\label{fel-eq38}
		\Pp\left(\sup_{s\leq t}|X_s| \leq r\right) \asymp \exp\left(-ct^{1/(1-\alpha)} r^{-\alpha/(1-\alpha)}\right)
	\end{gather}
	for small $r>0$, where the constants $c$ in the lower and upper bound may differ. For all other stable processes, it holds that
	\begin{gather}\label{fel-eq39}
			\Pp\left(\sup_{s\leq t}|X_s| \leq r\right) \asymp \exp\left(-ct r^{-\alpha}\right)
	\end{gather}
    for small $r>0$, again with possibly different constants in the lower and upper bound.  Corollary~\ref{fel-15} yields the estimate $\Pp\left(\sup_{s\leq t}|X_s| \leq r\right) \leq \exp\left[-ctr^{-\alpha}\right]$ for any stable process, i.e.\ it does not capture the information on the support of the L\'evy measure.
\end{example}

\begin{theorem} \label{fel-19} \index{inequality!maximal tail}
	Let $(X_t)_{t \geq 0}$ be a L\'evy-type process with symbol $q(x,\xi)$. Then
	\begin{gather}
		\Pp^x \left( \sup_{s \leq t} |X_s-x| \leq r \right)
		\leq
        \dfrac{c}{1 + t \sup\limits_{|\xi| \leq 1/(rk(x,r))} \inf\limits_{|y-x| \leq r} \RE q(y,\xi)},
		\label{fel-eq40}
	\intertext{where $c\coloneqq 4/\cos \sqrt{\tfrac{2}{3}}$ and\footnotemark}\notag
        k(x,r)
        \coloneqq  \inf\left\{k \geq \Big(\arccos \sqrt{\tfrac{2}{3}}\Big)^{-1} \mid \inf_{|\xi| \leq 1/(kr)} \inf_{|y-x| \leq r} \frac{\RE q(y,\xi)}{|\xi| \, |\IM q(y,\xi)|} \geq 2r \right\}.
	\end{gather}
\footnotetext{The definition of $k(x,r)$ in B\"{o}ttcher \emph{et al.}~\cite[Theorem~5.5]{matters3} has a misprint which we corrected here: $\sup\limits_{(\cdots)}\sup\limits_{(\cdots)}$ should read $\inf\limits_{(\cdots)}\inf\limits_{(\cdots)}$.}
\end{theorem}

The estimate \eqref{fel-eq40} was first shown in Schilling~\cite[Lem.~6.3]{rs-growth} for Feller processes whose symbol satisfy a sector condition, see also B\"{o}ttcher \emph{et al.}~\cite[Thm.~5.5]{matters3}.\footnote{The additive constant \enquote{$1+$} in the denominator in \eqref{fel-eq40} comes from the trivial observation that a probability is less than $1$, i.e.\ taking the minimum with $1$ in $\Pp^x(\dots)\leq \frac cp$, we get $\frac cp \wedge 1 \leq c' (1+p)^{-1}$.} It turns out that the proof goes through for the wider class of L\'evy-type processes without any major modifications, and so wo do not present the proof in this exposition. \par

The definition of $k(x,r)$ is tailor-made in such a way that
\begin{equation*}
	\RE q(y,\xi) \geq 2r |\xi| \, |\IM q(y,\xi)|
\end{equation*}
for all $|\xi| \leq 1/(rk(x,r))$ and $|y-x| \leq r$ -- this is exactly what is needed for some technical estimates in the proof of Theorem~\ref{fel-19}. If the symbol $q$ satisfies the \emph{sector condition}, i.e.\
\begin{gather}
	|\IM q(x,\xi)| \leq C \RE q(x,\xi), \quad x,\xi \in \rd, \label{fel-eq42}
\end{gather}
for some constant $C>0$, then $k(x,r) = k_0 \coloneqq  \max\left\{2C,\big(\arccos \sqrt{{2}/{3}}\big)^{-1}\right\}$, and so we arrive at the following statement.

\begin{corollary} \label{fel-21}
    Let $(X_t)_{t \geq 0}$ be a L\'evy-type process with symbol $q(x,\xi)$ satisfying the sector condition. Then there is a constant $c>0$ depending only on the sector constant such that
	\begin{gather}\label{fel-eq44}
		\Pp^x \left( \sup_{s \leq t} |X_s-x| \leq r \right)
		\leq \dfrac{c}{1+ t \sup\limits_{|\xi| \leq 1/(rk_0)} \inf\limits_{|y-x| \leq r} \RE q(y,\xi)}
	\end{gather}
	for all $x \in \rd$, $r>0$ and $t >0$.
\end{corollary}

Using an iteration argument based on the Markov property -- similar to the proof of Corollary~\ref{fel-15} --, we can improve the above estimates to obtain an exponential bound.

\begin{corollary} \label{fel-23}
	Let $(X_t)_{t \geq 0}$ be a L\'evy-type process with symbol $q(x,\xi)$. Then
	\begin{gather*}
		\Pp^x \left( \sup_{s \leq t} |X_s-x| \leq r \right)
		\leq \exp\left[ -c t \sup_{|\xi| \leq 1/(2rk^*(x,r))} \inf_{|y-x| \leq 3r} \RE q(y,\xi)\right]
	\intertext{for $c\coloneqq  \cos \sqrt{\tfrac{2}{3}}/(4e)$ and}
		k^*(x,r)
        \coloneqq  \inf\left\{k \geq \Big(\arccos \sqrt{\tfrac{2}{3}}\Big)^{-1} \mid \inf_{|\xi| \leq 1/(2kr)} \inf_{|y-x| \leq 3r} \frac{\RE q(y,\xi)}{|\xi| \, |\IM q(y,\xi)|} \geq 4r \right\}.
	\intertext{Moreover,}
		\Ee^x \left[\tau_r^x\right]
        \leq \dfrac{C}{\sup\limits_{|\xi| \leq 1/(2rk^*(x,r))} \inf\limits_{|y-x| \leq 3r} \RE q(y,\xi)}.
	\end{gather*}
    If the symbol $q$ satisfies the sector condition \eqref{fel-eq42}, then $k^*(x,r)$ may be replaced by $k_0 = \max\left\{2C,\, \big(\arccos \sqrt{{2}/{3}}\big)^{-1}\right\}$.
\end{corollary}

\begin{remark}\label{fel-25}
Maximal inequalities and, in particular, the estimates on the mean exit time have interesting applications in the study of path properties and
distributional properties of L\'evy-type processes. For example, one can prove
\begin{itemize}
\item
    \emph{Harnack inequalities} for L\'evy-type generators following the strategy of Bass \& Levin~\cite{bas-lev02} -- see also Song \& Vondra\v{c}ek~\cite{son-von04} for a systematic exposition of this method and the comments in Example~\ref{lp-47};  this method is explained in Schilling \& Uemura~\cite[\S5]{schilling-uemura}.
\item
    \emph{tightness estimates} for L\'evy-type process $(X_t)_{t\geq 0}$ which can be applied to show the continuity of expressions of the form $x\mapsto \Ee^x \left[f(X_{\tau_U})\right]$ for exit times $\tau_U$, the \emph{\textup{(}strong\textup{)} Feller property} for certain L\'evy-type processes, see Schilling \& Wang~\cite{schilling-wang12}, and the existence of \emph{Markovian selections} for martingale problems, see K\"{u}hn~\cite{markovian-selection}.
\item
    \emph{moment estimates} for L\'evy-type processes using the layer-cake formula (Lemma~\ref{pre-10}), see Deng \& Schilling\cite{deng-rs}, K\"{u}hn~\cite{ltp-moments} and Theorem~\ref{lp-53} for an illustration of the method in the L\'evy case.
\item
    results on \emph{upper and lower functions} and \emph{small-time asymptotics} for L\'evy-type processes as in Knopova \& Schilling \cite{kno-schi14}, K\"{u}hn \cite{upper-fn}, K\"{u}hn \& Schilling \cite{ihke}.
\item
    estimates on the \emph{Hausdorff dimension} of the image of a Feller process Schilling~\cite{schilling98}, Knopova \emph{et al.} \cite{kno-schi-wang}.
\item
    results on the \emph{p-variation} of the sample paths, see Manstavi\v{c}ius~\cite{manstavicius04}, Manstavi\v{c}ius \& Schnurr~\cite{mansta-schnurr}.
\item
	results on \emph{random time changes} of L\'evy-type processes, cf.\ Kr\"{u}hner \& Schnurr~\cite{kruehner-schnurr}, K\"{u}hn~\cite{perpetual}.
\end{itemize}
\end{remark}

\section{Markov Processes}\label{mp}

Let us resume the discussion started in Chapter~\ref{gen} and turn to a situation where we want to make as few assumptions on the stochastic process as possible. Throughout, $(\Omega,\Ascr,\Pp,\Fscr_t)$ is a filtered probability space which satisfies the usual conditions, i.e.\ it is complete and the filtration is right continuous and $\Fscr_0$ contains all $\Pp$-null sets. Let $(X_t)_{t\geq 0}$ be a strong Markov process taking values in $\rd$. We will consider only temporally homogeneous processes; in this case, $p_t(x,dy) = \Pp^x(X_t\in dy)$ is the transition function, and we write $\Ee^x(\dots) = \int_\Omega \dots d\Pp^x$ for the corresponding expectation. We assume that $t\mapsto X_t(\omega)$ is, almost surely for all $\Pp^x$, right continuous with finite left limits (c\`adl\`ag). The standard argument how one can construct a c\`adl\`ag modification can be found in Revuz \& Yor~\cite[Thm.~III.(2.7)]{rev-yor99} (for Feller processes, i.e.\ Markov processes whose transition semigroup $P_t u(x) \coloneqq  \Ee^x u(X_t) = \int u(y)\,p_t(x,dy)$ preserves the family $C_\infty$ of continuous functions vanishing at infinity) or in Sharpe~\cite[Thm.~I.(2.7)]{sharpe88} (for processes satisfying Meyer's first \emph{hypoth\`ese droite} (HD1)).

\begin{definition}\label{mp-05}\index{uniform stochastic continuity}
    A (strong) Markov process $(X_t)_{t\geq 0}$ is said to be \emph{uniformly stochastically continuous} if
    \begin{equation}\label{mp-e06}
        \forall \epsilon>0,\; \forall t\geq 0\::\quad
        \lim_{s\to t} \sup_{x\in\rd} \Pp^{x} (|X_s-X_t|>\epsilon)=0.
    \end{equation}
\end{definition}
Because of the Markov property, it is enough to require \eqref{mp-e06} for $t=0$.

Since we do not assume further structural properties on a Markov process, we will get only a few results on the maximal process $X_t^*$. The first result is a version of Etemadi's inequality, see Theorem~\ref{lp-39}.

\begin{theorem}[Etemadi-type maximal inequality]\label{mp-07} \index{inequality!Etemadi}
    Assume that $(X_t)_{t\geq 0}$ is a strong Mar\-kov process. For all $a>0$ and $x\in\rd$ one has
    \begin{equation}\label{mp-e08}
        \Pp^x \left(\sup_{s\leq t} |X_s-x| > 3 a\right)
        \leq  2 \sup_{s\leq t} \sup_{y\in\coball{3a}{x}} \Pp^y (|X_s-y| > a).
    \end{equation}
\end{theorem}
\begin{proof}
    Denote by $\tau \coloneqq  \tau^x_{3a} \coloneqq  \inf\left\{ s>0 \mid |X_s-x|>3a\right\}$ the first exit time from the ball $\cball{3a}{x}$. Because of the inclusion $\left\{\sup_{s\leq t} |X_s-x|>3 a\right\} \subseteq \left\{\tau \leq t\right\}$ we have
    \begin{equation*}
    \begin{split}
        \Pp^x &\left(\sup_{s\leq t} |X_s-x| > 3a\right)\\
        &\leq\Pp^x \left(\tau \leq t,\; |X_t-x| > a\right) + \Pp^x \left(\tau \leq t,\; |X_t-x|\leq a\right)\\
        &\leq \Pp^x \left(|X_t-x| > a\right) + \Pp^x\left(\tau \leq t,\; |X_t-x|\leq a\right)\\
        &\leq \Pp^x \left(|X_t-x| > a\right) + \Pp^x\left(\tau \leq t,\; |X_t-X_\tau| \geq 2a\right).
    \end{split}
    \end{equation*}
    In the last line we use the triangle inequality $|X_t-X_\tau| \geq |X_\tau - x| - |X_t-x|$ and the fact that $|X_\tau - x| \geq 3a$. We will now estimate the second term appearing on the right hand side. Using the strong Markov property we get
    \begin{equation}\label{mp-e10}\begin{aligned}
        \Pp^x &\left(\tau \leq t, \; |X_t-X_\tau| \geq 2a\right)\\
        &= \int_{\tau \leq t} \Pp^{X_\tau(\omega)}\left(|X_{t-\tau(\omega)} - X_0| \geq 2a\right) \Pp^x(d\omega)\\
        &\leq \sup_{s\leq t} \sup_{y\in\coball{3a}{x}} \Pp^y\left(|X_s-y| \geq 2a\right) \Pp^x\left(\tau \leq t\right)\\
        &\leq \sup_{s\leq t} \sup_{y\in\coball{3a}{x}} \Pp^y\left(|X_s-y| \geq 2a\right).
    \end{aligned}\end{equation}
    In the penultimate estimate we use that $X_\tau\in\coball{3a}{x}$. Both estimates together prove
    \begin{equation*}\begin{aligned}[b]
        \Pp^x &\left(\sup_{s\leq t} |X_s-x| > 3a\right)\\
        &\leq \Pp^x (|X_t-x| > a) + \Pp^x(\tau \leq t,\; |X_t-X_\tau| \geq 2a)\\
        &\leq \Pp^x (|X_t-x| > a) + \sup_{s\leq t}\sup_{ y\in\coball{3a}{x} } \Pp^y\left(|X_s-y| \geq 2a\right)\\
        &\leq \Pp^x (|X_t-x| > a) + \sup_{s\leq t}\sup_{ y\in\coball{3a}{x} } \Pp^y(|X_s-y| > a).
    \end{aligned}\qedhere\end{equation*}
\end{proof}

If we assume uniform stochastic continuity, then we get the following Kol\-mo\-go\-rov- or Ottaviani--Skorokhod-type inequality. Our proof follows -- and simplifies -- It\^o's presentation~\cite[Ch.~2.9]{ito-aarhus}. Alternatively, one could use the argument sketched in Remark~\ref{lp-41} and the penultimate estimate in \eqref{mp-e10}, without crudely estimating $\Pp^x\left(\tau\leq 1\right)$ with $1$ to get an equivalent version of the next lemma, see also K\"{u}hn~\cite[Lemma 5.4]{upper-fn}. Further proofs of similar results are in Blumenthal~\cite[Lem.~2.1, 2.2]{blumenthal57} and Gikhman \& Skorokhod~\cite[p.~420, Lem.~2]{gih-sko-1}.

\begin{theorem}[Ottaviani--Skorokhod-type maximal inequality]\label{mp-21} \index{inequality!Ottaviani--Skorokhod}
    Assume that \linebreak $(X_t)_{t\geq 0}$ is a uniformly stochastically continuous strong Markov process. For any $\epsilon>0$ there is some $\delta=\delta(\epsilon)$ such that for all $t,s\geq 0$ satisfying $0 < t-s <\delta$
    \begin{equation}\label{mp-e22}
        \Pp^x \left(\sup_{s\leq u,v\leq t} |X_u-X_v|>4\epsilon\right) \leq 4 \Pp^x \left(|X_s-X_t|>\epsilon\right).
    \end{equation}
\end{theorem}
\begin{proof}
    Pick any $\epsilon>0$. Because of uniform stochastic continuity, there is some $\delta = \delta(\epsilon)>0$ such that
    \begin{gather*}
        \sup_{x\in\rd, h<\delta} \Pp^x \left(|X_h -x|\geq \epsilon\right)\leq \frac 12,
    \end{gather*}
    or, equivalently,
    \begin{equation}\label{mp-e24}
        \inf_{x\in\rd,\, h<\delta} \Pp^x \left(|X_h -x|< \epsilon\right)\geq \frac 12.
    \end{equation}
    Using the triangle inequality, \eqref{mp-e22} follows from
    \begin{equation}\label{mp-e25}
        \Pp^x \left( \sup_{s\leq u\leq t} |X_u-X_s|>2\epsilon\right) \leq 2 \Pp^x \left(|X_s-X_t|>\epsilon\right).
    \end{equation}
    Denote by  $\tau = \tau_{\ball{2\epsilon}{z}}$  the first exit time from the open ball $\ball{2\epsilon}{z}$ of radius $2\epsilon$ and centre $z \in \rd$. Combining \eqref{mp-e24} and the strong Markov property yields for $t-s<\delta$
    \begin{align*}
        \Pp^z &\left(\tau\leq t-s\right)\\
        &= \Pp^z \left(|X_\tau -z|\geq 2\epsilon, \; \tau \leq t-s\right)\\
        &\leq \int \underbracket[.6pt]{2\Pp^{X_\tau(\omega)} \left(| X_{t-s-\tau(\omega)} -X_0|<\epsilon\right)}_{\geq 1,\text{\ by\ \eqref{mp-e24}}}\,
        \I_{\{| X_\tau-z|\geq 2\epsilon\}}(\omega) \I_{\{\tau \leq t-s\}}(\omega)\,\Pp^z(d\omega)\\
        &= 2\Pp^z \left( |X_\tau -z|\geq 2 \epsilon, \; \tau \leq t-s,\; |X_{t-s}-X_\tau|< \epsilon\right)\\
        &\leq 2\Pp^z \left(|X_{t-s}-z|>\epsilon,\; \tau \leq t-s\right)\\
        &\leq 2\Pp^z \left(|X_{t-s}-z|>\epsilon\right).
    \end{align*}
    Again by the Markov property,
    \begin{align*}
        \Pp^x \left( \sup_{s\leq u\leq t} |X_u-X_s|>2\epsilon\right)
        &= \Ee^x \left[\Pp^{X_s} \left(\sup_{u\leq t-s} |X_u-X_0 | >2\epsilon\right)\right]\\
        &\leq 2 \Ee^x \left[\Pp^{X_s}  \left(|X_{t-s}-X_0|>\epsilon\right)\right]\\
        &=2 \Pp^x \left(|X_t-X_s|>\epsilon\right),
    \end{align*}
    and this proves \eqref{mp-e25}.
\end{proof}

Maximal estimates of this type can be used to prove sample path regularity of strong Markov processes. For example, we have \emph{quasi-left continuity}, i.e.\ if $\sigma_n$ and $\sigma$ are stopping times, then
\begin{gather*}
    \sigma_n\uparrow\sigma<\infty\text{\ \ a.s.}
    \implies
    \Pp^x\left(\lim_{n\to\infty} X_{\sigma_n} = X_\sigma\right)=1,
\end{gather*}
see It\^o~\cite[Ch.~2.9, Thm.~2]{ito-aarhus}, or we can prove the following

\begin{theorem}[Dynkin--Kinney criterium]\label{mp-25} \index{Dynkin--Kinney criterium}
    Assume that $(X_t)_{t\geq 0}$ is a strong Markov process satisfying the condition
    \begin{gather}\label{mp-e26}
        \forall\epsilon>0\::\quad \lim_{t\to 0} \frac{1}{t} \sup_{x\in\rd} \Pp^x\left(|X_t - x|>\epsilon\right) = 0,
    \end{gather}
    then $(X_t)_{t\geq 0}$ has a modification with continuous sample paths.
\end{theorem}

Obviously, the condition \eqref{mp-e26} implies uniform stochastic continuity. Hence, we can use Theorem~\ref{mp-21} to get a rather simple proof of the Dynkin--Kinney criterium, see It\^o~\cite[Ch.~2.9, Thm.~3]{ito-aarhus}. For a more standard proof, we refer to Wentzell~\cite[Ch.~9.1]{wentzell79} or Dynkin~\cite[Ch.~6,\S5]{dynkin61}.

If we assume a bit more about the process $(X_t)_{t\geq 0}$, e.g.\ that it is Feller and that the test functions $C_c^\infty(\rd)$ are in the domain of the infinitesimal generator, then \eqref{mp-e26} is essentially equivalent to the continuity of the sample paths or to the fact that the infinitesimal generator is a second-order differential operator, see the discussion in Schilling~\cite[Thm.~7.38, Thm.~7.39, Thm.~23.3]{schilling-bm}.

A further application of such maximal inequalities are estimates on the $p$-variation of stochastic processes, see Manstavi\v{c}ius~\cite{manstavicius04}; typically, the results of Chapter~\ref{fel} (e.g.\ Theorem~\ref{fel-5}) in combination with Theorem~\ref{mp-21} establish the conditions needed in Manstavi\v{c}ius~\cite{manstavicius04}, see B\"{o}ttcher \emph{et al.}~\cite[Ch.~5.4]{matters3}.

\begin{remark}[on sublinear processes]\label{mp-27} \index{sublinear process}\index{G-L\'evy process@$G$-L\'evy process}
    It is also possible (and interesting) to study maximal inequalities for sublinear expectations. A \emph{sublinear} Markov semi\-group $(T_t)_{t \geq 0}$ satisfies the usual semigroup property $T_{t+s} = T_t T_s$, but each operator $T_t$ is only subadditive and positively homogeneous, i.e.\
\begin{equation*}
	T_t(f+g) \leq T_t(f) + T_t(g)
    \quad\text{and}\quad
    T_t(\lambda f) = \lambda T_t(f),
\end{equation*}
for all $\lambda \geq 0$ and all bounded measurable functions $f,g$. Sublinear Markov semigroups appear in the study of stochastic processes on sublinear expectation spaces, which can be interpreted as stochastic processes under uncertainty, cf.\ Hollender~\cite{hollender}. Using techniques similar to those in Chapter~\ref{fel}, it is possible to derive a maximal inequality for this class of processes, cf.\ K\"{u}hn~\cite[Proposition 5.1]{sublinear1}. Applications include an analogue of the Courr\`ege-von Waldenfels theorem, cf.\ K\"{u}hn~\cite{sublinear-generator}, and the study of sublinear Hamilton--Jacobi--Bellman equations, see Denk \emph{et al.}~\cite{denk} and Hollender~\cite{hollender} for further details and references.
\end{remark}

\section{Gaussian Processes}\label{gau}

The methods used in the study of Gaussian processes are quite different from those introduced in the previous sections; quite often, they come from discrete mathematics or functional analysis or from a combination of both.  Moreover, the full strength of these methods unfolds, if one considers Gaussian processes which are indexed by elements of a (compact) metric space. As before, we restrict ourselves to processes indexed by $[0,\infty)$ and taking values in $\real$ -- this is enough to illustrate some of the methods used in the theory. Our standard references for Gaussian processes are the books by Lifshits~\cite{lifshits95} and~\cite{lifshits12}. Gaussian random variables and Gaussian processes with values in Banach spaces are discussed in Ledoux \& Talagrand \cite{led-tal91}.

\begin{definition}\label{gau-03} \index{Gaussian process}
    A random process $(X_t)_{t\geq 0}$ taking values in $\real$ is said to be Gaussian, if every vector $(X_{t_1},\dots,X_{t_n})$, $0\leq t_1<\dots <t_n$ is an $n$-dimensional Gaussian random variable.
\end{definition}

By definition, a Gaussian random process is uniquely determined through the mean $m(t) = \Ee \left[X_t\right]$ and the covariance $K(t,s)=\Ee\left[X_t X_s\right] - m(t)m(s)$ functions.

\begin{example}\label{gau-05}\index{Brownian motion}%
Typical examples of Gaussian processes (indexed by $[0,\infty)$) are:
\begin{enumerate}
\item
    \emph{Brownian motion:} $K(t,s) = t\wedge s$ and $m(t)\equiv 0$.
\item
    \emph{Fractional Brownian motion:} $K(t,s) = \frac 12\left(t^\alpha + s^\alpha -|s-s|^\alpha\right)$ with the Hurst index $H = \alpha/2 \in (0,1]$ and $m(t)\equiv 0$.
\item
    \emph{Bi-fractional Brownian motion:} $K(t,s) = \frac 12\left( \left(t^\alpha + s^\alpha\right)^k -|s-s|^{\alpha k}\right)$ with $\alpha, k \in (0,2]$, $\alpha k \leq 2$ and $m(t)\equiv 0$.
\item
    \emph{Gaussian Markov processes:} Let $f,g\colon [0,\infty)\to[0,\infty)$ be functions such that $g$ is strictly increasing. The process $X_t \coloneqq  f(t) W_{g(t)}$ where $(W_t)_{t\geq 0}$ is a standard Brownian motion, is a mean-zero Gaussian Markov process.

    Conversely, every mean-zero Gaussian Markov process is of this form, see Marcus \& Rosen~\cite[Chap.~5.1.1.]{mar-ros06}.

    Examples in this class are Brownian motion, the Brownian bridge ($K(s,t) = s\wedge t - st$, $X_t = (1+t)W_{t(1+t)^{-1}}$) or the Ornstein--Uhlenbeck process ($K(s,t) = \exp(-|s-t|)$, $X_t = e^{-t} W_{e^{2t}}$ -- this is the only stationary Gaussian Markov process).
\end{enumerate}
\end{example}

\subsection{Concentration of Measure Methods}

A key tool in the analysis of Gaussian processes are \emph{tail estimates} (sometimes called \emph{Borell estimates}). We will sketch a method how to obtain such estimates from concentration inequalities. As a preparation we need the following logarithmic Sobolev inequality in Gaussian spaces which was discovered by Gross~\cite{gross75}. We do not provide a proof, but refer to Ledoux~\cite[Thm.~5.1]{ledoux01} (for a semigroup proof) or Shigekawa~\cite[Thm.~2.12]{shige04} (for a probabilistic proof). Recall that the \emph{entropy of a random variable} $Z\geq 0$ is defined as $\Ent(Z) = \Ee\left(Z\log Z\right) - \left(\Ee Z\right) \log \left(\Ee Z\right)$.

\begin{theorem}[logarithmic Sobolev inequality]\label{gau-07} \index{inequality!log Sobolev}
    Let $G=(G_1,\dots,G_n)\in\rn$ be a vector whose entries are iid standard normal random variables and $f\colon \rn\to\real$ a $C^1$-function. Then the entropy of $f^2(G)$ satisfies
    \begin{gather}\label{gau-e08}
        \Ee\left[f^2(G)\log f^2(G)\right] - \Ee\left[f^2(G)\right] \log\Ee\left[f^2(G)\right]
        \leq 2\Ee\left[|\nabla f(G)|^2\right].
    \end{gather}
\end{theorem}

\begin{theorem}\label{gau-11}
    Let $G=(G_1,\dots,G_n)\in\rn$ be a vector whose entries are iid standard normal random variables and $f\colon \rn\to\real$ a Lipschitz function with global Lipschitz constant $L$, i.e.\ $|f(x)-f(y)|\leq L|x - y|$ holds for all $x,y\in\rn$. Then
    \begin{gather}\label{gau-e12}
        \Ee\left[e^{\lambda(f(G) - \Ee f(G))}\right] \leq e^{\frac 12 \lambda^2 L^2},\quad \lambda > 0.
    \end{gather}
\end{theorem}
\begin{proof}
    It is enough to show \eqref{gau-e12} for $C^1$-functions whose derivative is bounded by $L$. Indeed, if $f$ is Lipschitz with Lipschitz constant $L$, we can approximate $f$ pointwise by a sequence of $C^1$-functions $f_n$ such that $|\nabla f_n(x)|\leq L$, and \eqref{gau-e12} remains stable under the limit if we use Fatou's lemma. So, without loss of generality, $f\in C^1$ and $|\nabla f(x)|\leq L$. Moreover, we can always assume that $\Ee f(G)=0$, otherwise we use $f \rightsquigarrow f - \Ee f(G)$.

    We apply Theorem~\ref{gau-07} to the function $\lambda\mapsto e^{\lambda f(G)/2}$ to get
    \begin{align*}
        \Ee\left[\lambda f(G) e^{\lambda f(G)}\right] - \Ee\left[e^{\lambda f(G)}\right]\cdot \log \Ee\left[e^{\lambda f(G)}\right]
        &\leq 2\Ee\left[\left|\nabla e^{\frac 12\lambda f(G)}\right|^2\right]\\
        &= \frac{\lambda^2}{2} \Ee\left[e^{\lambda f(G)}\left|\nabla f(G)\right|^2\right]\\
        &\leq \frac{\lambda^2}{2} L^2 \Ee\left[e^{\lambda f(G)}\right].
    \end{align*}
    If we set $\phi(\lambda) \coloneqq  \Ee \left[e^{\lambda f(G)}\right]$, this inequality becomes a differential inequality
    \begin{gather*}
        \lambda \phi'(\lambda) - \phi(\lambda)\log\phi(\lambda) \leq \frac{\lambda^2}{2} L^2\phi(\lambda),
    \end{gather*}
    which we can rewrite as
    \begin{gather*}
        \frac{d}{d\lambda}\left(\frac{\log\phi(\lambda)}{\lambda}\right)
        = \frac{\phi'(\lambda)}{\lambda\phi(\lambda)} - \frac{\log\phi(\lambda)}{\lambda^2} \leq \frac 12 L^2.
    \end{gather*}
    From this we see that $\phi(\lambda) \leq e^{\lambda^2 L^2/2}$, and the claim follows.
\end{proof}

\begin{corollary}[Gaussian concentration]\label{gau-13} \index{inequality!Gaussian concentration}
    Let $G=(G_1,\dots,G_n)\in\rn$ be a vector whose entries are iid standard normal random variables and $f\colon \rn\to\real$ a Lipschitz function with global Lipschitz constant $L$. Then
    \begin{gather}\label{gau-e14}
        \Pp\left( \pm \big(f(G) - \Ee f(G)\big)> \lambda\right) \leq e^{-{\lambda^2}/{2L^2}},\quad \lambda > 0.
    \end{gather}
\end{corollary}
\begin{proof}
    For $x,\lambda>0$ we find using the Markov inequality and Theorem~\ref{gau-11}
    \begin{align*}
        \Pp\left( \pm\left(f(G) - \Ee f(G)\right) > \lambda\right)
        &= \Pp\left( \pm x\left(f(G) - \Ee f(G)\right) > x\lambda\right)\\
        &\leq e^{-x\lambda} \Ee\left(e^{\pm x\left(f(G) - \Ee f(G)\right)}\right)
        \leq e^{-x\lambda + x^2L^2/2},
    \end{align*}
    and the claim follows, since the right hand side is minimal for $x=\lambda/L^2$.
\end{proof}

An important feature in the estimates \eqref{gau-e12} and \eqref{gau-e14} is that the right hand side does not depend on the dimension $n$ of the Gaussian vector $(G_1,\dots,G_n)$. This allows us to derive
\begin{corollary}[Borell's inequality]\label{gau-15} \index{inequality!Borell}
    Let $(X_t)_{t\geq 0}$ be \textup{(}a separable modification\footnote{A necessary and sufficient condition for the existence of a separable modification is that $(s,t)\mapsto \|X_s-X_t\|_{L^2}$ is separable, e.g.\ if $K(s,t)$ is continuous; see Lifshits~\cite[p.~26]{lifshits95} and the discussion at the beginning of Chapter~\ref{gen}.} of\textup{)} a centred Gaussian process and $M_t \coloneqq \sup_{s\leq t} X_s$ its running maximum. If $\sigma_t^2 \coloneqq  \sup_{s\leq t}\Ee \left[|X_s|^2\right]<\infty$ for all $t\geq 0$, then $\Vv\left[M_t \right]<\infty$ and
    \begin{gather}\label{gau-e16}
        \Pp\left(\pm\left(M_t - \Ee M_t  \right) > \lambda\right) \leq e^{-{\lambda^2}/{2\sigma_t^2}},\quad\lambda>0.
    \end{gather}
\end{corollary}
\begin{proof}
    Fix $t>0$ and let $0\leq t_1<\dots < t_n\leq t$. Since $\Xi \coloneqq  (X_{t_1},\dots,X_{t_n})^\top$ is a Gaussian random vector, we can write it as $\Xi = \sqrt C\Gamma$ where $\sqrt C = (\gamma_{ij})_{i,j}$ is the (unique, positive semidefinite) square root of the covariance matrix $C = (K(t_i,t_j))_{i,j}$ and $\Gamma \coloneqq  (G_1,\dots,G_n)^\top$ is a Gaussian vector with iid standard normal entries. In particular,
    \begin{gather*}
        \max_{1\leq i\leq n} X_{t_i} \sim \max_{1\leq i\leq n} (\sqrt C \Gamma)_i  \eqqcolon  f(G_1,\dots,G_n).
    \end{gather*}
    For $x,y\in\rn$ we have for every $i=1,2,\dots,n$
    \begin{gather*}
        \left|\big(\sqrt C x\big)_i - \big(\sqrt C y\big)_i\right|
        = \left| \sum_{j=1}^n \gamma_{ij} (x_j-y_j) \right|
        \leq \left(\sum_{j=1}^n \gamma_{ij}^2\right)^{1/2} |x-y|.
    \end{gather*}
    Since $\sum_{j=1}^n \gamma_{ij}^2 = \Vv(X_{t_i})$, we see that
    \begin{gather*}
        \left|f(x) - f(y)\right| \leq \sigma_t |x-y|.
    \end{gather*}

    Pick the $t_i$ in such a way that $\{t_1,\dots,t_n\}\uparrow \rat\cap[0,t]$. Take $y=0$ and use the separability of the process to get $\Vv(M_t) = \sup_n \Vv f(G_1,\dots,G_n) \leq \sigma_t$. Finally, we can apply Corollary~\ref{gau-13} and see that
    \begin{equation*}\begin{aligned}[b]
        \Pp\left(\pm\left(M_t - \Ee M_t \right) > \lambda\right)
        &= \lim_{n\to\infty} \Pp\left(\pm\Big[\max_{1\leq i\leq n} X_{t_i} - \Ee \big[\max_{1\leq i\leq n} X_{t_i}\big]\Big] > \lambda\right)\\
        &\leq e^{-\lambda^2/2\sigma_t^2}.
    \end{aligned}\qedhere\end{equation*}
\end{proof}

\begin{corollary}[Fernique's theorem]\label{gau-17}\index{theorem!Fernique}
    Let $(X_t)_{t\geq 0}$ be a \textup{(}separable modification of a\textup{)} centred Gaussian process and $M_t \coloneqq \sup_{s\leq t} X_s$ its running maximum. If $\sigma_t^2 \coloneqq  \sup_{s\leq t}\Ee \left[X_s^2\right]<\infty$ for all $t\geq 0$, then
    \begin{gather}\label{gau-e18}
        \Ee\left[ e^{p M_t^2 }\right] < \infty,\quad 0\leq p< \frac 1{2\sigma_t}.
    \end{gather}
\end{corollary}
\begin{proof}
    From Borell's inequality we see that
    \begin{gather*}
        \Pp\left( |M_t - \Ee M_t| >\lambda\right)
        \leq \Pp\left( M_t - \Ee M_t >\lambda\right) + \Pp\left( \Ee M_t - M_t > \lambda\right)
        \leq 2 e^{-\lambda^2/2\sigma_t^2},
    \end{gather*}
    and with the layer-cake formula we get
    \begin{equation*}\begin{aligned}[b]
        \Ee\left[ e^{p (M_t - \Ee M_t )^2}\right]
        &= \int_0^\infty 2\lambda p e^{p\lambda^2} \Pp\left(|M_t - \Ee M_t| >\lambda\right) d\lambda\\
        &\leq \int_0^\infty 4 \lambda p e^{p\lambda^2} e^{-\lambda^2/2\sigma_t^2}\, d\lambda
        = \frac{4 p\sigma_t^2}{2p\sigma_t^2 - 1}.
    \end{aligned}\qedhere\end{equation*}
\end{proof}

The setting and proofs of Corollaries~\ref{gau-15} and \ref{gau-17} show that we are in a truly exceptional situation: From $\Pp\left( |M_t - \Ee M_t | > \lambda\right) \leq  2e^{-\lambda^2/2\sigma_t^2}$ we conclude
\begin{gather*}
    \Pp\left(M_t < \infty\right) = 1 \iff \Ee M_t < \infty.
\end{gather*}
Together with \eqref{gau-e18} we get, that for a centred Gaussian process with $\sigma_t^2 \coloneqq  \sup_{s\leq t}\Vv X_s < \infty$, we even have
\begin{gather*}
 M_t < \infty\text{\ a.s.}
    \iff \Ee M_t < \infty
    \iff \Ee e^{p M_t^2 } < \infty,\; p < (2\sigma_t)^{-1}.
\end{gather*}

\subsection{Metric Entropy Methods}

Let us briefly discuss a further technique in the study of maximal inequalities for Gaussian processes. It is connected with the way in which we can cover the index set optimally. These methods came up at the end of the 1960s, mainly due to Dudley~\cite{dudley67}, and led to a breakthrough in the study of Gaussian processes.

\begin{definition}\label{gau-31}\index{covering number}\index{metric entropy}
Let $(\Tt,d)$ be a compact metric space. One defines for every $\epsilon>0$ the
\begin{align*}
    \emph{covering number}& & N(\epsilon) &\coloneqq  \min\left\{n \mid \Tt \subseteq {\textstyle \bigcup_{k=1}^n} A_{k},\; \diam(A_i)\leq\epsilon \right\};\\
    \emph{metric entropy}& & H(\epsilon) &\coloneqq  \log N(\epsilon);\\
    \emph{packing number}& & M(\epsilon) &\coloneqq  \max\Big\{ n \mid t_1,\dots,t_n\in\Tt, \forall i\neq j\::\: d(t_i,t_j)>\epsilon\Big\};\\
    \emph{metric capacity}& & C(\epsilon) &\coloneqq  \log M(\epsilon).
\end{align*}
\end{definition}
Note that $N(2\epsilon)\leq M(\epsilon)\leq N(\epsilon)$, see Lifshits~\cite[Prop.~10.1]{lifshits12}. We write $N_\Tt, H_\Tt$ etc.\ if we need to specify the index set $\Tt$.

Entropy numbers play also an important role in the theory of function spaces and in the approximation of operators. Among the standard references in this direction are Carl \& Stephani \cite{car-ste90} or Edmunds \& Triebel~\cite{edm-tri96}. The notions introduced in Definition~\ref{gau-31} have all their roots in the concept of metric entropy of a set which Kolmogorov introduced in the 1930s and 1950s, see Kolmogorov~\cite{kolmogorov36,kolmogorov55,kolmogorov56} and Kolmogorov \& Tikhomirov~\cite{kolmogorov59}.

We mention, without proof, a key result based on the notions introduced in Definition~\ref{gau-31}.
\begin{theorem}[Dudley--Sudakov bounds]\label{gau-33}\index{inequality!Dudley--Sudakov}
    Assume that $(X_t)_{t\geq 0}$ is \textup{(}the separable modification of\textup{)} a centred Gaussian process and $\sigma_t^2 = \sup_{s\leq t}\Ee\left(X_{s}^2\right)<\infty$ for every $t\geq 0$. Then the following estimates hold
    \begin{gather}\label{gau-e34}
        \frac{\epsilon}{2\sqrt 2}  \sqrt{C_{[0,t]}(\epsilon)}
        \leq
        \Ee \left[X_t^*\right]
        \leq
        2\sqrt 2 \int_0^{\sigma_t} \sqrt{H_{[0,t]}(2\epsilon)}\,d\epsilon,\quad \epsilon > 0.
    \end{gather}
\end{theorem}
For a proof and further information we refer to the discussion in Lifshits~\cite[Thm.~1, p.~179, Thm.~5, p.~193]{lifshits95} and~\cite[Thm.~10.1, Thm.~10.5]{lifshits12}, see also Lemma~\ref{gen-15}. Note that the finiteness of the \emph{Dudley integral} appearing in the upper bound of \eqref{gau-e34}, is \emph{necessary and sufficient for the continuity and boundedness of the process $(X_t)_{t\in [0,T]}$ for any fixed $T>0$}, see Lifshits~\cite[p.~224]{lifshits95}. The idea behind the (upper part of the) Dudley--Sudakov bounds can be vastly generalized. We would like to mention here the pioneering work by Talagrand~\cite{talagrand96}, who uses systematically majorizing measures and the generic chaining method to derive bounds for $X^*$, pushing the theory well beyond Gaussian processes. For example, one can get upper and lower bounds for infinitely divisible processes. The (second edition of the) research monograph Talagrand~\cite{talagrand14} contains the latest developments in this direction.

\medskip
A further important topic in Gaussian processes are the so-called \emph{small ball probabilities} or \emph{small deviations}. The typical problem is to study the behaviour of
\begin{gather}
    \Pp\left(\sup_{s\leq t}|X_s| \leq \epsilon\right)  \text{\ \ as $\epsilon\to 0$},
\end{gather}
or of any other interesting functional $\Phi(X_s, s\leq t)$ replacing the supremum $\sup_{s\leq t}|X_s|$. With the Wiener process in mind, one would expect asymptotics like $\sim \kappa\epsilon^a \exp(-c\epsilon^{-b})$, where $b$ is the \emph{small deviation rate} and $c$ is the \emph{small deviation constant}. For a Wiener process we get $\kappa = 4/\pi$, $a=0$, $c = \pi^2/8$ and $b=2$; details can be found in Lifshits~\cite[Chap.~11]{lifshits12}, see also Schilling~\cite[Lem.~12.8]{schilling-bm}. Such estimates are interesting if one wants to determine the exact short-time behaviour of the sample paths of $t\mapsto X_t$.

The following example is a classic application of small-ball probabilities; at the same time, it nicely illustrates the scope and strength of small-ball methods. The result is due to L\'evy~\cite[p.~96, pp.~107--109]{levy53}, and it is the archetype of a so-called Stroock-and-Varadhan \emph{support theorem} for diffusions (see e.g.\ Stroock \& Varadhan~\cite{str-var72}, Ikeda \& Watanabe~\cite[Ch.~VI.8]{ike-wat89}) or for jump processes (e.g.\ Simon~\cite{simon00}, Kulik~\cite{kulik21}).

\begin{example}[L\'evy] \label{gau-35}\index{Brownian motion}%
	\textit{Let $(B_t)_{t \geq 0}$ be a one-dimensional Brownian motion. Then
	\begin{equation*}
		\Pp \left( \sup_{0 \leq t \leq 1} |B_t-f(t)|<\epsilon \right)>0, \quad \epsilon>0,
	\end{equation*}
	for any continuous function $f\colon [0,1] \to \real$ with $f(0)=0$.}

\begin{proof}
    Fix $\epsilon>0$. By the uniform continuity of $f$ on $[0,1]$, there is some $n \in \nat$ such that $|f(s)-f(t)| < \epsilon$ for all $|s-t| \leq 1/n$. Set $t_i \coloneqq i/n$, then
	\begin{equation*}
		\left\{ \sup_{0 \leq t \leq 1} |B_t-f(t)|<\epsilon \right\}
		\supseteq \bigcap_{i=0}^{n-1} \left\{ \sup_{t \in [t_i,t_{i+1}]} |B_t-f(t_i)|<\frac{\epsilon}{2(n-i)} \right\} \eqqcolon  \bigcap_{i=0}^{n-1} A_i,
	\end{equation*}
	and so, by the Markov property,
	\begin{align*}
		&\Pp \left( \sup_{0 \leq t \leq 1} |B_t-f(t)|<\epsilon \right)
		\geq \Pp \left( \bigcap_{i=0}^{n-1} A_i \right) \\
        &\quad\geq \Ee \left( \prod_{i=0}^{n-2} \I_{A_i} \cdot \Pp^{B_{t_{n-1}}} \left[ \sup_{s \leq 1/n} |B_s-f(t_{n-1})| < \frac{\epsilon}{2} \right] \right).
	\end{align*}
	Since $|B_{t_{n-1}}-f(t_{n-1})| < \epsilon/4$ on $A_{n-2}$, it follows that
	\begin{gather*}
		\Pp \left( \sup_{0 \leq t \leq 1} |B_t-f(t)|<\epsilon \right)
		\geq \Pp \left( \bigcap_{i=0}^{n-2} A_i \right) \Pp\left( \sup_{s \leq 1/n} |B_s| < \frac{\epsilon}{4} \right).
	\end{gather*}
	The reflection principle shows that for any $\epsilon>0$ and $n \in \nat$ the probability
	\begin{gather*}
		\Pp\left( \sup_{s \leq 1/n} |B_s| < \frac{\epsilon}{4} \right)
	\end{gather*}
	is strictly positive, and the claim follows by iteration.
\end{proof}
\end{example}

\newpage

\section*{Index of Notation}\label{index-notation}

\begin{small}
\setlength{\columnseprule}{.5pt}

This is intended to aid cross-referencing, so notation
that is specific to a single section is generally not listed. Some
symbols are used locally, without ambiguity, in senses other than
those given below.

Unless otherwise stated, binary operations between
functions such as $f\pm g$, $f\cdot g$, $f\wedge g$, $f\vee g$,
comparisons $f\leq g$, $f < g$ or limiting relations $f_n
\xrightarrow[n\to\infty]{} f$, $\lim_n f_n$, $\liminf_n f_n$,
$\limsup f_n$, $\sup_i f_i$ or $\inf_i f_i$  are always understood
pointwise.

\begin{raggedright}
\noindent
\begin{multicols}{2}
\raggedright
\parindent=0pt
\begin{nindex}\itemsep.5pt

\item[\bfseries Generalities]
\item[]

\item[positive]                 always  $\geq 0$
\item[negative]                 always  $\leq 0$
\item[increasing]               always  non-strict sense
\item[decreasing]               always  non-strict sense

\item[$\nat$]                   natural numbers: $1,2,3,\ldots$
\item[$\nat_0$]                 positive integers: $0,1,2,\ldots$
\item[$\integer,\rat,\real,\comp$] integer, rational, real, complex numbers
\item[$\inf\emptyset, \sup\emptyset$] $\inf\emptyset = +\infty$, $\sup\emptyset = -\infty$
\item[$a\vee b$]                maximum of $a$ and $b$
\item[$a\wedge b$]              minimum of $a$ and $b$
\item[$\entier{x}$]             $\max\{n\in\integer\mid n\leq x\}$

\item[]
\item[\bfseries Functions and sets]
\item[]

\item[$\I_A$]                       $\I_A(x)=\begin{cases} 1, & x\in A \\ 0, & x\not\in A\end{cases}$
\item[$f^+$, $X^+$]                 positive part
\item[$f^-$, $X^-$]                 negative part
\item[$f\asymp g$]                  %
                                    $\forall x\::\: cf(x)\leq g(x)\leq Cf(x)$
\item[$\log^+ x$]                   $\max\{\log(x),0\}$

\item[$\widehat f$, $\widecheck g$] (inverse) Fourier transform, p.~\pageref{hidden-ft}
\item[$f^*(x)$]                     Hardy-Littlewood maximal function, p.~\pageref{H-L-max}
\item[$f^\diamond(x)$]              square maximal function, p.~\pageref{square-max}

\item[$C(\epsilon)$]                metric capacity, p.~\pageref{gau-31}
\item[$H(\epsilon)$]                metric entropy, p.~\pageref{gau-31}
\item[$N(\epsilon)$]                covering number, p.~\pageref{gen-e15}, \pageref{gau-31}
\item[$M(\epsilon)$]                packing number, p.~\pageref{gau-31}

\item[$\psi$, $\psi_X$]             characteristic exponent, p.~\pageref{pre-e71}, \pageref{lp-e10}
\item[$\psi^*$, $\psi_X^*$]         maximum function, p.~\pageref{pre-e73}
\item[$|\psi|^*$, $|\psi_X|^*$]     maximum function, p.~\pageref{pre-e73}
\item[$(b,Q,\nu)$]                  L\'evy triplet, p.~\pageref{pre-e71}, \pageref{lp-e10}
\item[$(b(x),Q(x),\nu(x,dy))$]      infinitesimal characteristics, p.~\pageref{fel-eq2}

\item[$\|\cdot\|_{\mathrm{BMO}}$]   bounded mean oscillation, p.~\pageref{gen-e38} \eqref{gen-e36}, \eqref{gen-e38}, p.~\pageref{mar-17}

\item[$\ball{r}{x}$]                open ball with radius $r$ and centre $x$

\item[]
\item[\bfseries Random variables and processes]
\item[]

\item[$\Pp$, $\Ee$, $\Vv$]  probability, expectation, variance
\item[$X\sim Y$]            $\mathrm{law}(X) = \mathrm{law}(Y)$
\item[$X\sim \mu$]          $\mathrm{law}(X) = \mu$
\item[$M_t$]                $\sup_{s\leq t} X_s$
\item[$X_t^*$]              $\sup_{s\leq t}|X_s|$
\item[$X^*$]                $\sup_{s\geq 0}|X_s|$
\item[$\Delta X_t$]         $X_t - X_{t-}$, $(t>0)$

\item[$c\bullet X_n$]       martingale transform, p.~\pageref{mg-trafo}
\item[$H\bullet X_t$]       stochastic integral, p.~\pageref{lp-44}

\item[$\sbracket X_n$]      compensator, p.~\pageref{compen}
\item[$\bracket X_t$, $\sbracket X_t$]
                            (previsible) quadratic variation, p.~\pageref{mar-43}, \pageref{mar-e70}

\item[$\alpha_X(\ell)$]     Milman concentration function, p.~\pageref{pre-e61}

\item[$D(\mu,\ell), G(\mu,\ell), K(\mu,\ell), M(b,\nu,\ell)$]
                            truncated moments, p.~\pageref{pre-e75}
\item[$Q_X(\ell)$]          L\'evy concentration function, p.~\pageref{pre-e41}

\item[]
\item[\bfseries Abbreviations]
\item[]

\item[a.a., a.e.]     almost all/every(where)
\item[BDG]            Burkholder--Davis--Gundy
\item[BMO]            bounded mean oscillation
\item[c\`adl\`ag]     right-continuous, finite left-hand limits
\item[iid]            independent and identically distributed

\item[\eqref{Lstat},\eqref{Lindep},\eqref{Lcont}]
                      properties of L\'evy processes, pp.~\pageref{Lindep}--\pageref{Lcont-bis}

\item[]

\end{nindex}
\end{multicols}
\end{raggedright}

\newpage

\renewcommand{\indexname}{\normalsize Index}
\printindex

\end{small}

\begin{acks}[Acknowledgments]
    We thank the editor-in-chief of \emph{Probability Surveys} for the invitation to contribute to this journal and for the careful handling of the submission. The efforts and advice of two anonymous reviewers are much appreciated. We are deeply indebted to our colleagues and students, Dr.~David Berger, Dr.~Wojciech Cygan, Mr.~Mustafa Hamadi and Ms.~Cailing Li,  for proofreading and helpful comments.
    A substantial part of this paper was written while the first-named author was research associate at the Mathematics Department of TU Dresden. She would like to thank the department for the good working conditions.
    Financial support through the DFG-NCN Beethoven Classic 3 project SCHI419/11-1 \& NCN 2018/31/G/ST1/02252 is gratefully  acknowledged.
\end{acks}

\end{document}